\documentclass[a4paper]{article}

\usepackage[T1]{fontenc}
\usepackage[english]{babel}
\usepackage[utf8]{inputenc}
\usepackage{lmodern}
\usepackage{amsthm}
\usepackage{amsmath}
\usepackage{amscd}
\usepackage{amssymb}
\usepackage[dvips,final]{graphicx}
\usepackage[dvips]{geometry}
\usepackage{latexsym}
\usepackage{mathrsfs}
\usepackage{color}
\usepackage{version}
\usepackage{blkarray}
\usepackage{bigints}
\usepackage[inline]{enumitem} 
\usepackage{hyperref}

\newcommand{\RR}{\mathbb{R}}
\newcommand{\Sk}{\mathbb{S}^k}

\newcommand{\Dkk}{\mathbb{D}^k}
\newcommand{\Sn}{\mathbb{S}^n}

\newcommand{\Dn}{\mathbb{D}^{n+1}}
\newcommand{\Dnn}{\mathbb{D}^n}

\newcommand{\Mp}{\mathcal{M}_p(\mathbb{Z})}

\DeclareMathOperator{\Ima}{Im}
\DeclareMathOperator{\coker}{coker}%
\DeclareMathOperator{\coIma}{coim}%
\DeclareMathOperator{\Hom}{Hom}%
\DeclareMathOperator{\ind}{index}%
\DeclareMathOperator{\Char}{Char}

\newtheorem{theorem}{Theorem}[section]
\newtheorem{lemma}{Lemma}[section]
\newtheorem{remark}{Remark}[section]
\newtheorem{definition}{Definition}[section]
\newtheorem{proposition}{Proposition}[section] 

\newtheorem*{nntheo}{Theorem}
\newtheorem*{acknow}{Acknowledgment}
\newtheorem*{address}{Address}
\newtheorem*{email}{e-mail}
\title{Extending functions from a neighborhood of the sphere to the ball}

\author{Valentin Seigneur}

\begin{document}

\maketitle

\begin{abstract}
In this article, we are interested in the problem of extending the germ of a smooth function $\tilde f$ defined along the standard sphere of dimension $n$ to a function defined on the ball which has no critical points.

The article gives a necessary condition using the Morse chain complex associated to the function $f$, restriction of $\tilde{f}$ to the sphere $\Sn$, which is assumed to be a Morse function.
\end{abstract}

\section{Introduction}

The aim of the article is to study the following question.

Consider the germ of a smooth function with no critical point, defined along the standard sphere of dimension $n$ denoted $\Sn$. 
Can we extend it to a function on the ball with no critical points?

Let $\partial M$ be a closed manifold, boundary of a manifold $M$.
We will use the terminology \textit{Morse germ} to denote the germ along $\partial M$ of a real function $\tilde f$ with no critical point and whose restriction $f$ to the boundary is Morse.  
To our knowledge, this question has been tackled for the first time in the article of Blank and Laudenbach \cite{BL}, who answer it for $n=1$, then in the article of Curley \cite{Cur}, answering it for $n=2$.
In these articles, answers are combinatorical.
 More recently, Barannikov \cite{Bar} gives a necessary condition again of combinatorical nature about the Morse complex of the function $f$, with coefficients in a field. 
 
 In the present article, we give a necessary condition of algebraic nature.
This condition uses the Morse complex with coefficients in $\mathbb{Z}$ of the restriction $f$ of the germ $\tilde{f}$ to $\Sn$ and the normal data given by any representative $\tilde{f}$ of the germ. 
The question is also raised by Arnol'd in \cite{A}, see Problem 1981-8.

As $\tilde{f}$ is non-critical, the set of critical points of $f$, denoted $\mathcal{C}(f)$, is separated into two sets, depending on the derivative of $f$ along the vector normal to the sphere pointing towards the ball:
\begin{itemize}
\item points $x$ for which $\partial_t(\tilde{f}(x,0)) > 0$, points labeled "$-$" forming the set $\mathcal{C}^-(\tilde{f})$;
\item points $x$ for which $\partial_t(\tilde{f}(x,0)) <0$, points labeled "$+$" forming the set $\mathcal{C}^+(\tilde{f})$.
\end{itemize}

Given a Morse-Smale pseudo-gradient adapted to $f$ (see Section \ref{pseudograd} for definitions) we denote by $\partial_k$ the boundary operator of the Morse chain complex restricted to $\mathbb{Z}\mathcal{C}_k(f)$, and write it:
\[ \partial_k := \begin{pmatrix}
\partial_{++,k} & \partial_{+-,k} \\
\partial_{-+,k} & \partial_{--,k} \\
\end{pmatrix},\]
where, for $( \ell_1, \ell_2 )\in \{ +,- \}$, the matrix $\partial_{ \ell_1 \ell_2 ,k}$ sends $\mathbb{Z}\mathcal{C}_k^{\ell_2}(\tilde{f})$ into $\mathbb{Z}\mathcal{C}^{\ell_1 }_k(\tilde{f})$.

In general, the $\mathbb{Z}$-module $\mathbb{Z}\mathcal{C}^+(\tilde{f})$ (resp. $\mathbb{Z}\mathcal{C}^-(\tilde{f})$), freely generated by $\mathcal{C}^+(\tilde{f})$ (resp. by $\mathcal{C}^-(\tilde{f})$), is not a chain complex.
However, it becomes a chain complex with some prescribed homology groups if $\tilde{f}$ extends non-critically.
 It is the purpose of the main theorem of the article, the notation being explained more precisely in Section \ref{nec}.
  We will introduce some subgroup $G(\tilde{f})$ of the group of graded isomorphisms defined on $\mathbb{Z}\mathcal{C}(\tilde{f})$.
   It depends on the order of the critical values of the critical points of $f$ and the splitting of $\mathcal{C}(f)$ into the sets $\mathcal{C}^+(\tilde{f})$ and $\mathcal{C}^-(\tilde{f})$.
\begin{nntheo}[Theorem \ref{th1}, Section \ref{nec}]
If the germ $\tilde{f}$ has a non-critical extension then there is a matrix $M$ in $G(\tilde{f})$ such that $ (M \partial M^{-1} )_{-+} =0$ and such that $(M\partial M^{-1})_{++}$ defines a boundary operator on $(\mathbb{Z}\mathcal{C}^+(\tilde{f}))_{0\leq k \leq n}$ whose homology vanishes in all degree except in degree $n$ for which it is $\mathbb{Z}$.
\end{nntheo}

We also have the following theorem:

\begin{nntheo}[Theorem \ref{best}, Section \ref{theexample}]
Let $n \geq 6$.
Let $\tilde{f}$ be a Morse germ along $\Sn$ fulfilling the conclusion of the previous theorem.
We assume that $f$ has only one local maximum, one local minimum and no points of index $1$ or $n-1$.
 There is a Morse germ $\tilde{f_1}$ such that:
\begin{itemize}
\item $G(\tilde{f_1}) = G(\tilde{f})$, in particular $f_1$ has the same number of critical points as $f$, with same indices and labels;
\item $\tilde{f_1}$ extends non-critically.
\end{itemize}

\end{nntheo}

The proof of Theorem \ref{best} exhibits the function $f_1$ which is the endpoint of a generic path of functions starting at $f$ that presents no birth or death bifurcations.
We then have a natural bijection between the critical points of $f$ and those of $f_1$. 
The difference between $f_1$ and $f$ is that $f_1$ is \emph{ordered}, in the sense that $f_1(a)>f_1(b)$ whenever $a \in \mathcal{C}_k(f_1)$ and $b\in \mathcal{C}_{k-1}(f_1)$ for any index $k$.
 However, the order of the critical values of points of same index is preserved by the natural bijection between $\mathcal{C}(f)$ and $\mathcal{C}(f_1)$.
 
The necessary condition of non-critical extension given by Theorem \ref{th1} is not sufficient in the general case.
It becomes sufficient for $n\geq 6$ when the indices of the critical points of $f$ which are not extrema take only two values $k$ and $k+1$, where $k$ is between $2$ and $n-2$.
Moreover, if all points of label $+$ are above all points of label $-$ of same index, we can derive a computable arithmetical condition on the matrix of the boundary operator:

\begin{nntheo}[Theorem \ref{2indices}, Section \ref{twoindices}]
Let $\tilde{f}$ be a non-critical Morse germ along $\Sn$ for $n\geq 6$.
Assume that $f$ has only one local maximum and one local minimum, and that the indices of the other critical points can only be $k$ or $k+1$, with $2 \leq k \leq n-2$. 
Assume also that $f(a)>f(b)$ if $\ind (a) = \ind(b)$ and the label of $a$ is $+$ and the label of $b$ is $-$.
Let $X$ be a Morse-Smale pseudo-gradient adapted to $f$ and denote by $\partial$ its boundary operator.
The germ $\tilde{f}$ extends non-critically if and only if 
\[\det (\partial_{++,k+1}) \equiv ~\pm 1~[d_1(\partial_{+-,k+1})],\]
where $d_1(\partial_{+-,k+1})$ is the $g.c.d.$ of the coefficients of the matrix $\partial_{+-,k+1}$.

\end{nntheo}

Here is the structure of the article:

In \textbf{Section 2}, we explain the starting point of the article, which is the one of Barannikov \cite{Bar} and which uses generic paths of Morse functions and Cerf theory.

In \textbf{Section 3}, we prove some lemmas necessary to prove Theorem \ref{th1}, and then prove Theorem \ref{th1}, the main result of the present article.

In \textbf{Section 4}, we show that the condition of Theorem \ref{th1} is not sufficient in all generality, by using results on Morse theory on manifolds with boundary. We also prove that given a germ $\tilde{f}$ which fulfills conditions of Theorem \ref{th1}, one can always find another germ $\tilde{f_1}$ which has the same homological properties as $\tilde{f}$ (in fact the same adapted pseudo-gradient) and which extends non-critically.

In \textbf{Section 5}, we give the computable condition when the critical points of the function which are not extrema can only take two values, $k$ and $k+1$, with $k$ between $2$ and $n-2$, and with some assumptions on the critical values.

 
  The main techniques used in this article are those of the $h$-cobordism theorem, explained in the classical book of Milnor \cite{Miln}. 
  We also use results about the change of topologies of level manifolds of a Morse function defined on a manifold with boundary and Cerf's theory about paths of Morse function, \cite{Cerf}.

\section{Preliminaries}
This section introduces the notation used throughout the article.
It also recalls classical results of Morse theory and Cerf theory.

\subsection{Notation}
 
\begin{itemize}
\item If two groups, $\mathbb{Z}$-modules or chain complexes $G$ and $H$ are isomorphic, it will be denoted by $G \simeq H $.
\item If $\mathcal{C}$ is a set and $A$ an integral domain, we define $A\mathcal{C}$ the free $A$-module generated by the elements of $\mathcal{C}$.
As a convention, we set $A \emptyset := 0_A$, the $A$-module reduced to $0$.
Most of the time, we will in fact have $A=\mathbb{Z}$, the ring of integers, except in Section \ref{barra}, where $A$ will be a field.
\item If $A \mathcal{C}_1,..., A\mathcal{C}_n$ is a sequence of such modules, we will denote by $A \mathcal{C}$ their direct sum $\bigoplus_k A \mathcal{C}_k$.
\item If we have a sequence of homomorphisms indexed by $(j,k)\in \mathbb{Z}^2$ as $\phi_{j,k} : A \mathcal{C}_k \to A \mathcal{D}_j$, we denote by $\phi$ their extension, such that $\phi : A \mathcal{C} \to A \mathcal{D}$ and $\phi (x) = \sum_j \phi_{j,k} (x)$ for $x \in A \mathcal{C}_k$.
\item In the same way, if $\phi$ is a homomorphism of $A$-modules defined on the direct sum $A \mathcal{C}$ of $A$-modules $A \mathcal{C}_k$ then $\phi_k$ will denote its restriction to $A \mathcal{C}_k$.
\item If $f: M \to \mathbb{R}$ is a continuous function, the words \emph{above} and \emph{below} will always be taken relatively to $f$ if there is no possible confusion with another function.
\item If $f : M \to \mathbb{R}$ is smooth, then $\mathcal{C}(f)$ will denote the set of critical points of $f$.
\item If $f : M \to \mathbb{R}$ is a Morse function, $\mathcal{C}_k(f)$ will denote the set of critical points of $f$ of index $k$. 
\item A smooth function $f:M \to \RR$ is said to be \emph{non-critical} if $\mathcal{C}(f)=\emptyset$.
\item A Morse function $f$ is said to be \emph{excellent} if any two of its critical points have distinct critical values.
\item The capital letter $I_p$ will often denote the identity matrix. We may sometimes forget the subscript $p$, denoting the dimension of the module on which we operate, if there is no possible confusion.
\end{itemize}

\subsection{Definitions of Morse theory for manifolds with boundary}

We now introduce a few definitions of Morse theory on manifolds with boundary.
The theory developed in the literature is much larger than the following paragraph, thus, we refer to \cite{Laud2} and \cite{BNR} for details.

Let $M$ be a manifold with boundary $\partial M$ such that $\partial M$ is a closed manifold.
We can consider a neighborhood $\mathcal{U}$ of $\partial M$ in $M$ such that $\mathcal{U}$ is diffeomorphic to $\partial M \times [0, \varepsilon )$, where $\varepsilon > 0$.
Such a neighborhood is called a \emph{collar neighborhood} of $\partial M$ in $M$.
We will use the notation $(x,t) \in \partial M \times [0,\varepsilon )$ for such a neighborhood.
We give the following definition of a Morse function on a manifold with boundary, taken from \cite{Laud2}:
\begin{definition}
A smooth function $F:M\rightarrow \mathbb{R}$ is a \emph{Morse function} when all critical points of $F$ lie in the interior of $M$, are non-degenerate and if $F$ restricts to a Morse function on $\partial M$.
\end{definition}
\begin{remark}
The definition of a Morse function on a manifold with boundary varies in the literature. 
For example, the definition taken by Borodzik, Nemethi and Ranicki \cite{BNR} allows critical points on the boundary but we explain why the two definitions are somewhat \emph{topologically equivalent} in Section \ref{surgmb}.
\end{remark}

Let $F:M \to \mathbb{R}$ be a Morse function.
Given the neighborhood $\mathcal{U}$ of $\partial M$, we denote by $d_x F (x,t)$ the derivative of $F$ tangent to $\partial M$ and by $\partial_t F (x,t)$ the derivative of $F$ with respect to $t$.
Let $x \in \mathcal{C}(f)$. 
As $ \mathcal{C}(F) = \emptyset$, we have that $\partial_t F (x,0) \neq 0$.
Thus, the critical set of $f$ splits into two sets, the set of points $x$ for which $\partial_t F (x,0)<0$, that we denote by $\mathcal{C}^+(F)$, and the set of points $x$ for which $\partial_t F(x,0) > 0$, denoted by $\mathcal{C}^-(F)$.
Notice that these two sets depend on $F$, whereas $\mathcal{C}(f)$ only depends on $f$.
Here, we took the notation of Curley \cite{Cur}.
We denote by $\mathcal{C}_k^{\ell}(\tilde{f})$ the points of label $\ell \in \{+,-\}$ and index $k$. If $\ell=+$, its cardinality will be denoted by $p_k$ and if $\ell = -$ by $q_k$. 

\begin{remark} 
In \cite{Laud2}, a Dirichlet point is a point labeled $+$ and a Neumann point is a point labeled $-$.
Borodzik, Nemethi and Ranicki \cite{BNR} call a labeled $+$ point a "boundary stable critical point", and call a labeled $-$ point a "boundary unstable critical point".
\end{remark}

We give the definition of the germ extending a Morse function:

\begin{definition}[Non-critical germ of a Morse function]
Given a Morse function $f$ defined on a closed manifold $\partial M$, a \emph{Morse germ} extending $f$ is the equivalence class of functions ${\tilde{f}:\partial M \times [0,\varepsilon) \rightarrow \RR}$ such that $\tilde{f}$ restricts to $f$ on $\partial M \times \{0\}$, up to restriction to a smaller collar $\partial M \times [0,\varepsilon ')$ with $\varepsilon ' < \varepsilon $ and such that $\tilde{f}$ has no critical point.
\end{definition}

We will often identify implicitly a representative $\tilde{f}$ of the germ and the germ itself.

We will only consider $\partial M=\Sn$, the standard sphere of dimension $n$ embedded in the euclidean space $\RR ^{n+1}$.
 We will always see the collar neighborhood $\Sn \times [0, \varepsilon )$ as a neighborhood of the unit sphere in the ball $\Dn$, with the embedding $(x,t) \mapsto (1-t)x$.
  The sphere $\Sn \times \{t\}$ then represents the sphere of radius $1-t$.

The main problem that the article tackles is the following:

Let $\tilde{f}$ be a non-critical Morse germ defined on $\Sn \times [0,\varepsilon)$. 
When does $\tilde{f}$ extend to a non-critical function $F$ on $\Dn$?

Throughout the article, all restricted functions $f$ that we wish to extend non-critically will be assumed excellent.

 \subsection{Labeled Reeb graph, or Curley graph}
  A useful tool to visualize a Morse function is its Reeb graph.

The \emph{Reeb graph} of a Morse function $f$ defined on a closed manifold $M$ is the graph $\Gamma(f)$ obtained by the equivalence relation:

\begin{center}$x\sim y \Leftrightarrow$ $x$ and $y$ are in the same connected component in the level set $f^{-1}(f(x))$.\end{center}

Notice that $x\sim y$ implies $f(x)=f(y)$.
 We define $pr : M \rightarrow \Gamma(f) $, the projection map on the graph.

The vertices of the graph are in correspondence with the critical values of $f$ and we equip the graph with a height function which maps a point $pr(p)$ of the graph to $f(p) $.
From the previous definition, we can define a graph from a non-critical Morse germ $\tilde{f}$ by adding information to the Reeb graph of its restriction $f$.
We just label each vertex of the graph with the label of the corresponding critical point of $f$, that is we add a $+$ or a $-$ next to the vertex.
 We call this augmented graph the \emph{Curley graph} of the non-critical germ $\tilde{f}$ (see \cite{Cur} from which we take the notation).
 See Figure \ref{curleygraph} for an example of a Curley graph.

In Sections \ref{homo} and \ref{twoindices}, we consider Morse functions having one local maximum and one local minimum.
\begin{proposition}
\label{reebprop} 
Let $n \geq 2$.
If a Morse function defined on $\Sn$, has only one local minimum and one local maximum, the level sets of $f$ are connected.
\end{proposition}
To prove the proposition, we prove the following lemma which is interesting in itself:
\begin{lemma}
\label{embedreeb}
If $f:M \to \RR$ is a Morse function defined on a manifold $M$, it is always possible to embed its Reeb graph $\Gamma (f)$ into $M$ through a map $\iota$, such that $pr \circ \iota = id_{\Gamma (f)}$.
\end{lemma}
\begin{proof}[Proof of Lemma \ref{embedreeb}]
It is sufficient to link by a strictly increasing line any two consecutive critical points whose projections to $\Gamma (f)$ are linked by an edge. 
The embedding of the whole $\Gamma (f)$ is then given by the union of the embeddings of these lines.
Let $a$ and $b$ be two such critical points, with $f(a)>f(b)$.
If we denote by $\ell$ the closed edge connecting $pr(a)$ and $pr(b)$, then $pr^{-1}(\ell)$ is a connected manifold with boundary, where it is easy to find the desired line.
Denote by $\iota : \Gamma (f) \hookrightarrow M$ the embedding.
We have that $pr \circ \iota = id_{\Gamma(f)}$ by construction.
\end{proof}

\begin{proof}[Proof of Proposition \ref{reebprop}]
It is classical Morse theory that for $n\geq 2$, the number of connected components strictly increases only when one passes above critical points of index $0$ or $n-1$.
In the same way, the number of connected components strictly decreases only when one passes above critical points of index $1$ or $n$.
Assume now that $M=\Sn$, and that $f$ has only one local minimum and one local maximum.
 If one of the level sets of $f$ has more than two connected components, then the Reeb graph must present a non-trivial loop.
 Indeed, let $x$ be a point in one of the connected component and $y$ be a point in the other.
 For a generic set of pseudo-gradient (see Section \ref{pseudograd} for a definition), the gradient line (which is strictly decreasing) passing through $x$ (resp. $y$) connects the global maximum to the global minimum.
 The union of these two gradient lines thus forms a loop in $\Sn$ which projects to a loop in $\Gamma (f)$, which is non-trivial by construction. 
 But it becomes trivial through the embedding $\iota$, since $\pi_1(\Sn)\simeq\{1\}$.
 As $pr \circ \iota=id_{\Gamma (f)}$ there is a contradiction.   
\end{proof}
 
 The Curley graph of a germ whose induced Morse function has only one local maximum and one local minimum is then an ordered sequence of labeled vertices, and each vertex is linked to the one above and the one below by a segment.
 
\begin{figure}[!ht]
\begin{center}
\includegraphics[scale=0.7]{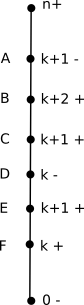}
\caption{Curley graph of a germ with no other critical points of index $0$ or $n$ different from the global minimum and maximum}
\label{curleygraph}
\end{center}
\end{figure}

\subsection{Cobordisms between two spheres \texorpdfstring{$\Sn$}{Lg}}
\label{cobord}
We explain in this subsection the starting point of this article, which is from Barannikov \cite{Bar}, see also \cite{Lau}. 

Let $\tilde{f}$ be a Morse germ.
Suppose that $\tilde{f}$ has an extension $F:\Dn \to \RR$ without critical points, and pick a regular point $z$ in $\Dn$ and a small ball of dimension $n+1$ denoted by $B$ around $z$ such that $F|_{\partial B}$ has only two critical points: a minimum and a maximum.
 There is a germ $\widetilde{f^1}$ whose representative is the function $F$ restricted to a collar neighborhood of $\partial B$ in $B$.
  This germ has its maximum labeled $+$ and its minimum labeled $-$.
   As we will see in Subsection \ref{trivialsec}, the germ $\widetilde{f^1}$ is trivial in the sense that it is the simplest germ that extends to the ball non-critically.
 Notice that $\Dn \setminus B$ is diffeomorphic to a cylinder $\Sn \times [0,1]$.
 To each $t$ of the second factor, the restriction of $F$ to $\Sn \times \{t\}$ is a function $f^t$ on $\Sn$.
 Through this identification, $F$ defines a path of functions from $f$ to $f^1$.
 We can slightly modify the function $F$ to make this path generic, such that $f^t$ is an excellent Morse function for all but finitely many times for which we have one of the three bifurcations --- \emph{birth}, \emph{death} or \emph{crossing} --- described in Cerf \cite[p. 24]{Cerf}.
       For $t$ in $[0,1]$, we also have a germ extending $f^t$ represented by 
       \[\widetilde{f^t} : (x,s) \mapsto f(x,s+t),~~(x,s) \in \Sn \times [0,\varepsilon).\]
Then, a non-critical extension gives a generic path of function $(f^t)_{t\in [0,1]}$ and a path of germs $(\tilde{f^t})_{t\in [0,1]}$.
 As $f^1$ has only one maximum and one minimum as critical points, any critical point of $f^0$ other than the maximum and the minimum gets killed during this path, with a death bifurcation as explained by Cerf in \cite[p. 24-25]{Cerf}.
 We say that two critical points cancel each other \emph{non critically} when the germ defined by the path has no critical point at the death bifurcation.

One may notice the following lemmas, which are already included in \cite[Theorem 1]{Bar}:

\begin{lemma}
\label{fate}
Suppose $(a_t)_{t\in[0,1]}$ is a smooth path of non-degenerate critical points for a generic non-critical path of function $(f^t)_{t\in [0,1]}$ such that the global function $F:(x,t) \mapsto f^t(x)$ has no critical point.
 Then the sign of the real number $\partial_t F (a_t)$ is constant, equal to the sign of $\partial_t F(a_0,0)$, which is the same than $\partial_t \tilde{f^0}(a_0,0)$.  
\end{lemma}

In other words, "if a critical point is labeled $+$ (resp. $-$), it goes down (resp. up) during the path until its possible death".

\begin{proof}
If there is a $t_0$ such that $\partial_t F(a_{t_0},t_0) <0$ whereas $\partial_t \tilde{f^0}(a_0,0)>0$, then, by the intermediate value theorem, there is a time $t'$ such that $\partial_t F(a_{t'},t')=0$.
 But as $a_{t'}$ is a critical point for $f^{t'}$, we would have $d F (a_{t'},t') =0$, which is exactly what we cannot have.
\end{proof}

\begin{lemma}
\label{noncrit}
During a non-critical path $(\tilde{f^t})_{t\in [0,1]}$, if two points cancel each other non-critically then they have the same label.
\end{lemma}

\begin{proof}
Suppose we kill two critical points of $f^0$, say $b$ of index $k+1$ and $a$ of index $k$, with a generic path of functions $f^t$, such that $b$ is labeled $+$ and $a$ is labeled $-$.
 If the death bifurcation happens at time $t_0$, we then have two paths of non-degenerate critical points ${b}_t$ and $a_t$ for $f^t$ with $t$ in $[0,t_0)$.
We identify the path of functions with a non-critical extension of $f$ and denote it by $F : (x,t) \mapsto F(x,t)$.
 Since $\partial_t F(a_t)>0$ and $\partial_t F({b}_t)<0$ for $t<t_0$, then at the death time $t_0$, we have a limit point 
 \[c_l=\lim\limits_{t\to t_0} b_t =\lim\limits_{t\to t_0} a_t \]
  such that $\partial_t F({b}_{t_0})\leq 0$ and $\partial_t F(a_{t_0})\geq 0$ since the partial derivative $\partial_t F$ is continuous. 
 Then $\partial_{t_0} F(c_l)=0$, but as $c_l$ is also a critical point of the function $F(\centerdot , t_0):x\mapsto F(x,t_0)$, we have that $c_l$ is a critical point of $F$, but we assumed that $F$ has no critical point. 
\end{proof}

\subsection{Adapted pseudo-gradients and handle slides}
\label{pseudograd}

\subsubsection{Definitions and handle slides}
In this subsection, we fix a Morse function $f$ on a closed manifold $M$ of dimension $n$.
The proof of the main theorem of the article, Theorem \ref{th1} stated in Section \ref{nec}, deals with pseudo-gradient vector fields \emph{adapted} to $f$, with the following definition:
\begin{definition}
A vector field $X$ is a \emph{pseudo-gradient vector field adapted to $f$} when:
\begin{itemize}
\item $df_x(X(x))<0$ for all $x \notin \mathcal{C}(f)$,
\item for all $a\in \mathcal{C}(f)$, there are Morse coordinates $(y_1,...,y_n)$, for which \[f(y) = f(a) - \sum\limits_{1\leq j \leq k} y_j^2 +\sum\limits_{k+1 \leq j \leq n} y_j^2,\] and
\[ X(y)=(2y_1,...,2y_k,-2y_{k+1},...,-2y_n)\]
in these coordinates.
\end{itemize}
\end{definition}

Such a pseudo-gradient is said to be \emph{Morse-Smale} when unstable manifolds and stable manifolds of critical points intersect transversally. 
With such assumption, and a choice of orientations for each unstable manifold, we get a \emph{Morse} boundary operator $\partial$.

We recall in this subsection results about pseudo-gradients adapted to Morse functions.
The material can be found in \cite{Laud4} and \cite[Sec. 7]{Miln}.
We are in particular interested in the description of generic paths of adapted pseudo-gradients.

We first define the notion of handle slide, and give its effect on the boundary operator. 
It is a notion introduced in \cite[Sec. 7]{Miln}.

Let $a$ and $b$ be two critical points of some fixed Morse function $f$ of same index $k$ with $f(a)>f(b)$.
We assume we are given an order on $\mathcal{C}_j(f)$ for all index $j$, and that for this order $a$ is the $i$-th point of $\mathcal{C}_k(f)$ and $b$ is the $\ell$-th point.
With such orders on critical points, we can write $\partial$ in matricial notation.
  For any critical point $c$ and any pseudo-gradient $X$ adapted to $f$, denote by $W^u(c,X)$ the unstable manifold of $c$ relative to $X$.
  Assume we are given the orientations of the unstable manifolds $W^u(c,X)$ for all $c$, the choices of these orientations being arbitrary.
   Then, it is not difficult to give to $f^{-1}(d)\cap W^u(d,X)$ an orientation which varies in a consistent way when $d$ changes continuously.
   Given an oriented manifold $M$, we denote by $-M$ the manifold with orientation reversed.
  We denote by $\sharp$ the connected sum of two oriented manifolds --- not necessarily closed ---, as introduced in the beginning of \cite{KM}.
Let $t\mapsto X^t$ be a path of pseudo-gradients adapted to $f$.

\begin{definition}[Handle slide]
We say that there is a \emph{handle slide} of $a$ over $b$ at time $t_0$ if $X^t$ is Morse-Smale for all $t$ except at some time $t_0$ for which:
\begin{itemize}
\item There is an orbit of $X^{t_0}$ connecting $a$ and $b$,
 \item $f^{-1}(f(b)-\varepsilon) \cap W^u(a,X^{t_0+\varepsilon})$ is diffeomorphic to \[ \left[ f^{-1}(f(b)-\varepsilon) \cap W^u(a,X^{t_0-\varepsilon}) \right] \sharp \pm \left[f^{-1}(f(b)-\varepsilon) \cap W^u(b,X^{t_0-\varepsilon})  \right] .\]
 \end{itemize}
 \end{definition}
 At time $t_0$, the pseudo-gradient $X^{t_0}$ is still adapted to $f$ but it is no longer Morse-Smale.

We have the following effect on the boundary operator $\partial^t$ associated to $X^t$:
\begin{align*}
\partial^{t_0+\varepsilon}_j & = \partial^{t_0-\varepsilon}_j & \text{ for $j\neq k,k+1$},\\ 
\partial^{t_0+\varepsilon}_{k+1} & =  (I+s E_{\ell,i}) \partial^{t_0-\varepsilon}_{k+1}, \\
\partial^{t_0+\varepsilon}_{k}   & =  \partial^{t_0-\varepsilon}_{k}(I-sE_{\ell,i}) . \\
\end{align*}
The matrix $E_{\ell,i}$ stands for the elementary matrix whose coefficients are all $0$ but the one in position $(\ell,i)$ which is $1$, and $I$ is the identity matrix.
In the equation, $s\in \{+1,-1\}$. 
When $s=+1$, we will speak of \emph{positive handle slide}, and if $s=-1$, we will speak of \emph{negative handle slide}.

Notice that the effect of a handle slide on the boundary operators is asymmetric in $a$ and $b$.
For that reason, when a handle slide involving two critical points $a$ and $b$ of same index occurs, we will always precise the order of the critical values of $a$ and $b$.
Moreover, we see that it is necessary to have a strictly descending line joining the level set of $a$ and the one of $b$ for a handle slide to be available.
Up to reparametrisation, it is a line $ u \mapsto \gamma (u)$ such that $f\left(\gamma(u)\right)=(1-u)f(a)+uf(b)$.

We have the two properties:

\begin{theorem}
\label{pggenericity}
Let $t\mapsto X^t$ be a path of pseudo-gradients adapted to $f$.
Then, we can always slightly modify $t\mapsto X^t$ to have a path of pseudo-gradients adapted to $f$ which is Morse-Smale for all but finitely many times for which an handle slide occurs.
\end{theorem}

We will also use \cite[Theorem 7.6]{Miln}, in Section $4$ and $5$:

\begin{theorem}
\label{change}
If there is a strictly descending line connecting the level sets $f^{-1}(f(a))$ and $f^{-1}(f(b))$, and if $f(a)>f(b)$, any handle slide of $a$ over $b$ is possible.
In other words, for any Morse-Smale pseudo-gradient $X^0$ adapted to $f$, there is a Morse-Smale pseudo-gradient $X^1$ linked by a generic path of pseudo-gradients to $X^0$ whose boundary operator $\partial^1$ is obtained from $\partial^0$ by the following equations:
\begin{align*}
\partial^1_k &= \partial^0_k (I -sE_{\ell,i}),\\
\partial^1_{k+1} &= (I + sE_{\ell,i})\partial^0_{k+1},
\end{align*}
for $s=1$ or $s=-1$ depending on whether the handle slide is positive or negative.   
\end{theorem}

We also recall \cite[Corollary 2.2]{Laud5} about handle crossings, that we slightly modify in order to adapt it to our needs:
\begin{proposition}
Let $t\mapsto f^t$ be a generic path of functions between two Morse functions $f^0$ and $f^1$.
Assume that the only bifurcations occuring during the path are handle crossings and that two different points only cross once during the path.
Then, there is a vector field $X$ which is a Morse-Smale pseudo-gradient adapted to $f^0$ and $f^1$.
\end{proposition}

\subsubsection{Independent bifurcations}
Finally, we recall results about independent bifurcations, for birth/death singularities. 
 The definition comes from \cite[Lemma 6.1]{HatWag} related to the notion of \emph{independent birth-death singularities}.
 \begin{definition}
 Let $X$ be a pseudo-gradient adapted to a Morse function $f$.
 Two critical points $a$ and $b$ of a function $f$ are independent for $X$ when we have:
 \[\left(W^u(a,X)\cup W^s(a,X) \right)\cap \left(W^u(b,X)\cup W^s(b,X) \right) = \emptyset .\]
 \end{definition}
We have from \cite[Lemma 6.1]{HatWag}:
\begin{lemma} [Independent singularity]
\label{independent}
 If $(f^t,X^t)$ is a generic path of functions and adapted pseudo-gradients, then $X^t$ can be deformed to a path of adapted pseudo-gradients such that all birth/death bifurcations of pairs of indices different from $(1,0)$ or $(n,n-1)$ are independent from points of indices comprised between $1$ and $n-1$.
 \end{lemma}
 We adapt \cite[Lemma 6.1]{HatWag} and its proof to the case of a $(1,0)$ or $(n,n-1)$ birth/death bifurcation:
 \begin{lemma}[Independent singularity for extremal indices]
 We assume that a birth/death bifurcation occurs during a generic path, and that the pair $(a,b)$ which dies or appears is of index $(1,0)$ or $(n,n-1)$.
 We have the following, where $\partial^t$ denotes the boundary operator associated to $X^t$:
 \label{independentextended}
 \begin{itemize}

 \item \textbf{Death bifurcation of a pair $(a,b)$ of index $(1,0)$.} There is a Morse-Smale pseudo-gradient $X^{t_0-\varepsilon}$ adapted to $f^{t_0-\varepsilon}$ such that $\partial^{t_0-\varepsilon} d$ has no component along $a$ or $b$ for any critical point ${d \in \mathcal{C}_k(f^{t_0-\varepsilon}) \setminus \{a\} }$ with $k \in \{1,2\}$.
 \item \textbf{Death bifurcation of a pair $(a,b)$ of index $(n,n-1)$.} There is a Morse-Smale pseudo-gradient $X^{t_0-\varepsilon}$ adapted to $f^{t_0-\varepsilon}$ such that $\partial^{t_0-\varepsilon} a=\pm b$ and $\partial^{t_0-\varepsilon}b=0$.
 \item \textbf{Birth bifurcation of a pair $(a,b)$ of index $(1,0)$.} There is a Morse-Smale pseudo-gradient $X^{t_0+\varepsilon}$ adapted to $f^{t_0+\varepsilon}$ such that $\partial^{t_0+\varepsilon} d$ has no component along $a$ or $b$ for any critical point ${d \in \mathcal{C}_k(f^{t_0+\varepsilon}) \setminus \{a\} }$ with $k \in \{1,2\}$.
 \item \textbf{Birth bifurcation of a pair $(a,b)$ of index $(n,n-1)$.} There is a Morse-Smale pseudo-gradient $X^{t_0+\varepsilon}$ adapted to $f^{t_0+\varepsilon}$ such that $\partial^{t_0+\varepsilon} a=\pm b$ and $\partial^{t_0+\varepsilon}b=0$.
 
 \end{itemize} 
 \end{lemma}
 \begin{proof}[Proof of Lemma \ref{independentextended}]
 We refer to the proof of \cite[Lemma 6.1]{HatWag} for details.
 
 Assume we are in the case of the death bifurcation of a pair $(a,b)$ of index $(1,0)$.
 Let $\alpha$ be the critical value of $a$.
 Then at time $t_0-\varepsilon$, in the level set $f^{-1}(\alpha+\eta)$, for any adapted pseudo-gradient $X$, the manifold \[W:= f^{-1}(\alpha+\eta)\cap\left[W^s(a,f^{t_0-\varepsilon},X^{t_0-\varepsilon})\cup W^s(b,f^{t_0-\varepsilon},X^{t_0-\varepsilon})\right]\] is a closed disk of dimension $n$.
 The sphere bounding the disk is $W^s(a,f^{t_0-\varepsilon},X^{t_0-\varepsilon})$. 
 By a dimensional argument, there is a point $x \in W$ which is not in
 \[\bigcup\limits_{c\in \mathcal{C}_1(f^{t_0-\varepsilon})\cup\mathcal{C}_2(f^{t_0-\varepsilon})} W^u(c,f^{t_0-\varepsilon},X).\]

 By an isotopy of $f^{-1}(\alpha +\eta )$ we can deform $X^{t_0-\varepsilon}$, and shrink $W$ into a small neighborhood of $x$.
 This done, the new adapted pseudo-gradient is such that $\partial^{t_0-\varepsilon} d$ has no component along $a$ or $b$ for any critical point ${d \in \mathcal{C}_k(f^{t_0-\varepsilon}) \setminus \{a\} }$ with $k \in \{1,2\}$.
 
 In the same manner, by trading "$W^u$" and "$W^s$" above, we get the second item.
 By inverting time, we get the two last items. 
\end{proof}

With an abuse of notation, we will say that a birth/death singularity of a pair of extremal index is \emph{independent} if the path $t\mapsto(f^t,X^t)$ is as in Lemma \ref{independentextended}. 

\subsection{Trivial germ}

\label{trivialsec}

\begin{definition}[Trivial germ]

A non-critical germ $\tilde{f}$ defined along the sphere is trivial if the function $f$ has only two critical points: a maximum and a minimum, and if the maximum is labeled $+$ and the minimum is labeled $-$. 
\end{definition}

\begin{proposition}[Trivial extension\footnote{ I am deeply indebted to Fran\c{c}ois Laudenbach for the hints of the proof of this proposition.}]
\label{trivial2}
A trivial germ can be extended without critical points to the ball.
\end{proposition}

We need the following lemma before the proof of the proposition.
Recall that a pseudo-isotopy between two diffeomorphisms $g_0$ and $g_1$ of a manifold $M$ is a diffeomorphism of $M\times [0,1]$ which restricts to $g_0$  on $M \times \{0\}$ and to $g_1$ on $M \times \{1\}$. 

\begin{lemma}
\label{pseudoiso}
If a diffeomorphism $g$ defined on a manifold $M$ is pseudo-isotopic to the identity via $\phi : M \times [0,1] \rightarrow M \times [0,1]$, then $\phi$ is pseudo-isotopic to the identity.
\end{lemma}

\begin{proof}[Proof of Lemma \ref{pseudoiso}]
The proof is inspired from Hatcher and Wagoner \cite{HatWag}, citing the following result of Cerf \cite[Theorem 5 p.293]{Cerf2}.
 Let $g$ be diffeomorphism of a manifold $M$ pseudo-isotopic to the identity.
  Consider the set of pseudo-isotopies constant on the neighborhood of the boundaries of the cylinder $M\times [0,1]$, that is, pseudo-isotopies which are $(x,t)\mapsto (g(x),t)$ on $M\times [0,\varepsilon]$ and $(x,t)\mapsto (x,t)$ on $M\times [1-\varepsilon ,1]$.
   The result of Cerf states that this set is a deformation retract of the set of pseudo-isotopies from $g$ to the identity. 

Consider $\phi$, a pseudo-isotopy from $g$ to $id_M$ being constant on the neighborhood of the boundaries.
\begin{figure}
\begin{center}
\includegraphics[scale=0.5]{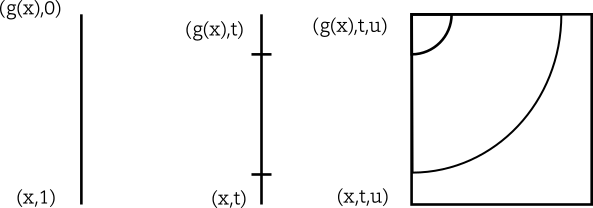}
\caption{Extension of the pseudo-isotopy at $x$ in $M$ fixed}
\label{coord}
\end{center}
\end{figure}
We denote by $(x,t,u)$ points of the double cylinder $M\times[0,1]\times [0,1]$ in Cartesian coordinates.

 We also use polar coordinates $ (x,r,\theta)$ on $\left(M\times[0,1]\times[0,1]\right)\setminus \left(M \times \{0\}\times \{0\}\right)$, with $r$ being $\sqrt{t^2+u^2}$ and $\tan (\theta) = \frac{u}{t}$. 
We define the pseudo-isotopy of $M\times ]0,1]\times [0,\frac{\pi}{2}] $ to be $(x,\theta,r) \mapsto (\phi(x,r),\theta)$, that is, we apply $\phi$ on each cylinder $\{\theta = constant\}$. 

 One can see that it is defined everywhere, since $\phi$ is locally constant on neighborhoods of $M\times \{0\}$ and $M\times \{1\}$. It is a diffeomorphism of $M\times[0,1]\times[0,1]$. 
 It coincides with $\phi$ on $M\times[0,1]\times \{0\}$ and $id_{M\times[0,1]}$ on $M\times[0,1]\times\{1\}$.
It defines a pseudo-isotopy from $\phi$ to the identity.  
\end{proof}

\begin{proof}[Proof of Proposition \ref{trivial2}]
 Let $\tilde{f}$ be a Morse germ.
  The problem remains the same when we consider $h\circ \tilde{f}$ instead of $\tilde{f}$, where $h$ is a diffeomorphism of $\RR$.
  Thus, we can use germs with their maximum (resp. minimum) being 1 (resp. $-1$). The point on which a function takes its minimum (resp. maximum) will be the \emph{south pole} (resp. \emph{north pole}), that we denote by $S$ (resp. $N$).
  Consider the germ given by the ($n+1$)-st coordinate of the ball $\Dn$ embedded in $\RR ^{n+1}$, that is the projection $pr_{n+1}: \Dn \to \RR$.
  It is the simplest example of a germ that extends without critical point. 
  Morse's lemma gives two diffeomorphisms 
  \[\phi_S : \mathcal{N}(S) \rightarrow pr_{n+1}^{-1} ( [-1,-1+ \delta])\]
   and  
  \[\phi_N : \mathcal{N}(N) \rightarrow pr_{n+1}^{-1} ( [1- \delta, 1])\]
   such that $\mathcal{N}(S)$ and $\mathcal{N}(N)$ are neighborhoods of the south pole and the north pole in the collar neighborhood of the sphere, with 
   \[\tilde{f}^{-1}(\{-1+\delta\})=\overline{\partial (\mathcal{N}(S)) \setminus \Sn}\] and 
   \[\tilde{f}^{-1}(\{1-\delta\})=\overline{\partial (\mathcal{N}(N)) \setminus \Sn}.\]
   In other words, the boundaries interior to the ball of those neighborhoods are level sets of the germ $\tilde{f}$. 
  See Figure~\ref{trivial}. 
 
 We have a level preserving diffeomorphism: 
\[\phi : \mathcal{D} \rightarrow \Sn \cap pr_{n+1}^{-1}([-1+\delta,1-\delta])\]  where $\mathcal{D}$ is $f^{-1}([-1+\delta,1-\delta])$.
Notice that $\Sn \cap pr_{n+1}^{-1}([-1+\delta,1-\delta])$ is diffeomorphic to $\mathbb{S}^{n-1}\times [-1+\delta , 1-\delta]$. We can assume $\phi: \mathcal{D} \cap f^{-1}(\{-1+\delta\}) \rightarrow \mathbb{S}^{n-1}$ is equal to $\phi_S$ restricted to $ f^{-1}(\{-1+\delta\})$.
 
 We choose pseudo-gradients $X(\tilde{f})$ and $X(pr_{n+1})$ for $\tilde{f}$ and $pr_{n+1}$ such that their flows preserve the respective Morse foliations away from the poles.
 It is possible up to renormalisation of the vector fields away from neighborhoods of the poles because the only critical points of $f$ or $pr_{n+1}$ are their respective extrema. 
 We also denote by $G_s(\tilde{f})$ the map that sends any point in $\tilde{f}^{-1}([-1+\delta,1-\delta])$ to $\tilde{f}^{-1}(\{s\})$ by the flow of the pseudo-gradient.
  In other words, we send a point $p$ in $\mathcal{D}$ to the point in $\tilde{f}^{-1}(\{s\})$ which is in the orbit of $p$ by the flow of $X(\tilde{f})$.
 We use an equivalent notation for $pr_{n+1}$. Notice that these maps are diffeomorphisms when restricted to a level set. 
 The diffeomorphism
  \[\Phi : (x,s) \mapsto G_s(pr_{n+1}) \circ \phi_N \circ G_{1-\delta}(\tilde{f}) \circ \phi ^{-1} (x,s)\]
 is a pseudo-isotopy on $\mathbb{S}^{n-1}\times [-1+\delta, 1-\delta]$. 
 Using Lemma \ref{pseudoiso}, we can define a pseudo-isotopy $\tilde{\Phi} : (x,s,t)\mapsto \Phi(x,s,t)$  on $\mathbb{S}^{n-1}\times [-1+\delta, 1-\delta]\times [0,\varepsilon]$ such that it is $\Phi$ on $\mathbb{S}^{n-1}\times [-1+\delta, 1-\delta]\times \{0\}$ and the identity on $\mathbb{S}^{n-1}\times [-1+\delta, 1-\delta]\times \{\varepsilon\}$.
 Then, we can smoothly glue a disk $\mathbb{D}^{n}$ to each level set of $\tilde{f}$ 
 using the extension of $\phi$.
We finally get a Morse function $F$ without critical point, as there is no topological changes on its level sets, defined on a manifold with boundary denoted by $W$, which is diffeomorphic to $\mathbb{D}^{n} \times [-1+\delta, 1-\delta]$.
 It is easy to glue $W$ to $\mathcal{N}(S)$, since the isotopy is the identity near the south pole.
Using a diffeomorphism of the disk $\mathbb{D}^{n}$, we can extend $F$ to $\mathcal{N}(N)$, and finally to a whole $\Dn$, without critical point.
 We obtain a non-critical function $F$ whose restriction to a collar neighborhood of the sphere is $\tilde{f}$, as desired.
\end{proof}

\begin{figure}

\begin{center}
\includegraphics[scale=0.4]{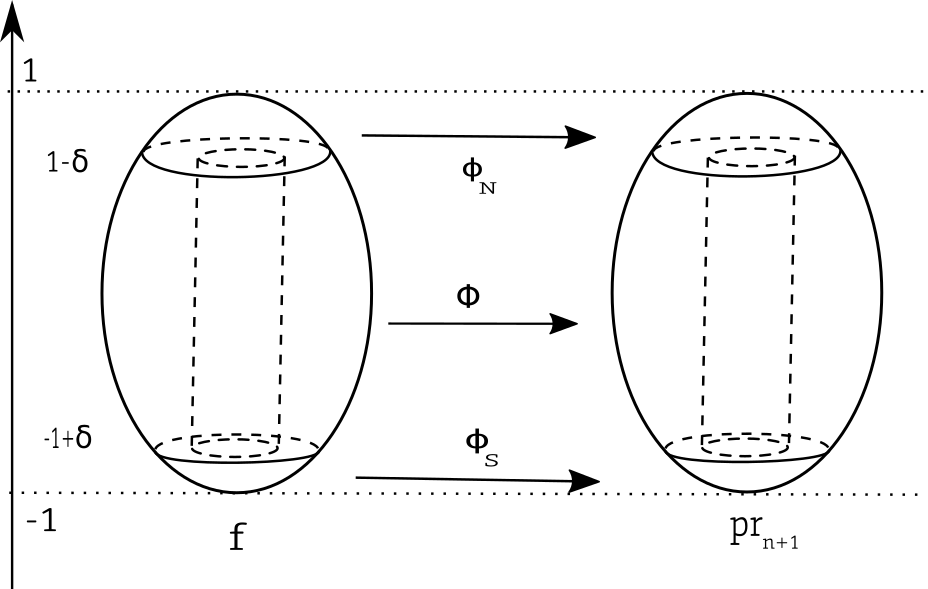}
\caption{Linking the Morse foliations}
\label{trivial}
\end{center}
\end{figure}

\subsection{Non-critical cancellation lemma}
Let $f$ be a Morse function on a closed manifold $M$ and $X$ a pseudo-gradient which is adapted to $f$.
 In this section, we do not suppose $f$ is excellent. 
 Let $a$ and $b$ be two critical points with the following assumptions: 
\begin{itemize}
\item $a$ is of index $k+1$ and $b$ is of index $k$.
\item $a$ and $b$ have consecutive critical values, that is, $f$ has no critical value in $]f(a),f(b)[$.
However, $f$ is not assumed excellent and $f(a)$ or $f(b)$ may correspond to several critical points.
\item There is one and only one transverse gradient line $\gamma$ from $a$ down to $b$, where transverse means that the unstable manifold of $a$ and the stable manifold of $b$ intersect transversely along this gradient line. 
\end{itemize}

With such assumptions, we say that $a$ and $b$ are in \emph{position of mutual cancellation}.
If a germ $\tilde{f}$ extends $f$ in a neighborhood of $\Sn$, and $a$ and $b$ are both of label $+$ (resp. $-$), we also assume that all gradient lines of $a$ (resp. $b$) except $\gamma$ reach $f^{-1}(\{f(b)-\delta\})$ (resp. $f^{-1}(\{f(a)+\delta\})$) with $\delta$ a positive number as small as wanted.
 We call this hypothesis the \emph{local excellence hypothesis}.

The following lemma deals with the cancellation of two such critical points.
 It is basically an adaptation of the cancellation lemma of Smale.
  The proof and the version of the lemma is taken from \cite{Laud3} and uses Cerf's methods. 

\begin{lemma}[Non-critical cancellation lemma][Laudenbach]
\label{cancellation}
Let $\tilde{f}$ be a Morse germ on $\Sn$ with $f$ not necessarily excellent. We choose an adapted pseudo-gradient $X$ for $f$. 
We suppose that there are $a \in \mathcal{C}^+_{k+1}(\tilde{f})$ and $b \in \mathcal{C}^+_{k}(\tilde{f})$ in position of mutual cancellation. 
We also suppose the local excellence hypothesis.
 Then, there is a path of functions $(f^t)_{t\in [0,1]}$ such that:
\begin{itemize}
 \item $\tilde{f}$ is represented by $(x,t)\mapsto f^t(x)$ for $0\leq t \leq \varepsilon$. 
 \item $\mathcal{C}_*(\tilde{f}_1)$ is naturally identified with $\mathcal{C}_*(\tilde{f})\setminus \{a,b\}$.
 \item $(d_x\tilde{f}^t (x,t),\partial_t\tilde{f}^t(x,t))\neq (0,0)$, for all $(x,t)$ in $\Sn \times [0,\varepsilon)$.
\end{itemize}
\end{lemma}

\begin{proof}
Let $\gamma$ be the orbit of the pseudo-gradient flow joining $a$ and $b$ and let $\mathcal{W}$ be a neighborhood of $\gamma$ in $\Sn$.
In \cite{Laud3}, the author finds a path of functions $t\mapsto f^t$ realizing the cancellation, and the path can be chosen such that 
 $\partial_t f^t(x) <0$ for all $x$ in $\mathcal{W}$.
 In particular, the path seen as a function of $(x,t) \in \Sn \times [0,1]$ has no critical point.
If the germ $\tilde{f}$ is such that $\partial_t \tilde{f}(x,0) <0$ for points $x$ in $\mathcal{W}$, applying the method of Laudenbach gives a path of functions $(t,x)\mapsto F(t,x)$ realizing the cancellation, such that $F$ restricts to $\tilde{f}$ for $t$ small and $dF \neq 0$, that is $F$ has no critical point.
 In order to get a germ with such property on its derivative, we will find a non-critical $\mathcal{C}^1$-path $H:(x,t)\mapsto H(x,t)$ for $t$ from $0$ to a real number $\delta$ such that:
 \begin{itemize}
 \item $H(x,t)$ gives a representative of $\tilde{f}$ for $t$ small,
 \item $H(x,\delta)=g(x)$ for all $x$ in $\Sn$,
 \item $\partial_t H(x,\delta) <0$ for $x$ in $\mathcal{W}$.
 \end{itemize}
  Consider the function 
  \[G :(x,t) \mapsto \tilde{f}(x,\varepsilon)+(t-\frac{t^2}{2\delta'})\partial_t \tilde{f}(x,\varepsilon) + \frac{t^2}{2 \delta'} g(x)\]
   with $t$ going from $0$ to a small $\delta '$, and with $g$ a function from $\Sn$ to $\mathbb{R}$ with $-m<g(x)<0$ on $\mathcal{W}$ for small $m$, and which is $0$ outside a small neighborhood of $\mathcal{W}$. 
  Notice that $\partial_t G(x,0)= \partial_t \tilde{f}(x,\varepsilon)$ and $\partial_t G(x,\delta')=g(x)<0$ for $x$ in $\mathcal{W}$.
  Concatenating the path $(t,x) \mapsto \tilde{f}(x,t)$ for $t$ from $0$ to $\varepsilon$ and the path $G$, we get the desired $\mathcal{C}^1$-path $H$.
  Concatenating $H$ with the path used to realize the cancellation, we get a $\mathcal{C}^1$ path realizing the cancellation with no critical point. 
  The path is smooth everywhere except at the junction of the two paths where it is just $\mathcal{C}^1$.
  However, it is always possible to smooth it keeping the property that it realizes the cancellation non-critically.
  
  The lemma is then proved.
\end{proof}

\section{Necessary homological condition for a non-critical extension }

 \label{nec}
We give in this section the main theorem of the article, which is a necessary condition for a germ to extend non critically. 
The proof will be based on Morse theory and Cerf theory. 
All of what is needed can be found in \cite{Laud4}.

Let $\tilde{f}$ be a non-critical Morse germ and $X$ an adapted pseudo-gradient which is Morse-Smale.
We assume $f$ excellent.
As we saw, if there is a non-critical extension, we have a generic path of functions $(f^t)_{t\in [0,1]}$ that extends the germ.
 The function $F:(x,t)\mapsto f^t(x)$ has no critical point, and $f^1$ has only two critical points, one maximum $max$ and one minimum $min$ such that $\partial_t F(max,1)<0$ and $\partial_t F(min,1)>0$.
In order to avoid conflicts of notation with the index and the label, the parameter $t$ will be used as a superscript to denote the dependency on time of the considered objects.
During this path of functions, all the critical points of $f$ and those born during the path get killed except two local extrema, one minimum and one maximum.
Notice that the two local extrema that remain at the end of the path are not necessarily some extrema of the initial function.
Indeed, births of pairs of critical points of indices $(1,0)$ or $(n,n-1)$ may happen during the path.

As we saw in Lemma \ref{noncrit}, two critical points of $f$ can get killed during the path $(f^t)_{t\in[0,1]}$ only if they have the same label.
For $k$ between $0$ and $n$, denote by $p_k$ the rank of $\mathbb{Z}\mathcal{C}_k^+(\tilde{f})$ and by $q_k$ the rank of $\mathbb{Z}\mathcal{C}_k^-(\tilde{f})$.
We also choose an order on $\mathcal{C}^{\ell}_k(f)$ for all $0\leq k \leq n$ and $\ell \in \{+,-\}$.
With such orders, given a boundary operator $\partial$ on $\mathbb{Z}\mathcal{C}(f)$, we can use matricial notations for all $\partial_{\ell_1 \ell_2 , k}$ for all labels $\ell_1$ and $\ell_2$, and all index $k$, such that $\partial_{\ell_1 \ell_2, k} :\mathbb{Z}\mathcal{C}^{\ell_2}_k(\tilde{f}) \to \mathbb{Z}\mathcal{C}^{\ell_1}_k(\tilde{f}) $. We have: \nopagebreak
\begin{center}
$\partial_{k+1} =\begin{pmatrix}
     \partial_{++,k+1} & \partial_{+-,k+1} \\
     \partial_{-+,k+1} & \partial_{--,k+1}
  \end{pmatrix} \in \mathcal{M}_{(p_k+q_k)\times (p_{k+1}+q_{k+1})}(\mathbb{Z})$,  
\end{center}

The next subsection introduce the key object of the article, namely the group of isomorphisms $G(\tilde{f})$ and how it is modified during a generic non-critical path of functions.
Subsection \ref{opposite} expose the links between $\tilde{f}$ and $-\tilde{f}$.
Subsection \ref{main} finally proves the theorem in two steps.
First we show that the hypotheses of Theorem \ref{th1}, which may seem to depend on the adapted pseudo-gradient, in fact only depend on the germ.
Second, we prove the theorem with a descending induction on the number of bifurcations occurring during the generic path of functions linking the germ that extends $\tilde{f}$ to a trivial germ $\tilde{h}$.

\subsection{ \texorpdfstring{$G(\tilde{f})$}{Lg}}
 \subsubsection{Definition}
  Let $k$ be an integer between $1$ to $n-1$.
 Denote by $(a_j)_{1\leq j \leq p_k}$ (resp. $(b_i)_{1\leq i \leq q_k}$) the points of $\mathcal{C}_{k}^+(\tilde{f})$ (resp. $\mathcal{C}_{k}^-(\tilde{f})$).
 Recall that if $f(a_j)<f(b_i)$, no handle slide of $a_j$ over $b_i$ is available.
 Moreover, as $F(a^t_j,t)$ is decreasing and $F(b^t_i,t)$ is increasing, we will not be able to perform such a handle slide at any time $t$.
 We then need to define a group of graded isomorphisms of $\mathbb{Z}\mathcal{C}(\tilde{f})$ representing the allowed handle slides between points of different labels, that is, the handle slides we are able to perform at time $0$ of points $(a_j)_{1\leq j \leq p_k}$ over points $(b_i)_{1\leq i \leq q_k}$. 
 It is the purpose of the groups $G_k(\tilde{f})$.

The following definition also consider $k=0$ and $k=n$. 
It is because we will indeed need to use groups $G_0(\tilde{f})$ and $G_n(\tilde{f})$ even if they do not have a geometric realization.
We denote by $N_{(i,j)}$ the coefficient in place $(i,j)$ of the matrix $N$.
\begin{definition}
\label{group}
Let $\tilde{f}$ be a Morse germ along $\Sn$.
We assume $f$ is excellent.
Let $0\leq k \leq n$.
We denote by $G_{k}(\tilde{f})$ the following group of automorphisms of the $\mathbb{Z}$-module $\mathbb{Z}\mathcal{C}_k(f)$.
An element of this group is a $(p_{k}+q_{k})\times (p_{k}+q_{k})$ invertible matrix $M_k$, such that 
\[M_k=\begin{pmatrix}
     I_{p_k} & 0 \\
      N_k & I_{q_k}
  \end{pmatrix} \in GL_{p_{k}+q_{k}}(\mathbb{Z})\]
   where ${N_k \in \mathcal{M}_{q_{k},p_{k}}(\mathbb{Z})}$ and  
   ${(N_k)_{(i,j)}=0}$
  if ${f(a_j)\leq f(b_{i}).}$
\end{definition}

Notice that 
\[M_k M'_k = \begin{pmatrix}
I_{p_k} & 0 \\
N_k + N'_k & I_{q_k} \\
\end{pmatrix},\]
thus the imposed nullity property of some coefficients is conserved under multiplication and the group is abelian.
The matrix \[N_k : \mathbb{Z}\mathcal{C}_k^+(\tilde{f}) \to \mathbb{Z}\mathcal{C}^-_k(\tilde{f})\] sends a point $a_j$ of label $+$ into the module generated by points which are below $a_j$ and have the same index and label $-$.

We define the global group $G(\tilde{f})$.
\begin{definition}
We define $G(\tilde{f})$ to be the group of graded isomorphisms $M:\mathbb{Z} \mathcal{C}(f)\rightarrow \mathbb{Z} \mathcal{C}(f)$ such that each restriction of $M$ to $\mathbb{Z}\mathcal{C}_{k}(f)$ for any $k$ between $0$ and $n$ is in the group $G_{k}(\tilde{f})$.
\end{definition} 
In the same way, this abelian group acts by conjugation on $\partial$. 
We use the notation $N_k$ for the down left submatrix of a matrix $M_k$ coming from an element of $G_k(\tilde{f})$.
 
If one conjugates a boundary operator $\partial$ by an element $M$ in $G(\tilde{f})$ then its restriction to each $\mathbb{Z}\mathcal{C}_{k+1}(f)$ reads:
\[ (M\partial M^{-1})_{k+1}= M_k\partial_{k+1} M_{k+1}^{-1}.\]
It leads to the four equations:
\begin{eqnarray*}
(M\partial M^{-1})_{++,k+1} & = & \partial_{++,k+1}-\partial_{+-,k+1}N_{k+1}, \\
(M\partial M^{-1})_{--,k+1} & = & \partial_{--,k+1}+N_k\partial_{+-,k+1}, \\
(M\partial M^{-1})_{-+,k+1} & = & \partial_{-+,k+1}-N_k\partial_{+-,k+1}N_{k+1}-\partial_{--,k+1}N_{k+1}+N_k\partial_{++,k+1}, \\ 
(M\partial M^{-1})_{+-,k+1} & = & \partial_{+-,k+1}.
\end{eqnarray*}

The following remark is important.

\begin{remark}[Difference between algebra and geometry]
\label{groupremark}
    For germs whose functions have more than two local extrema, the action of $G(\tilde{f})$ on the boundary operator is purely algebraic and does not necessarily correspond to the results of handle slides (which are of geometrical nature).
    It is due to several things. 
    First, we saw that to perform a handle slide of $a$ over $b$, there must be a descending line joining the level set of $a$ to the one of $b$.
    But the presence of other local extrema induces apparitions of connected components in the level sets, and for general functions with numerous local extrema, such lines between points may not exist.
     For example, the group $G(\widetilde{f_1})$ of the germ $\widetilde{f_1}$ whose Reeb graph is pictured on Figure \ref{afaire3} is not reduced to $0$.
    However, we cannot operate handle slides between the points of index $n-1$, as there is no descending line joining the level sets of those two points.

  \begin{figure}[!ht]
    \begin{center}
    
    \includegraphics[scale=0.6]{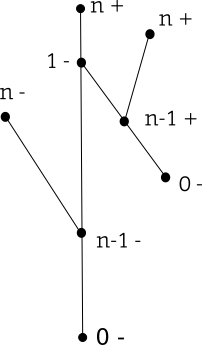}
    \caption{Reeb graph of $\tilde{f_1}$}
    \label{afaire3}
    \end{center}
    \end{figure}
    
    Second, handle slides between points of index $n$ or index $0$ are not defined.
    Thus, if $G_n(\tilde{f})$ or $G_0(\tilde{f})$ are not reduced to the identity, there is not necessarily a geometric interpretation to the conjugation of $\partial$ by an element of $G(\tilde{f})$.
    
    If $f$ has only two extrema, we have in particular that $G_n(\tilde{f})=\{I_{p_n+q_n}\}$ and $G_0(\tilde{f})=\{I_{p_0+q_0}\}$.
   If there is no points of indices $1$ and $n-1$ either, any action of an element $M\in G(\tilde{f})$ on a boundary operator $\partial$ given by a Morse-Smale pseudo-gradient adapted to $f$ corresponds to results of handle slides.
    It is mainly because of Proposition \ref{reebprop} and Theorem \ref{change}.
    
    As a conclusion, we emphasize that:
    \begin{itemize}
    \item we will often consider boundary operators $\partial$ given by Morse-Smale pseudo-gradients;
    \item if $\partial^1=M \partial M^{-1}$ and $M$ is an element in $G(\tilde{f})$ corresponding to actual handle slides, then there is a Morse-Smale pseudo-gradient adapted to $f$ whose associated boundary operator is $\partial^1$;
    \item however, in all generality, if $\partial^1=M \partial M^{-1}$ and $M \in G(\tilde{f})$, then there is no reason that there is a Morse-Smale pseudo-gradient adapted to $f$ whose associated boundary operator is $\partial^1$ (except if $f$ has only one local maximum, one local minimum, and no points of index $1$ or $n-1$). 
    \end{itemize}
      
    \end{remark}

\subsubsection{Modifications of \texorpdfstring{$G(\tilde{f})$}{Lg} through non-critical path of functions}
Let $(f^t)_{t\in[0,1]}$ be a non-critical path of functions continuing some Morse germ $\widetilde{f^0}$.
 Assume that one and only one bifurcation occurs during this path.
We describe the modification of the group $G(\widetilde{f^t})$ between times $t=0$ and $t=1$ according to the kind of bifurcation occurring.  
It will be used to prove Theorem \ref{th1}.

\begin{itemize}

\item  \textbf{Crossing of critical points of same label.}\\
 \textbf{Crossing of critical points of different labels and different indices.}
 
 In these cases, we have ${G(\widetilde{f^{1}})=G(\widetilde{f^0}).}$

 \item \textbf{Crossing of critical points of different labels and same index.}

 Let $k$ be the index of the points.
If there is a crossing between two points of different labels, then it is a point of label $+$ that goes below a point of label $-$, as the opposite is not possible.
After the crossing, we cannot make handle slides of the point labeled $+$ over the point labeled $-$ anymore, thus we have that $G(\widetilde{f^1})$ is a strict subgroup of $G(\widetilde{f^0})$.
 Precisely, if the point labeled $+$ is the $j$-th point of $\mathbb{Z}\mathcal{C}(f)$, and the point labeled $-$ is the $i$-th point, then $G_k(\widetilde{f^1})$ is the subgroup of matrices $M$ in $G(\tilde{f})$ whose restrictions $M_k$ have their $(i,j)$ coefficients being zero.
 It is isomorphic to $G_k(\widetilde{f^0})/ (I_{p_k+q_k}+\mathbb{Z} E_{i,j})$.

 \item \textbf{Death bifurcation.} 

 We first describe the modification when the canceled pair is of label $+$.

Let $k$ be such that the pair of critical points $(a,b)$ getting killed is of index $(k+1,k)$ and label $+$.
We have \[{\mathbb{Z}\mathcal{C}^+_j(\tilde{f}^0)\simeq \mathbb{Z}\mathcal{C}^+_j(\tilde{f}^1)\oplus \mathbb{Z},}\]
for $j\in \{k,k+1\}.$
\textit{A priori}, $\partial^1$ and $\partial^0$ do not act on isomorphic modules, but we have a natural injection
\[{ \mathbb{Z}\mathcal{C}(f^1) \hookrightarrow \mathbb{Z}\mathcal{C}(f^0).}\]
We will then identify $\mathbb{Z}\mathcal{C}(f^1)$ as a submodule of $\mathbb{Z}\mathcal{C}(f^0)$.
It induces a group injection
 \[{G(\widetilde{f^1})\hookrightarrow G(\widetilde{f^0}) ,}\]
 where a matrix $M^1$ in $G(\widetilde{f^1})$ is sent to a matrix $M^0$ restricting to $M^1$ on $\mathbb{Z}\mathcal{C}(f^1)$ and being the identity on $\mathbb{Z}\{a\}$ and $\mathbb{Z}\{b\}$.
 Thus we have $M^0 c =M^1 c$ if $c\in \mathbb{Z}\mathcal{C}(f^1)$, using the injection of modules above, and $M^0 c=c$ if $c\in \{a,b\}$.
 We then have for $j=k$, or $j=k+1$:
 \begin{equation}
 \label{eqmatrice}
 M^0_j=\begin{pmatrix}
  &           &  & 0      &  &   & \\
  & I_{p^1_j} &  & \vdots &  & 0 & \\
  &           &  & 0      &  &   & \\
 0& \hdots    & 0& 1      & 0& \hdots  & 0 \\
  &           &  & 0      &  &   & \\ 
  &  N^1_j    &  & \vdots &  & I_{q^1_j} & \\
  &           &  & 0      &  &           & \\
 \end{pmatrix}, 
 \end{equation}
 where $N^1_j$ is the down-left submatrix of $M^1_j$. 
  For the sake of notation, instead of taking the natural order given by the critical values of the points, we chose here to put the point $a$ (resp. $b$) after the critical points of $\mathcal{C}^+_{k+1}(\widetilde{f^1})$ (resp. $\mathcal{C}^+_{k}(\widetilde{f^1})$).
 We have $M^0_j=M^1_j$ for any index $j \neq k$, and $j\neq k+1$.
 
 If the label of the pair is $-$, we still have an injection 
 \[{G(\widetilde{f^1})\hookrightarrow G(\widetilde{f^0}).}\]
 Now, in matricial notation, we have: 
  \begin{equation}
 \label{eqmatriceb}
 M^0_{j}=\begin{pmatrix}
   &                      &  &  &                     & & 0 \\
   & I_{p^1_{j}}          &  &  & 0                   & & \vdots \\
   &                      &  &  &                     & & 0 \\
   &                      &  &  &                     & & 0 \\
   & N^1_j                &  &  &  I_{q^1_j}          & & \vdots \\
   &                      &  &  &                     & & 0 \\
 0 &    \hdots            & 0& 0&  \hdots             &0& 1 \\
\end{pmatrix},
 \end{equation}
for $j=k$ or $j=k+1$ and $M^0_j=M^1_j$ in other degrees.

 \item \textbf{Birth bifurcation.}
 
  The description is similar, intertwining $t=0$ and $t=1$ in the superscripts.
 There are injections:
 \begin{align*}
  \mathbb{Z}\mathcal{C}(f^0)& \hookrightarrow \mathbb{Z}\mathcal{C}(f^1),\\
  G(\widetilde{f^0})&\hookrightarrow G(\widetilde{f^1}).
 \end{align*}

  Inverting $t=0$ and $t=1$ in the equation \ref{eqmatrice} above, we have an injection $G(\widetilde{f^0})\hookrightarrow G(\widetilde{f^1})$ given by the equation, for $j=k$, or $j=k+1$:
 \begin{equation}
 \label{eqmatrice2}
 M^1_j=\begin{pmatrix}
  &           & 0      &  &   & \\
  & I_{p^0_j} & \vdots &  & 0 & \\
 0& \hdots    & 1      & 0 & \hdots & 0 \\
  &           & 0      &  &   & \\ 
  &  N^0_j    & \vdots &  & I_{q^0_j} & \\
  &           & 0      &  &           & \\
 \end{pmatrix}, 
 \end{equation}
 where $N^1_j$ is the down-left submatrix of $M^1_j$.
 
  If the label is $-$, we have injections:
  \begin{align*}
  \mathbb{Z}\mathcal{C}(f^0)& \hookrightarrow \mathbb{Z}\mathcal{C}(f^1),\\
  G(\widetilde{f^0})&\hookrightarrow G(\widetilde{f^1}).
 \end{align*}
 We also have equations:
 \begin{equation}
 \label{eqmatrice2b}
 M^1_{j}=\begin{pmatrix}
   &                      &  &  &                     & & 0 \\
   & I_{p^0_{j}}          &  &  & 0                   & & \vdots \\
   &                      &  &  &                     & & 0 \\
   &                      &  &  &                     & & 0 \\
   & N^0_j                &  &  &  I_{q^0_j}          & & \vdots \\
   &                      &  &  &                     & & 0 \\
 0 &    \hdots            & 0& 0&  \hdots             &0& 1 \\
\end{pmatrix},
 \end{equation}
for $j=k$ or $j=k+1$ and $M^0_j=M^1_j$ in other degrees.
 \end{itemize}

\subsection{The main theorem} 
 \label{main}
 \subsubsection{Property (\texorpdfstring{$\mathcal{P}$}{Lg}) and statement of the theorem}
  In this subsection, we give and prove the main theorem of the article.
   We use the notation described in the previous subsection. 
  If $M$ is a matrix in $G(\tilde{f})$, the matrix $M_k$ will be its restriction to $G_k(\tilde{f})$.
   The matrix $N_k$ will denote the down-left submatrix of $M_k$.
   Our theorem and the proof of it will mainly focus on points of label $+$, but we show in Subsection \ref{opposite} that an equivalent theorem can be stated focusing on points of label $-$. 
 However, we show that a theorem using data about points of label $-$ would be strictly equivalent to Theorem \ref{th1}.
   
   Let $(\tilde{f},X)$ be a couple such that $\tilde{f}$ is a Morse germ, and $X$ is a Morse-Smale pseudo-gradient which is adapted to $f$.
   
   \begin{definition}[Property ($\mathcal{P}$)]
   \label{defP}
   We say that $(\tilde{f},X)$ has \emph{property $(\mathcal{P})$} when
   there is a matrix $M$ in $G(\tilde{f})$ such that:
  \begin{itemize}
  \item $ (M \partial M^{-1} )_{-+} =0$, that is, for all $k$ between $0$ and $n$
  \[ \partial_{-+,k} - \partial_{--,k}N_k + N_{k-1}\partial_{++,k} - N_{k-1}\partial_{+-,k} N_k = 0 . \]
  \item  ${ (\mathcal{C}^+_k(\tilde{f}),\partial_{++,k}+\partial_{+-,k}N_k)_{0\leq k \leq n} }$ is a chain complex. Its homology vanishes in degree $k<n$ and is $\mathbb{Z}$ in degree $n$.
  
   \end{itemize}  
\end{definition}

\begin{remark}
The first item implies that $(M\partial M^{-1})_{++}$ is a chain complex, as we get:
\[(M\partial M^{-1})_k=\begin{pmatrix}
(M\partial M^{-1})_{++,k} & (M\partial M^{-1})_{+-,k} \\
0                         & (M\partial M^{-1})_{--,k} \\
\end{pmatrix}\]
in all degree $k$.
Thus, we get $(M\partial M^{-1})_{++}^2=0$ as $(M\partial M^{-1})^2=0$. 
\end{remark}

The theorem is:
\begin{theorem}
\label{th1}
Let $\tilde{f}$ be a Morse germ along $\Sn$.
Let $X$ be an adapted pseudo-gradient which is Morse-Smale and let $\partial$ be the associated boundary operator.

  If the germ $\tilde{f}$ has a non-critical extension then $(\tilde{f},X)$ has property ($\mathcal{P}$).
\end{theorem}

\begin{remark}
Notice that if the germ extends non-critically, the theorem implies that $M\partial M^{-1}$ becomes the \emph{mapping cone} of the chain map 
\[\partial_{+-} : (\mathbb{Z}\mathcal{C}^+(\tilde{f}) , (M\partial M^{-1})_{++}) \to (\mathbb{Z}\mathcal{C}^-(\tilde{f}) , (M\partial M^{-1})_{--}).\]
The definition of the mapping cone of a map between chain complexes can be found in \cite[Sec. 1.5]{Wei}.
\end{remark}

\subsubsection{Taking the opposite}
\label{opposite}
In this subsection, we expose the algebraic relations between $\tilde{f}$ and $-\tilde{f}$.
Besides its own interest, it will simplify the proof of Theorem \ref{th1}.

The following lemma is in fact just a description of the homological algebra of $-\tilde{f}$ with respect to the one of $\tilde{f}$. If $\prec$ is an order on a set $\mathcal{S}$, then the opposite order $\prec^{op}$ is defined such that:
\[a\prec^{op}b \iff b \prec a.\]
We shall not give a proof to this lemma.
\begin{lemma}
\label{oppbound}
We have:
\begin{itemize}
\item for all $0\leq k \leq n$, if $a\in \mathcal{C}^+_k(\tilde{f})$ then $a\in \mathcal{C}^-_{n-k}(-\tilde{f})$. Thus $\mathcal{C}^+_k(\tilde{f})=\mathcal{C}^-_{n-k}(-\tilde{f})$,
\item in the same way, $\mathcal{C}^-_k(\tilde{f})=\mathcal{C}^+_{n-k}(-\tilde{f})$,
\item if $X$ is a Morse-Smale pseudo-gradient adapted to $f$ then $-X$ is a Morse-Smale pseudo-gradient adapted to $-f$.
 If we denote $\partial_k(f)=\begin{pmatrix}
\partial_{++,k}(\tilde{f}) & \partial_{+-,k}(\tilde{f}) \\
\partial_{-+,k}(\tilde{f}) & \partial_{--,k}(\tilde{f}) \\
\end{pmatrix}$ the boundary operator associated to $X$ where an order $\prec$ is given to the critical points of $f$, then, the  boundary operator associated to $-X$ is
$\partial_k(-f)=\begin{pmatrix}
{}^t \partial_{--,n-k+1}(\tilde{f}) & {}^t \partial_{+-,n-k+1}(\tilde{f}) \\
{}^t \partial_{-+,n-k+1}(\tilde{f}) & {}^t \partial_{++,n-k+1}(\tilde{f}) \\
\end{pmatrix}$, where the opposite order $\prec^{op}$ is given to the critical points of $-f$.
\end{itemize}

\end{lemma}  
Maybe it is worth noticing that there is no link between $\partial_{+-,n-k+1}(\tilde{f})$ and $\partial_{-+,k}(-\tilde{f})$.

From the previous lemma, we also have:
\begin{lemma}
\label{oppbirth}
If $t \mapsto f^t$ is a non-critical path of function, denote by $\widetilde{f^t}$ the germ whose representative is $(x,s)\mapsto f^{s+t}(x)$ for $x\in \Sn$ and $s\in [0,\varepsilon)$, the positive real number $\varepsilon$ being as small as wanted. We also consider a generic path $X^t$ of adapted pseudo-gradients.
We have:
\begin{itemize}
\item  a birth (resp. death) bifurcation of two points of label $+$ and index $k$ during the path $t\mapsto f^t$, correspond to a birth (resp. death) bifurcation of two points of label $-$ and index $n-k$ occurs during the path $t \mapsto -f^t$,
\item a handle slide of a point $a$ of label $+$ over a point $b$ of label $-$ occurs during the path $t \mapsto (f^t, X^t)$, corresponds to a handle slide of $b$ over $a$ during the path $t\mapsto (-f^t,-X^t)$.
\end{itemize}
\end{lemma}

In the definition of property ($\mathcal{P}$), we only take care of points of label $+$, and seem to forget the existence of points of label $-$.
The transcription of property $(\mathcal{P})$ for points of label $-$ would be:
\begin{definition}[Property ($\mathcal{P}-$)]
\label{P-}
We say that the couple $(\tilde{f},X)$ has property ($\mathcal{P}-$) when:
\begin{itemize}
\item For all $k$ between $0$ and $n$
  \[ \partial_{-+,k} - \partial_{--,k}N_k + N_{k-1}\partial_{++,k} - N_{k-1}\partial_{+-,k} N_k = 0 . \]

 \item  ${(\mathbb{Z}\mathcal{C}^-_k(\tilde{f}),\partial_{--,k}+N_{k-1}\partial_{+-,k})_{0\leq k \leq n} }$ is a chain complex and its homology vanishes in degree $k>0$ and is $\mathbb{Z}$ in degree $0$,
  
\end{itemize} 
\end{definition}
We have:
\begin{proposition}
\label{equivp-}
Property ($\mathcal{P}-$) is strictly equivalent to property $(\mathcal{P})$.
\end{proposition}
\begin{proof}
Notice that the first item is unchanged. 
The second item is implied by the definition of $(\mathcal{P})$ for label $+$ points. 
Indeed, if a germ $\tilde{f}$ has property $(\mathcal{P})$, we have a short exact sequence of chain complexes:
\[ 0 \to \mathbb{Z}\mathcal{C}^+(\tilde{f}) \to \mathbb{Z}\mathcal{C}(f) \to \mathbb{Z}\mathcal{C}^-(\tilde{f})\to 0.\]

Recall that the complex $\mathbb{Z}\mathcal{C}(f)$ has the homology of the sphere, that is, all homology groups vanish except in degree $0$ and $n$ for which it is $\mathbb{Z}$.
The long exact sequence in homology for this sequence then reduces to two non-trivial short exact sequences:
\[0 \to H_n(\mathbb{Z}\mathcal{C}^+(\tilde{f})) \to H_n(\mathbb{Z}\mathcal{C}(f)) \to H_n(\mathbb{Z}\mathcal{C}^-(\tilde{f}))  \to 0\]
and 
\[0 \to H_0(\mathbb{Z}\mathcal{C}^+(\tilde{f})) \to H_0(\mathbb{Z}\mathcal{C}(f)) \to H_0(\mathbb{Z}\mathcal{C}^-(\tilde{f}))  \to 0.\]
The other sequences for $k\neq 0$,or $k\neq n$ directly show that $H_k(\mathbb{Z}\mathcal{C}^-(\tilde{f}))=0$.
If $\tilde{f}$ has property $(\mathcal{P})$, the module $H_n(\mathbb{Z}\mathcal{C}^+(\tilde{f}))$ is $\mathbb{Z}$ and the module $H_0(\mathbb{Z}\mathcal{C}^+(\tilde{f}))$ vanishes.
Thus $\tilde{f}$ has property $(\mathcal{P}^-)$.
If $\tilde{f}$ has property $(\mathcal{P}^-)$, the module $H_n(\mathbb{Z}\mathcal{C}^-(\tilde{f}))$ vanishes and the module $H_0(\mathbb{Z}\mathcal{C}^-(\tilde{f}))$ is $\mathbb{Z}$.
Thus $\tilde{f}$ has property $(\mathcal{P})$.
We then see that the properties $(\mathcal{P})$ and $(\mathcal{P}^-)$ are equivalent.
\end{proof}

We also have the lemma:
\begin{lemma}
Let $\prec$ be an order on $\mathcal{C}_k(f)$.
We give $\mathcal{C}_{n-k}(f)$ the opposite order.
We have:
\[G_{k}(-\tilde{f})=\left\lbrace \begin{pmatrix}
I_{q_{n-k}} & 0 \\
{}^t N_{n-k} & I_{p_{n-k}} \\
\end{pmatrix}\textit{ such that } 
\begin{pmatrix}
I_{p_{n-k}} & 0 \\
 N_{n-k} & I_{q_{n-k}} \\
\end{pmatrix}\in G_{n-k}(\tilde{f})\right\rbrace.\]
\end{lemma}
\begin{proof}
We only need to notice that the two following items are equivalent:
\begin{itemize}
 \item $a$ is a critical point of $f$ of label $+$ and $b$ a critical point of $f$ of label $-$, both of index $k$ such that $f(b)>f(a)$,
 \item  $a$ is a critical point of $-f$ of label $-$ and $b$ a critical point of $-f$ of label $+$, both of index $n-k$ such that $-f(a)>-f(b)$ .
\end{itemize}

\end{proof}

Finally we state the not surprising proposition:
\begin{proposition}
$(\tilde{f},X)$ has property $(\mathcal{P})$ if and only if $(-\tilde{f},-X)$ has property $(\mathcal{P})$. 
\end{proposition}
\begin{proof}
\label{takingopp}
Let $(\tilde{f},X)$ be a couple of a Morse germ and a Morse-Smale adapted pseudo-gradient which has property $(\mathcal{P})$.
Denote by $\partial (\tilde{f})$ the boundary operator associated to $X$. 
Denote by $\partial(-\tilde{f})$ the boundary operator associated to $-X$.
We first show that $(-\tilde{f},-X)$ has property $(\mathcal{P}-)$.
Let $M\in G(\tilde{f})$ such that $(M\partial M^{-1})_{-+}=0$.
Then, for all $k$ between $0$ and $n$ we get:
\[\partial_{-+,k}(\tilde{f}) - \partial_{--,k}(\tilde{f})N_k + N_{k-1}\partial_{++,k}(\tilde{f}) - N_{k-1}\partial_{+-,k}(\tilde{f}) N_k = 0 .\]
Taking the transpose, we get, with Lemma \ref{oppbound}:
\[\partial_{-+,n-k+1}(-\tilde{f}) - {}^t N_k \partial_{++,n-k+1}(-\tilde{f})+ \partial_{--,n-k+1}(-\tilde{f}){}^t N_{k-1} - {}^t N_{k}\partial_{+-,n-k+1}(-\tilde{f}) {}^t N_{k-1}=0.\]
Taking $M^{op}$ to be such that in degree $k$ 
\[ M^{op}_k:=\begin{pmatrix}
I_{q_{n-k}} & 0 \\
-{}^t N_{n-k} & I_{p_{n-k}} \\
\end{pmatrix},\]
we see that $M^{op} \in G(-\tilde{f})$ and that $(M^{op}\partial(-\tilde{f})(M^{op})^{-1})_{-+}=0$.
Moreover, we have that ${}^t(M\partial(\tilde{f})M^{-1})_{++,k}=(M^{op}\partial(-\tilde{f})(M^{op})^{-1})_{--,n-k+1}$.
Thus, inverting the arrows in the chain complex:
\[0\to \mathbb{Z}\mathcal{C}^+_n(\tilde{f})\to ... \to \mathbb{Z}\mathcal{C}^+_0(\tilde{f})\to 0,\]
where the boundary operator is  $(M\partial(\tilde{f})M^{-1})_{++,k}$,
we get the chain complex:
\[0\leftarrow \mathbb{Z}\mathcal{C}^-_0(-\tilde{f})\leftarrow ... \leftarrow \mathbb{Z}\mathcal{C}^-_n(-\tilde{f})\leftarrow 0,\]
where the boundary operator is $(M^{op}\partial(-\tilde{f})(M^{op})^{-1})_{--,k}$.
Denote  $(M\partial(\tilde{f})M^{-1})_{++}$ by $\partial'$ and $(M^{op}\partial(-\tilde{f})(M^{op})^{-1})_{--}$ by $\partial''$.
In homology we have:
\begin{align*}
H_k(\mathbb{Z}\mathcal{C}^-(-\tilde{f})) = & \ker (\partial''_k)/\Ima (\partial''_{k+1}) \\
					= & \coker (\partial'_{n-k+1}) / \coIma (\partial'_{n-k}) \\
					\simeq & H_{n-k}(\mathbb{Z}\mathcal{C}^+(\tilde{f})). \\
\end{align*}
We then have that: \[H_k((M^{op}\partial(-\tilde{f})(M^{op})^{-1})_{--})=H_{n-k}((M\partial(\tilde{f})M^{-1})_{++}).\]
Thus, $(-\tilde{f},-X)$ has property $(\mathcal{P}-)$ and by Proposition \ref{equivp-}, the couple has property $(\mathcal{P})$.
\end{proof}
\subsubsection{Proof of Theorem \ref{th1}}

The next lemma shows that property ($\mathcal{P}$) is a property only depending on the germ $\tilde{f}$ and not on the pseudo-gradient adapted to $f$:

\begin{lemma}[Invariance of property $(\mathcal{P})$]
\label{invariance}
Let $\tilde{f}$ be a Morse germ.
Assume there is a Morse-Smale pseudo-gradient $X^0$ associated to $f$ such that the couple $(\tilde{f},X^0 )$ has property $(\mathcal{P})$.
Then, any other Morse-Smale pseudo-gradient $X^1$ adapted to $f$ has property $(\mathcal{P})$.
\end{lemma}

For all index $k$, we choose the following natural order on points in $\mathcal{C}_k(f)$:
\begin{center}
$a\prec b$ if and only if the label of $a$ is $+$ and the label of $b$ is $-$, and, when the points have same label, $f(a)>f(b)$.
\end{center}
With this order, the first point of $\mathcal{C}_k(f)$ is the point of label $+$ with highest critical value and the last point is the point of label $-$ with lowest critical value.
We have the lemma:
\begin{lemma}
\label{mult}
For all $k$, let $T_k$ be a lower triangular matrix of size $p_k \times p_k$.
  Let $M$ be in $G(\tilde{f})$.
  Then the matrix $M'$ which is defined in degree $k$ by the equation
  \[\begin{pmatrix}
  I_{p_k} & 0 \\
  N_k T_k & 0 \\
  \end{pmatrix}\]
  is in $G(\tilde{f})$.
\end{lemma}
\begin{proof}[Proof of Lemma \ref{mult}]
We need to prove that the nullity condition imposed on the coefficients by the definition of $G(\tilde{f})$ is respected.
If the $l$-th critical point of index $k$ and label $+$ is below the $i$-th critical point of index $k$ and label $-$, so is the $l'$-th critical point of index $k$ and label $+$.
Thus, with the given order on the set of critical points, if $(N_k)_{(i,l)}=0$, then $(N_k)_{(i,l')}=0$ for $l'>l$, as the order is the one of the critical values.
A straightforward computation then shows that if $(N_k)_{(i,l)}=0$, then $(N_k T_k)_{(i,l)}=0$ also, when $T$ is lower triangular.
Then $M' \in G(\tilde{f})$.
\end{proof}

\begin{proof}[Proof of Lemma \ref{invariance}]
Let $t\mapsto X^t$ be a generic path of pseudo-gradients adapted to $f$ between $X^0$ and $X^1$.

To prove the lemma, it is sufficient to consider a path with only one bifurcation.
Then, we assume that there is one and only one handle slide occurring during the path $t\mapsto X^t$.
As $(\tilde{f},X^0)$ has property ($\mathcal{P}$), there is $M\in G(\tilde{f})$ such that $M\partial^0 M^{-1}$ has the properties described in Definition \ref{defP}. 
The proof consists in finding an element $M'\in G(\tilde{f})$ such that conjugating the boundary operator by $M'$ undoes the effect of the handle slide. 
 We can restrict to the following cases:
\begin{itemize}
\item \textbf{Handle slide between two points of label $+$} 
Assume there is a handle slide of a point $a$ of index $k$ over a point $b$ of the same index.
 Both have label $+$.
 The modification of the boundary operator is given by:
 \[\partial^1 = U \partial^0 U^{-1}, \]
 where $U_j = I_{p_j+q_j}$ if $j \neq k$ and 
 \[U_k:= \begin{pmatrix}
 T & 0 \\
 0   & I_{q_k} \\
\end{pmatrix} , \]
where $T$ is some lower triangular matrix of type $I_{p_k}+E_{i,l}$, with $l<i$.

Let $M' = U M U^{-1}$.
We show that $M' \in G(\tilde{f})$.
We have that $M'_j=M_j$ for all $j\neq k$.
In degree $k$, we have after a simple computation:
\[M'_k= \begin{pmatrix}
I_{p_k}    & 0       \\
N_k T^{-1} & I_{q_k} \\
\end{pmatrix}.\]
As $T$ is lower triangular, $T^{-1}$ is also lower triangular, and by Lemma \ref{mult}, we have that $M'\in G(\tilde{f})$.

We have 
\[M' \partial^1 (M')^{-1} = U (M \partial^0 M^{-1}) U^{-1}. \]
As $U$ acts isomorphically on $\mathbb{Z}\mathcal{C}^+(\tilde{f})$ and acts as the identity on $\mathbb{Z}\mathcal{C}^-(\tilde{f})$, it implies that $(M' \partial^1 (M')^{-1})_{-+}=0$, and that $(M' \partial^1 (M')^{-1})_{++}$ and $(M \partial^0 M^{-1})_{++}$ have same homology.
This is what is required.

\item \textbf{Handle slide between two points of label $-$.}
Consider the path $t\mapsto - X^t$.
From Proposition \ref{takingopp}, we know that $(-\tilde{f},-X^0)$ has property $(\mathcal{P})$.
From Lemma \ref{oppbirth}, a handle slide between critical points of $f$ of label $-$ during $t\mapsto X^t$ corresponds to a handle slide between critical points of $-f$ of label $+$ during $t \mapsto -X^t$.
From what we did just above, we have that the couple $(-\tilde{f},-X^1)$ has property $(\mathcal{P})$. From Proposition \ref{takingopp}, we have that the couple $(\tilde{f},X^1)$ has property $(\mathcal{P})$.

\item \textbf{Handle slide of a point of label $+$ over a point of label $-$.} 
The pseudo-gradient $\partial^1$ is given from $\partial^0$ by some equation:
\[\partial^1= M' \partial^0 (M')^{-1},\]
where $M'$ is the matrix corresponding to the handle slide. 
From the very definition of $G(\tilde{f})$, the matrix $M'$ is then an element of $G(\tilde{f})$.
We know that there is a matrix $M$ such that $M\partial^0 M^{-1}$ is as wanted.
We can then take $M^1 = M(M')^{-1}$ to be the matrix in $G(\tilde{f})$ leading to property $(\mathcal{P})$ for $X^1$.
Indeed, we get:
\begin{align*}
M^1 \partial^1 (M^1)^{-1} & = (M (M')^{-1})M' \partial^0 (M')^{-1} (M' M^{-1}) \\
						& = M \partial^0 M^{-1}   \\
\end{align*}
This is what is required.

\item \textbf{Handle slide of a point of label $-$ over a point of label $+$.} 
This is the most tedious case.
We denote by $k$ the index of the two points. 
We call $a$ the point of label $-$ involved in the handle slide and $b$ the point of label $+$.
We have $f(a)>f(b)$. 
We assume that $\partial^1$ is obtained from $\partial^0$ by the equation $\partial^1=U \partial^0 U^{-1}$, where $U_j=I_{p_j+q_j}$ for all $j\neq k$ and 
\[U_k:= \begin{pmatrix}
I_{p_k} & E \\
0       & I_{q_k}
\end{pmatrix}\]
where $E$ is a matrix with only one non-zero coefficient $E_{i,l}$ such that $a$ is the $l$-th point of $\mathcal{C}^-_k(\tilde{f})$ and $b$ is the $i$-th point of $\mathcal{C}^+_k(\tilde{f})$.
Let $M \in G(\tilde{f})$ such that $(M\partial^0 M^{-1})_{-+}=0$ and $(M \partial^{0}M^{-1})_{++}$ is a boundary operator with the desired homology.
Notice that the matrix $EN_k$ is a lower triangular matrix in $\mathcal{M}_{p_k}(\mathbb{Z})$ whose diagonal coefficients are all $0$. 
The matrix $I_{p_k}-EN_k$ is then invertible and lower triangular.

Let $M'$ be the matrix such that $M'_j=M_j$ in every degree except in degree $k$ for which it is:
\[ M'_k:= \begin{pmatrix}
I_{p_k} & 0 \\
 N_k(I_{p_k}-EN_k)^{-1}  & I_{q_k} \\ 
\end{pmatrix}. \]
By Lemma \ref{mult}, it is in $G(\tilde{f})$.
Consider the boundary operator $\partial'$ given by  
\[\partial' := M' \partial^1 (M')^{-1}.\]

We prove that $\partial'_{-+}=0$, which implies that $\partial'_{++}$ is a boundary operator, and that $\partial'_{++}$ has same homology as $M\partial^0 M^{-1}$.
As $\partial^1_j\neq \partial^0_j$ only if $j\in\{k,k+1\}$, and that $M'_j=M_j$ in these degrees, we only need to look at what happens in degree $k$ and $k+1$.
It gives straightforward --- but tedious --- computations that the reader may skip.
\begin{enumerate}
\item We begin in degree $k$, the letter $I$ denotes the identity matrix, whatever the rank:
\begin{align*}
\partial'_k & =  \begin{pmatrix}
I & 0 \\
N_{k-1} & I \\
\end{pmatrix}
\begin{pmatrix}
\partial^0_{++,k} & \partial^0_{+-,k} \\
\partial^0_{-+,k} & \partial^0_{--,k} \\
\end{pmatrix}
\begin{pmatrix}
I & -E \\
0 & I \\
\end{pmatrix}
\begin{pmatrix}
I & 0 \\
-N_k (I-EN_{k})^{-1} & I \\
\end{pmatrix} \\
 & =  \begin{pmatrix}
I & 0 \\
N_{k-1} & I \\
\end{pmatrix}
\begin{pmatrix}
\partial^0_{++,k} & \partial^0_{+-,k}-\partial^0_{++,k}E \\
\partial^0_{-+,k} & \partial^0_{--,k}- \partial^0_{-+,k}E\\
\end{pmatrix}
\begin{pmatrix}
I & 0 \\
-N_k (I-EN_{k})^{-1} & I \\
\end{pmatrix} \\
 & = \begin{pmatrix}
\partial^0_{++,k} & \partial^0_{+-,k}-\partial^0_{++,k}E \\
N_{k-1}\partial^0_{++,k}+\partial^0_{-+,k} & N_{k-1}(\partial^0_{+-,k}-\partial^0_{++,k}E)+\partial^0_{--,k}- \partial^0_{-+,k}E\\
\end{pmatrix} \\
& \times
\begin{pmatrix}
I & 0 \\
-N_k (I-EN_{k})^{-1} & I \\
\end{pmatrix}\\
& = \begin{pmatrix}
\star_1 & \star_3 \\
\star_2 & \star_4 \\
\end{pmatrix}
\end{align*}
From definition of property ($\mathcal{P}$), we are only interested in computing $\star_1$ and $\star_2$.
For the sake of clarity, we will not indicate the superscript $0$ or the subscript $k$ in the matrices $\partial^0_{\ell_1 \ell_2,k}$.
We have:
\begin{align*}
\star_1 & = \partial_{++} - (\partial_{+-}-\partial_{++}E)N_k(I-EN_k)^{-1}\\
        & = \partial_{++}(I-EN_k)(I-EN_k)^{-1}+\partial_{++}EN_k(I-EN_k)^{-1} \\
        & - \partial_{+-}N_k(I-EN_k)^{-1} \\
        & = (\partial_{++}-\partial_{+-}N_k)(I-EN_k)^{-1}  
\end{align*}
And:
\begin{align*}
\star_2 & = N_{k-1}\partial_{++}- \left[ N_{k-1}(\partial_{+-}-\partial_{++}E)+\partial_{--}-\partial_{-+}E\right]N_k(I-EN_k)^{-1} \\
       & = N_{k-1}\partial_{++}\left[I+EN_k(I-EN_k)^{-1}\right] + \partial_{-+}\left[I+EN_k(I-EN_k)^{-1}\right] \\
       & - N_{k-1}\partial_{+-}N_k(I-EN_k)^{-1} - \partial_{--}N_k(I-EN_k)^{-1} \\
\end{align*}
As $EN_k(I-EN_k)^{-1}=(I-EN_k)^{-1}-I$, we get:
\begin{align*}
\star_2 & = \left(N_{k-1}\partial_{++}+\partial_{-+}-N_{k-1}\partial_{+-}N_k-\partial_{--}N_k\right)(I-EN_k)^{-1}
\end{align*}
But, since $(M\partial^0 M^{-1})_{-+}=0$, from Equation \ref{eq-+} we get that $\star_2=0$, which is what we want.

\item We have, in degree $k+1$:
\begin{align*}
\partial'_{k+1} & = 
\begin{pmatrix}
I & 0 \\
N_k (I-EN_{k})^{-1} & I \\
\end{pmatrix}
\begin{pmatrix}
I & E \\
0 & I \\
\end{pmatrix}
\begin{pmatrix}
\partial^0_{++,k+1} & \partial^0_{+-,k+1} \\
\partial^0_{-+,k+1} & \partial^0_{--,k+1} \\
\end{pmatrix}
 \begin{pmatrix}
I & 0 \\
-N_{k+1} & I \\
\end{pmatrix}\\
& = 
\begin{pmatrix}
I & 0 \\
N_k (I-EN_{k})^{-1} & I \\
\end{pmatrix}
\begin{pmatrix}
I & E \\
0 & I \\
\end{pmatrix}
\begin{pmatrix}
\partial^0_{++,k+1}-\partial^0_{+-,k+1}N_{k+1} & \partial^0_{+-,k+1} \\
\partial^0_{-+,k+1}-\partial^0_{--,k+1}N_{k+1} & \partial^0_{--,k+1} \\
\end{pmatrix}\\
&= 
\begin{pmatrix}
I & 0 \\
N_k (I-EN_{k})^{-1} & I \\
\end{pmatrix}\\
& \times
\begin{pmatrix}
\partial^0_{++,k+1}-\partial^0_{+-,k+1}N_{k+1}+E(\partial^0_{-+,k+1}-\partial^0_{--,k+1}N_{k+1}) & \partial^0_{+-,k+1}+E\partial^0_{--,k+1} \\
\partial^0_{-+,k+1}-\partial^0_{--,k+1}N_{k+1} & \partial^0_{--,k+1} \\
\end{pmatrix}\\
& = 
\begin{pmatrix}
\star_1 & \star_3 \\
\star_2 & \star_4 \\
\end{pmatrix}
\end{align*}
 As in degree $k$, we will not indicate the superscript $0$ and the subscript $k+1$ when writing the matrices $\partial^0_{\ell_1 \ell_2,k+1}$.
 We have:
 \begin{align*}
 \star_1 & = \partial_{++}-\partial_{+-}N_{k+1}+E(\partial_{-+}-\partial_{--}N_{k+1})
 \end{align*}
 From Equation \ref{eq-+}, we get:
 \begin{align*}
  \star_1 & = \partial_{++}-\partial_{+-}N_{k+1}+E(N_k\partial_{+-}N_{k+1}+N_{k}\partial_{++})\\
          & = (I-EN_k)(\partial_{++}-\partial_{+-}N_{k+1})
 \end{align*}
 And:
 \begin{align*}
 \star_2 & = \partial_{-+}-\partial_{+-}N_{k+1}+N_k(I-EN_k)^{-1}\left[ \partial_{++}-\partial_{+-}N_{k+1}+E(\partial_{-+}-\partial_{--}N_{k+1})\right] 
 \end{align*}
 Here again, using Equation \ref{eq-+} we get:
 \begin{align*}
 \star_2 & = N_k\partial_{+-}N_{k+1}+N_{k}\partial_{++}+N_k(I-EN_k)^{-1}\left[ \partial_{++}-\partial_{+-}N_{k+1}+E(N_k\partial_{+-}N_{k+1}+N_{k}\partial_{++})\right] \\
 & = -N_k(\partial_{++}-\partial_{+-}N_{k+1})+N_k(I-EN_k)^{-1}(I-EN_k)(\partial_{++}-\partial_{+-}N_{k+1})\\
 & = 0
 \end{align*} 
\end{enumerate}

We then have $\partial'_{-+}=0$.
Denote by $U'$ to be the matrix such that $U'_j$ is the identity in all degree except in degree $k$ for which we have:
\[U'_k:=\begin{pmatrix}
I_{p_k}-EN_k & 0 \\
0            & I_{q_k} \\
\end{pmatrix}.\]
It is an invertible matrix, acting non-trivially only on $\mathbb{Z}\mathcal{C}^+(\tilde{f})$.
We notice that we have proved:
\[\partial'_{++}= \left[ U' (M\partial^0 M^{-1})(U')^{-1}\right]_{++}. \]
Thus, $\partial'_{++}$ and  $(M\partial^0 M^{-1})_{++}$ have same homology, which is what is required.
With this last case, the proof is complete.
\end{itemize}
\end{proof}

The previous lemma shows that property $(\mathcal{P})$ only depends on $\tilde{f}$ and not on the adapted pseudo-gradient adapted to $f$.
We shall then say that the germ $\tilde{f}$ has property $(\mathcal{P})$ without specifying the adapted pseudo-gradient.
To prove Theorem \ref{th1}, if $\tilde{f}$ extends non-critically, we just have to find one Morse-Smale pseudo-gradient adapted to $f$ which has property $(\mathcal{P})$.
 
If $\tilde{h}$ is trivial, such a pseudo-gradient adapted to $h$ exists in a trivial way.
To prove Theorem \ref{th1}, we will build a pseudo-gradient adapted to $f$ from a pseudo-gradient adapted to a trivial germ by an induction on the number of bifurcations occurring during the generic path of functions linking $\tilde{f}$ to a trivial germ $\tilde{h}$.

Recall that time dependency is written as a superscript.
We have:
\begin{lemma}
\label{induction}
Let $(f^t)_{t\in [0,1+\varepsilon)}$ be a non critical generic path of functions extending a germ $\tilde{f}^0$. 
Assume there is one and only one bifurcation occurring during this path.
We also consider a path of adapted pseudo-gradients $X^t$ such that if the bifurcation is a birth or a death bifurcation, then this bifurcation is independent.
Let $\tilde{f}^1$ be the germ represented by the function 
\[(x,t)\in \Sn \times [1,1+\varepsilon) \mapsto f^t(x).\]
If $\widetilde{f^1}$ has property $(\mathcal{P})$, then $\widetilde{f^0}$ does too. 
\end{lemma}

\begin{proof}[Proof of the lemma]

We will prove the lemma with respect to the kind of bifurcation occurring.
We separate in the proof the death/birth of a pair of index $(1,0)$ or $(n,n-1)$ from the other death/birth bifurcations, as we saw that the notion of "being independent" is different in these cases.

A detailed proof would require some matricial computation which are not difficult, but technical as there is a lot of notation. 
For the sake of clarity, we will only do in the following proof some of the most complex computations, and leave the easier ones to the reader.

\begin{itemize}
\item \textbf{Crossing of critical points.}
We have that $G(\widetilde{f^1})\subset G(\widetilde{f^0})$, and from \cite[Corollary 2.2]{Laud5} we can even give $f^0$ and $f^1$ the same Morse-Smale pseudo-gradient $X$.
Identifying $\mathbb{Z}\mathcal{C}(\widetilde{f^1})$ with $\mathbb{Z}\mathcal{C}(\widetilde{f^0})$, we can then assume that $\partial^1=\partial^0$.
We know there is $M^1 \in G(\widetilde{f^1})$ such that $M^1\partial^1(M^1)^{-1}$ fulfills conditions of property ($\mathcal{P}$).
As $M^1 \in G(\widetilde{f^0})$, we have that $M^1 \partial^0 (M^1)^{-1}$ also fulfills property ($\mathcal{P}$). 
 From Lemma \ref{invariance}, property $(\mathcal{P})$ only depends on the germ and not on the pseudo-gradient.
 Thus $\widetilde{f^1}$ also has property $(\mathcal{P})$.

\item \textbf{Death bifurcation of a pair of index $(k+1,k)$, with $1 \leq  k \leq n-2$.} 

We denote by $(a,b)$ the pair of critical points of index $(k+1,k)$ dying during the path.
We saw that we can consider a pseudo-gradient $X^0$ such that $\partial^{1}=\partial^0$ on the points that still exist at time $1$, and such that $\partial^0 a= \pm b$ and $\partial^0 b =0$. 
We use here that this death bifurcation is independent.
Let $M^1$ in $G(\widetilde{f^1})$ such that $M^1 \partial^1 (M^1)^{-1}$ fulfills the hypotheses of the theorem.

Assume that the label of the pair is $+$.
In degree $k+1$, we have, from equation \ref{eqmatrice}:
\[\partial^0_{k+1}=\begin{pmatrix}
   &                      &  & 0      &  &                     &  \\
   & \partial^1_{++,k+1}  &  & \vdots &  & \partial^1_{+-,k+1} &  \\
   &                      &  & 0      &  &                     &  \\
 0 & \hdots               & 0 & \pm 1     & 0& \hdots              & 0 \\
   &                      &  & 0      &  &                     &  \\
   & \partial^1_{-+,k+1}   &  & \vdots &  &  \partial^1_{--,k+1}&  \\
   &                      &   & 0      & &                      &  \\
\end{pmatrix},\]
where the column with a lot of $0$ is the one of $a$ and the row with a lot of $0$ is the one of $b$.
As the bifurcation is independent, we also have in degree $k+2$:
\[\partial^0_{k+2}=\begin{pmatrix}
   &                      &   &  &                     &  \\
   & \partial^1_{++,k+2}  &   &  & \partial^1_{+-,k+2} &  \\
   &                      &   &  &                     &  \\
 0 & \hdots               & 0 & 0& \hdots              & 0 \\
   &                      &   &  &                     &  \\
   & \partial^1_{-+,k+2}  &   &  &  \partial^1_{--,k+2}&  \\
   &                      &   &  &                     &  \\
\end{pmatrix},\]
and in degree $k$:
\[\partial^0_{k}=\begin{pmatrix}
   &                      &  & 0      &  &                     &  \\
   & \partial^1_{++,k}  &  & \vdots &  & \partial^1_{+-,k} &  \\
   &                      &  & 0      &  &                     &  \\
   &                      &  & 0      &  &                     &  \\
   & \partial^1_{-+,k}   &  & \vdots &  &  \partial^1_{--,k}&  \\
   &                      &   & 0      & &                      &  \\
\end{pmatrix}.\]
In other degrees, the boundary operator is unchanged.
We show that $\widetilde{f^0}$ has property ($\mathcal{P}$) by showing it for $\partial^0$.
We verify the two points of the definition.

 We saw that there are natural injections of $G(\widetilde{f^1})$ in $G(\widetilde{f^0})$, and we denote by $M^0$ the image of $M^1$ by this injection, which is, in degree $j=k$ or $j=k+1$:
\[M^0_{j}=\begin{pmatrix}
   &                      &   & 0      &  &                     &   \\
   & I_{p^1_{j}}          &   & \vdots &  & 0                   &   \\
   &                      &   & 0      &  &                     &   \\
 0 & \hdots               & 0 & 1      & 0& \hdots              & 0 \\
   &                      &   & 0      &  &                     &   \\
   & N^1_{j}              &   & \vdots &  &  I_{q^1_{j}}        &   \\
   &                      &   & 0      &  &                     &   \\
\end{pmatrix},\]
where $N^1_j$ is the down-left submatrix of $M^1_j$.
We have $M^0_j=M^1_j$ in other degrees.
 
It leads to 
\[ (M^0\partial^0(M^0)^{-1})_{-+,j}=(M^1\partial^1(M^1)^{-1})_{-+,j}=0\]
and
\[(M^0\partial^0(M^0)^{-1})_{++,j}=(M^1\partial^1(M^1)^{-1})_{++,j}\]
for all $j\notin\{k,k+1,k+2\}$.
If $j=k+1$, from the equations above, a simple computation shows that \[ (M^0\partial^0(M^0)^{-1})_{-+,k+1}=0,\]
and 
\[ (M^0\partial^0(M^0)^{-1})_{++,k+1}=\begin{pmatrix}
 &                                    &  & 0 \\
 & (M^1\partial^1(M^1)^{-1})_{++,k+1} &  & \vdots \\
 &                                    &  & 0 \\
0& \hdots                             & 0& \pm 1 \\  
\end{pmatrix}.\]
The same kind of matricial computations show that 
\[ (M^0\partial^0(M^0)^{-1})_{-+,k}=0,\]
and
\[ (M^0\partial^0(M^0)^{-1})_{-+,k+2}=0.\]

Thus, we get a short exact sequence of chain complexes:
\[0\to \left(\mathbb{Z}\mathcal{C}^+(\widetilde{f^1}), (M^1 \partial^1 (M^1)^{-1})_{++}\right) \to \left(\mathbb{Z}\mathcal{C}^+(\widetilde{f^0}), (M^0 \partial^0 (M^0)^{-1})_{++}\right) \to (T, \pm 1)  \to 0,\]
where $(T, \pm 1 )$ is the acyclic chain complex: $0 \to \mathbb{Z}a \to \mathbb{Z}b \to 0$ such that $\partial_T a =\pm b$.
Using the long exact sequence in homology, we have that $(M^1 \partial^1 (M^1)^{-1})_{++}$ and $(M^1 \partial^1 (M^1)^{-1})_{++}$ have same homology.

Thus, the lemma is proved in the case of the death singularity of a pair $(a,b)$ of index $(k+1,k)$ and label $+$, with $1\leq k \leq n-2$.

If the label of the pair is $-$, then it corresponds to a death bifurcation of a pair of label $+$ for the path $t\mapsto -f^t$. 
We can then use results of Subsection \ref{opposite}.

\item \textbf{Birth bifurcation of a pair of index $(k+1,k)$, with $1 \leq k \leq n-2$.}

We call $(a,b)$ the pair of index $(k+1,k)$ which is born during the path.
The birth bifurcation is assumed independent for the path of pseudo-gradients $X^t$.
We have $\partial^1 a  = \pm b$, and the components of $\partial^1 d$ along $a$ or $b$ is $0$, for all critical point $d$ which is not $a$.
By assumption, there is a matrix $M^1$ in $G(\widetilde{f^1})$ such that \[\partial':= M^1 \partial^1 (M^1)^{-1}\] is a boundary operator with the properties displayed in the definition of property $(\mathcal{P})$.
Assume the label of $a$ and $b$ is $+$.

In matricial notation we have, inverting time in the equations arising for a death bifurcation:
\[\partial^1_{k+1}=\begin{pmatrix}
   &                      &  & 0      &  &                     &  \\
   & \partial^0_{++,k+1}  &  & \vdots &  & \partial^0_{+-,k+1} &  \\
   &                      &  & 0      &  &                     &  \\
 0 & \hdots               & 0 & \pm 1     & 0& \hdots              & 0 \\
   &                      &  & 0      &  &                     &  \\
   & \partial^0_{-+,k+1}   &  & \vdots &  &  \partial^0_{--,k+1}&  \\
   &                      &   & 0      & &                      &  \\
\end{pmatrix},\]
where the column with a lot of $0$ is the one of $a$ and the row with the lot of $0$ is the one of $b$.
As the bifurcation is independent, we also have in degree $k+2$:
\[\partial^1_{k+2}=\begin{pmatrix}
   &                      &   &  &                     &  \\
   & \partial^0_{++,k+2}  &   &  & \partial^0_{+-,k+2} &  \\
   &                      &   &  &                     &  \\
 0 & \hdots               & 0 & 0& \hdots              & 0 \\
   &                      &   &  &                     &  \\
   & \partial^0_{-+,k+2}  &   &  &  \partial^0_{--,k+2}&  \\
   &                      &   &  &                     &  \\
\end{pmatrix},\]
and in degree $k$:
\[\partial^1_{k}=\begin{pmatrix}
   &                      &  & 0      &  &                     &  \\
   & \partial^0_{++,k}    &  & \vdots &  & \partial^0_{+-,k} &  \\
   &                      &  & 0      &  &                     &  \\
   &                      &  & 0      &  &                     &  \\
   & \partial^0_{-+,k}    &  & \vdots &  &  \partial^0_{--,k}&  \\
   &                      &  &  0     &  &                      &  \\
\end{pmatrix}.\]
In other degrees, the boundary operator is unchanged.
Let us verify the items of the definition of property ($\mathcal{P}$).

 Let $M^0$ be the restriction of $M^1$ to $\mathbb{Z}\mathcal{C}(f^0)$.
That is, if for $j=k$ (resp. $j=k+1$)
\[M^1_j= \begin{pmatrix}
  &           &   & 0      &  &   & \\
  & I_{p^0_j} &   & \vdots &  & 0 & \\
  &           &   & 0      &  &   & \\
 0& \hdots    & 0 & 1      & 0 & \hdots  & 0 \\
  &  N_j      &   & L_j    &  & I_{q^0_j} & \\
 
 \end{pmatrix}, \]
  where $L_j$ is the column of the down-left submatrix of $M^1_j$ corresponding to $b$ (resp. $a$), then
\[M^0_j= \begin{pmatrix}
  I_{p^0_j} & & 0  \\
  N_j       & & I_{q^0_j}  \\
 
 \end{pmatrix}.\] 
It is an element of  $G(\widetilde{f^0})$.

In degree $k+1$ we have:
\begin{equation}
\label{eq-+}
(M^1\partial^1(M^1)^{-1})_{-+,k+1}=  
\begin{pmatrix}
(M^0\partial^0 (M^0)^{-1})_{-+,k+1} &  &  L_k-N_k \partial^0_{+-}L_{k+1} - \partial^0_{--,k+1}L_{k+1} \\
\end{pmatrix},
\end{equation}
where we have:
\[(M^0\partial^0 (M^0)^{-1})_{-+,k+1} = \partial^0_{-+,k+1} - \partial^0_{--,k+1}N_{k+1} - N_k\partial^0_{+-,k+1}N_{k+1} + N_k\partial^0_{--,k+1}\]

As ${(M^1\partial^1(M^1)^{-1})_{-+,k+1}=0}$, we easily see from Equation \ref{eq-+} and the definition of $M^0$ above that ${(M^0\partial^0(M^0)^{-1})_{-+,k+1}=0}$.
In degrees $k$ or $k+2$, we get the same kind of computation, which may seem quite technical, but are simple.
The fact that {$(M^0\partial^0(M^0)^{-1})_{-+,j}=0$} in degree $j$ different from $k$, $k+1$ and $k+2$ is immediate, as $M^0_j=M^1_j$.

  We get, using the descriptions of the matrices above:
\begin{equation}
\label{eq++}
(M^1\partial^1(M^1)^{-1})_{++,k+1}=\begin{pmatrix} 
  & \partial^0_{++,k+1}-\partial^0_{+-,k+1}N_{k+1} &   & &-\partial^0_{+-,k+1}L_{k+1} \\
0 & \hdots                                         & 0 & & 1  \\ 
\end{pmatrix}.
\end{equation}

We also have:
\begin{equation}
\label{eq+-}
(M^1\partial^1(M^1)^{-1})_{+-,k+1}=
\begin{pmatrix}
   & \partial^0_{+-,k+1} &   \\
 0 & \hdots            & 0 \\
\end{pmatrix}
\end{equation}
\begin{equation}
\label{eq--}
(M^1\partial^1(M^1)^{-1})_{--,k+1}= \partial^0_{--,k+1}+N_k\partial^0_{+-,k+1}.
\end{equation}

Thus, $\partial'^0_{++}:=\left( M^0 \partial^0 (M^0)^{-1} \right)_{++}$ defines a boundary operator.
For $t=0$ or $t=1$, denote by $\mathbb{Z}\mathcal{C}^t$ the acyclic chain complex with boundary operator $\partial'^t_{++}$:
\[0\to \mathbb{Z}\mathcal{C}^+_{n}(\widetilde{f^t}) \to \mathbb{Z}\mathcal{C}^+_{n-1}(\widetilde{f^t})\to...\to \mathbb{Z}\mathcal{C}^+_{1}(\widetilde{f^t})\to \mathbb{Z}\mathcal{C}^+_{0}(\widetilde{f^t})\to 0  .\]
We get a short exact sequence of chain complexes:
\[ 0 \rightarrow \mathbb{Z}\mathcal{C}^0 \hookrightarrow \mathbb{Z}\mathcal{C}^1 \rightarrow T \rightarrow 0, \]
where $(T,1)$ is the acyclic chain complex : \[ 0 \to \mathbb{Z} a \overset{1}{\to} \mathbb{Z} b  \to 0.\]
Using the long exact sequence in homology, we then show that $\mathcal{C}^0$ is acyclic. 

 Here again, proving the result for a pair of label $-$ simply uses results of Subsection \ref{opposite}.
 
\item \textbf{Death of a pair of index $(1,0)$ or $(n,n-1)$.}

We assume that the label of the pair $(a,b)$ of index $(k+1,k)$ which dies at the bifurcation is $+$.
Assume first that $k=0$.
In this case, there is a pseudo-gradient $\partial^0$ such that $\partial^0 d$ has no component along $b$ (resp. $a$) for any point $d$ of index $2$ (resp. $1$) and such that $\partial^0 a = \pm b \pm c$ where $c$ is another point of index $0$.

We verify the points of the definition of property ($\mathcal{P}$).
   We define $M^0$ to be as in Equation \ref{eqmatrice}.
  We have that $M^0 \in G(\widetilde{f^0})$.
  We verify that for all point $d$ in $\mathcal{C}^+(\widetilde{f^0})$ we have $(M^0 \partial^0 (M^0)^{-1})_{-+} d=0$ except when $d=a$ and $c \in \mathcal{C}^-(\widetilde{f^0})$ in which case we have $(M^0 \partial^0 (M^0)^{-1})_{-+} a=\pm c$.
  In this case, $f^0(c)<f^0(b)$ since $a$ and $b$ have consecutive critical values.
  Assume for example that $(M^0 \partial^0 (M^0)^{-1})_{-+} a= c$.
  Consider $M'_0$ to be a matrix $\begin{pmatrix}
  I_{p_0} & 0 \\
  N'_0    & I_{q_0} \\
    \end{pmatrix}$ where all coefficients of $N'_0$ vanish except the one whose column correspond to $b$ and whose row is the one of $c$, which is $-1$.
   We have that $M'_0 \in G_0(\widetilde{f^0})$. 
   Let $M'$ be the element in $G(\widetilde{f^0})$ which is the identity in every degree except in degree $0$ for which it is $M'_0$.
A straight computation shows that $((M'M^0)\partial^0(M'M^0)^{-1})_{-+}=0$.
It is then not difficult to see that $((M'M^0)\partial^0(M'M^0)^{-1})_{++}$ has the desired homology.

If the pair is of index $(n,n-1)$, it is simpler as the independency of the bifurcation directly gives $\partial^0 a =\pm b$, and $M^0 \in G(\tilde{f})$ is such that $(M^0 \partial^0 (M^0)^{-1})_{-+}=0$ and that $(M^0 \partial^0 (M^0)^{-1})_{++}$ has the desired homology.
 
If the pair is of label $-$, then results of Subsection \ref{opposite} lead to the conclusion.

\item \textbf{Birth of a pair of index $(1,0)$ or $(n,n-1)$.} 
We first assume that the pair appearing during the path, denoted $(a,b)$, is of label $+$ and of index $(1,0)$.
There is a Morse-Smale pseudo-gradient $X^1$ adapted to $f^1$ and a Morse-Smale pseudo-gradient $X^0$ adapted to $f^0$ such that:
\[\partial^1_2=\begin{pmatrix}
  &                      &   &  &                     &  \\
   & \partial^0_{++,2}  &   &  & \partial^0_{+-,2} &  \\
   &                      &   &  &                     &  \\
 0 & \hdots               & 0 & 0& \hdots              & 0 \\
   &                      &   &  &                     &  \\
   & \partial^0_{-+,2}  &   &  &  \partial^0_{--,2}&  \\
   &                      &   &  &                     &  \\
\end{pmatrix}
\]
where the row with $0$ corresponds to $a$.
Assume that $\partial^0 a = \pm b \pm c$ and that $c$ is a critical point of $f^1$ of label $-$ and index $0$.
We get:
\[\partial^1_{1}=\begin{pmatrix}
   &                      &  & 0      &  &                     &  \\
   & \partial^0_{++,1}  &  & \vdots &  & \partial^0_{+-,1} &  \\
   &                      &  & 0      &  &                     &  \\
 0 & \hdots               & 0 & \pm 1     & 0& \hdots              & 0 \\
   &                      &  & 0      &  &                     &  \\
   & \partial^0_{-+,1}   &  & \vdots &  &  \partial^0_{--,1}&  \\
   &                      &   & 0      & &                      &  \\
   &                      &  & \pm 1  & &                      &  \\
\end{pmatrix},\]
where the last row is the one of $c$, and the row in the middle with a lot of $0$ is the one of $b$.
If $M^1 \in G(\widetilde{f^1})$ is such that $(M^1 \partial^1 (M^1)^{-1})_{-+}=0$ and $(M^1 \partial^1 (M^1)^{-1})_{++}$ has the desired homology, we denote:
\[M^1_j= \begin{pmatrix}
  &           &   & 0      &  &   & \\
  & I_{p^0_j} &   & \vdots &  & 0 & \\
  &           &   & 0      &  &   & \\
 0& \hdots    & 0 & 1      & 0 & \hdots  & 0 \\
  &  N_j      &   & L_j    &  & I_{q^0_j} & \\
 
 \end{pmatrix}, \]
when $j=0$ and $j=1$.
We consider $M^0 \in G(\widetilde{f^0})$ to be such that $M^0_j=M^1_j$ in all degree different from $1$ and $0$, and to be 
\[M^0_j= \begin{pmatrix}
  & I_{p^0_j} &   & 0 & \\
  &  N_j      &   & I_{q^0_j} & \\
 
 \end{pmatrix}, \]
in degree $1$ and $0$.
We then have the exact same equations as for the birth bifurcation of a pair of index $(k+1,k)$ for $1\leq k \leq n-2$, that is, Equations \ref{eq++}, \ref{eq+-}, \ref{eq-+}, and \ref{eq--}.
The results follow.

If the index of the pair is $(n,n-1)$, then we can also do as in the case of the birth bifurcation of a pair of index $(k,k-1)$ with $1\leq k \leq n-2$.

If the pair is of label $-$, then results of Subsection \ref{opposite} lead to the conclusion.
\end{itemize}

\end{proof}

\begin{proof}[Proof of Theorem \ref{th1}.]
The proof of the theorem is implied by the union of Lemma \ref{induction} and Lemma \ref{invariance}.
\end{proof}

\subsection{Differences with Barannikov's work}
\label{barra}
As the starting point of this article is the same as Barannikov's in \cite{Bar}, we may wonder if the condition displayed in Theorem \ref{th1} is contained in \cite{Bar}.

We give in this subsection the example of a germ $\tilde{f_0}$ that does not extend non-critically, because it does not fulfill the condition of Theorem \ref{th1}.
However, we show that Barannikov's theorems cannot say if the germ does not extend non-critically. 
We will briefly introduce what are the objects defined in \cite{Bar} and recall the results which are proved, but first, we describe the germ $\tilde{f_0}$.
We will also use the germ $\tilde{f_0}$ as an example to introduce the objects defined in \cite{Bar}.

We refer to \cite{Bar} and \cite{Lau} for any further details.
\subsubsection{A germ \texorpdfstring{$\tilde{f_0}$}{Lg} that does not extend non-critically}

Let $2\leq k \leq n-2$.
Let $\tilde{f_0}$ be a Morse germ such that $f_0$ has $6$ critical points:
\begin{itemize}
\item One maximum $max$ of label $+$ and one minimum $min$ of label $-$,
\item Two critical points $a$ and $b$ of index $k+1$. One, say $a$, of label $+$ and the other, say $b$, of label $-$ such that $f_0(a)>f_0(b)$. Thus, for any integer $r$, the matrix 
\[\begin{pmatrix}
1 & 0 \\
r & 1 \\
\end{pmatrix}\] is in $G_{k+1}(\tilde{f_0})$.
\item Two critical points $c$ and $d$ of index $k$. One, say $c$, of label $+$ and the other, say $d$, of label $-$, such that $f_0(c)>f_0(d)$.
 Thus, for any integer $r$ the matrix 
 \[\begin{pmatrix}
1 & 0 \\
r & 1 \\
\end{pmatrix}\] is in $G_{k}(\tilde{f_0})$.
\end{itemize}
We assume that we are given a Morse-Smale pseudo-gradient $X$ adapted to $f_0$ such that:
\[\partial_{k+1}=\begin{pmatrix}
7 & 5 \\
-3 & -2 \\
\end{pmatrix}
\]
where:
\begin{align*}
\partial_{++,k+1}& =7,\\
\partial_{+-,k+1}& =5,\\
\partial_{--,k+1}&=-2,\\
\partial_{-+,k+1}&=-3.\\
\end{align*}
Notice that the homology of $\partial$ is the one of $\Sn$, as $\partial_{k+1}\in SL_2(\mathbb{Z})$.

Theorem \ref{th1} states that if $\tilde{f_0}$ extends non-critically, we would have matrices $M_{k+1}$ and $M_k$ in respectively $G_{k+1}(\tilde{f_0})$ and $G_k(\tilde{f_0})$ such that 
\[M_{k+1}\partial_{k+1} M_k = \begin{pmatrix}
\pm 1 & 5 \\
0 & \pm 1 \\
\end{pmatrix}.\]
A straight computation shows that it would imply $7 \equiv ~\pm 1~ [5]$, see also Theorem \ref{2indices} for details. 
But this is false.
Thus, $\tilde{f_0}$ does not extend non-critically.

\subsubsection{Theorems of Barannikov}

We briefly recall the main results of \cite{Bar}.
Details can also be found in \cite{Lau}.

Let $\tilde{f}$ be a Morse germ along $\Sn$ and $X$ a Morse-Smale pseudo-gradient adapted to $f$.
 The \emph{abstract Framed Morse Complex} (that we will denote \emph{FMC}) of $\tilde{f}$ associated to $X$ is a collection of $n$ vertical lines, one for each index, on which we place vertices corresponding to the critical points of $f$ as follows.
 If $b$ is a critical point of $f$ of index $k$, we place a vertex on the $k$-th line at height $f(b)$.
 Then we join a pair of critical points $(a,b)$ where $a$ is of index $k+1$ and $b$ is of index $k$ by a segment when $\partial a$ has a non-null component along $b$.
 That is, we link $a$ to $b$ if $\partial a = n_{a,b}b + \hdots$, where $n_{a,b}$ is in $\mathbb{Z} \setminus \{0\}$.
 We write the integer $n_{a,b}$ above the segment joining $a$ and $b$.
 If $a$ is of label $+$, an arrow pointing down is added to the vertex representing $a$, and if it is of label $-$, the arrow points up.
 As an example, the FMC of the germ $\tilde{f_0}$ described above is pictured in Figure \ref{afaire}. 
 The FMC of a trivial germ is also pictured in Figure \ref{afaire2}.
 A FMC is then just a way to picture the boundary operator associated to a Morse-Smale pseudo-gradient adapted to a Morse germ $\tilde{f}$.
 The main idea of \cite[Theorem 1]{Bar} is that if a germ $\tilde{f}$ extends non-critically, then any associated FMC can be linked through some allowed modifications to the FMC of a trivial germ.
 The modifications of the diagram are given by bifurcations of a generic path of functions starting at $\tilde{f}$ and ending at a trivial germ.
 We will not describe in details those modifications, but will describe below the modifications of another simpler object introduced by Barannikov.
 We just briefly say that the cancellation of a pair $(a,b)$ gives a new diagram where the vertices corresponding to $a$ and $b$ and all segments whose endpoints contain $a$ or $b$ disappear. 
 The birth of a pair gives a diagram with two new vertices, and segments whose endpoints contain $a$ or $b$ in accordance with the new boundary operator.
 A handle slide changes the segments between vertices and the integers placed above. 
 We point out that Subsection \ref{cobord} of the present article is a detailed version of \cite[Theorem 1]{Bar}.
 
 \begin{figure}[!ht]
 \begin{center}
 \includegraphics[scale=0.5]{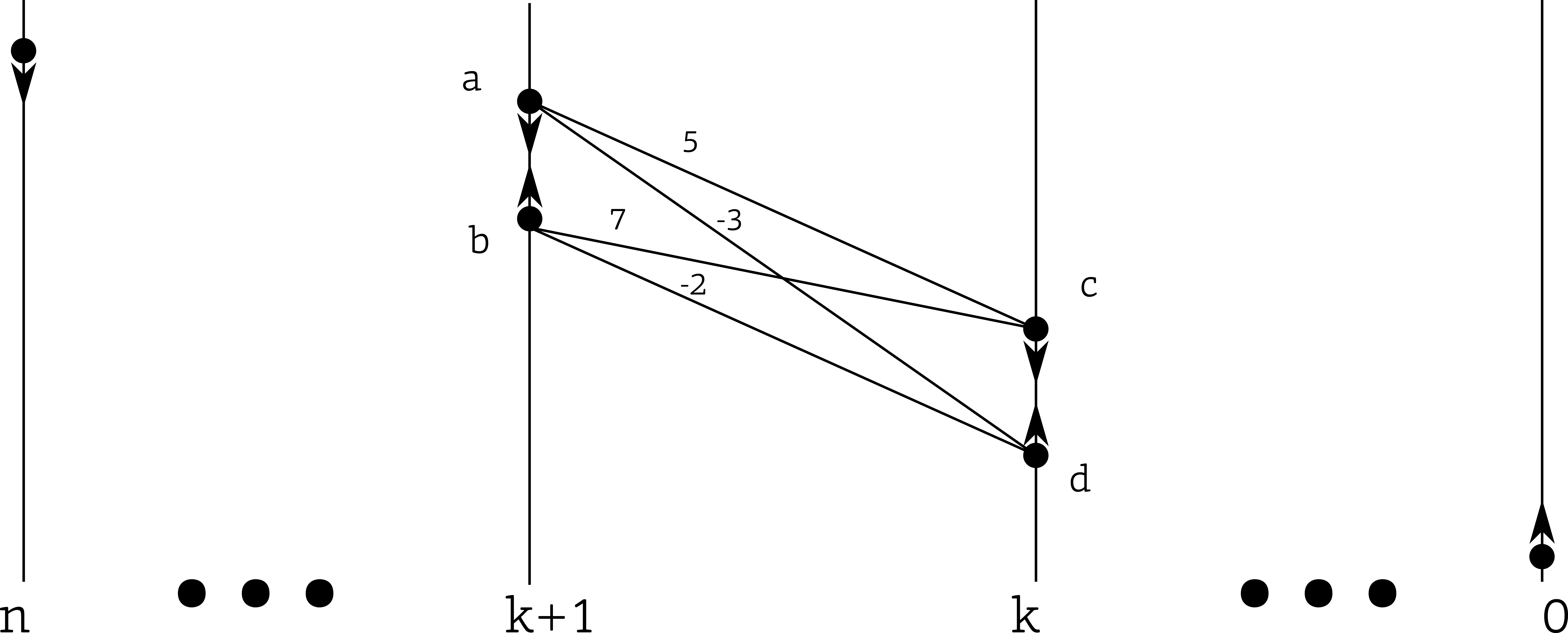}
 \caption{FMC of $\tilde{f_0}$}
  \label{afaire}
\end{center}  
 \end{figure}
 
  \begin{figure}[!ht]
 \begin{center}
 \includegraphics[scale=0.5]{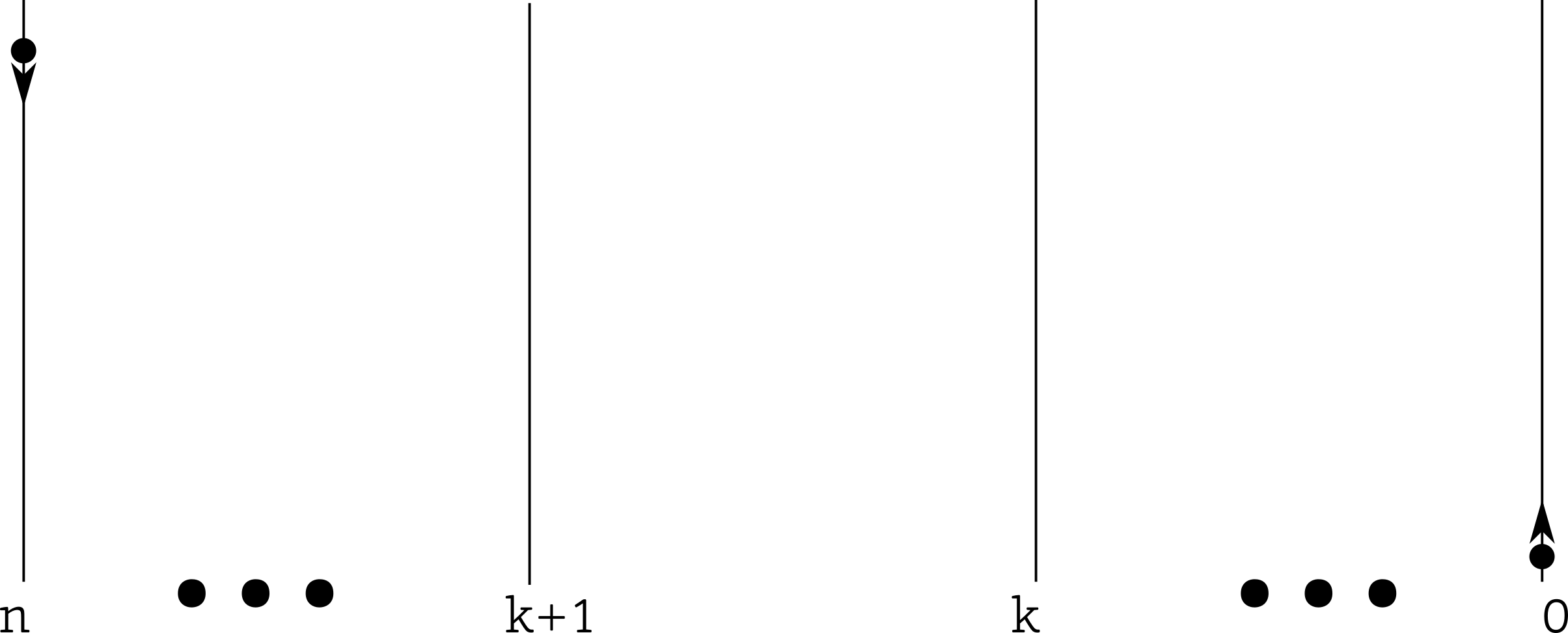}
 \caption{FMC of a trivial germ}
 \label{afaire2}
\end{center}  
 \end{figure}

 In order to give a condition of non-critical extension which is easier to handle, Barannikov considers complexes with coefficients in a field.
 Let $\mathbb{F}$ be a field such as $\mathbb{Q}$ or $\mathbb{Z}/(p)$ where $p$ is prime.
 Denote by $T(r,\mathbb{F})$ the group of invertible upper triangular matrices in $GL_r(\mathbb{F})$. 
 Given a permutation $\sigma$ in $\mathfrak{S}(r)$, the permutation matrix associated to $\sigma$ is a matrix $S$ such that $Se_j=e_{\sigma(j)}$, where $(e_j)_{1\leq j \leq r}$ is the canonical basis.
 We recall the following standard theorem:
 \begin{theorem}[Bruhat decomposition]
 Let $M \in GL_r(\mathbb{F})$. 
There exist a permutation matrix $S$ and two matrices $T$ and $T'$ in $T(r,\mathbb{F})$ such that
 \[M=TST'.\]
 \end{theorem}
 Using this theorem, Barannikov proves that there is a boundary operator $\partial^{\mathbb{F}}$ defined in a canonical way on $\mathbb{F}\mathcal{C}(f)$.
  Let $f$ be a Morse function on a manifold of dimension $n$.
 Denote by $(a^k_i)_{1\leq i \leq r_k}$ its critical points of index $k$, given with the decreasing order of their critical values.
 That is, $a^k_1$ is a critical point of $f$ of index $k$ with highest critical value, and $a^k_{r_k}$ has the lowest critical value.
 Let $X$ be a Morse-Smale pseudo-gradient adapted to $f$.
 Let $\mathbb{F}$ be a field.
 Let $\partial(X)$ be the associated boundary operator with coefficients in $\mathbb{Z}$.
 Denote by $\partial^{\mathbb{F}}(X)$ the boundary operator we get after tensorization with $\mathbb{F}$.
 Let $pr_{\mathbb{F}} : x \mapsto x\otimes 1$ going from $\mathbb{Z}\mathcal{C}(f)$ to $\mathbb{Z}\mathcal{C}(f)\otimes \mathbb{F}$, then $\partial^{\mathbb{F}}(X)(pr_{\mathbb{F}}(x)) = pr_{\mathbb{F}}(\partial(x)).$ 
 \begin{theorem}
There are matrices $T_k \in T(r_k,\mathbb{F})$ such that the boundary operator denoted $\partial^{\mathbb{F}}$ defined by 
\[\partial_{k+1}^{\mathbb{F}} := T_k \partial^{\mathbb{F}}(X) (T_{k+1})^{-1}\]
presents the following properties.
 For any $k$ between $0$ and $n-1$ and any critical point $a^{k+1}_i$, one and only one of the following occurs:
 \begin{itemize}
 \item $\partial^{\mathbb{F}} a^{k+1}_j = a^{k}_{\sigma(j)} $ for some critical point $a^k_{\sigma(j)}$, and where $ a^{k}_{\sigma(j)} $ is not in the image of any other point,
 \item $\partial^{\mathbb{F}} a^{k+1}_j = 0$.
 \end{itemize}
 Moreover, this boundary operator is unique and does not depend on $X$.
 \end{theorem}
If $\partial^{\mathbb{F}} a^{k+1}_j = 0$, either $a^{k+1}_j$ is the image of another point of index $k+2$, or it represents a non-null homology class of $\partial^{\mathbb{F}}$.
 As the homology of the sphere is $\mathbb{F}$ in degree $0$ and $n$, there are only one maximum $max$ and one minimum $min$ for which $\partial^{\mathbb{F}}max = 0$ and $\partial^{\mathbb{F}}min = 0$ and which are not the image of any point.
 
 The fact that this boundary operator does not depend on the pseudo-gradient vector field comes from that $T(r_k,\mathbb{F})$ represents in this setting the group of algebraic results of handle slides between points of index $k$.
 Thus the independence is a corollary of Theorem \ref{pggenericity}.
 
 Another interpretation of how are paired the points $(a^{k+1}_j, a^{k}_{\sigma(j)})$ is given in \cite{Lau}.
 
  For the example of $\tilde{f^0}$, it is sufficient to know that $\partial^{\mathbb{F}} b = c$ if $5 \neq 0$ in $\mathbb{F}$.
  In this case, $\partial^{\mathbb{F}} a =d$.
  If the characteristic of the field $\mathbb{F}$ is $5$, then $\partial^{\mathbb{F}} b = d$ and $\partial^{\mathbb{F}} a = c$. 
  
 One can define the diagram of the complex with coefficient in $\mathbb{F}$ in the exact same way as the construction of the FMC with coefficients in $\mathbb{Z}$.
 It is called the \emph{canonical form} of the FMC of $\tilde{f}$ with respect to the field $\mathbb{F}$, and it only depends on $\tilde{f}$ and $\mathbb{F}$.
 We give Barannikov's theorem about non-critical extensions of Morse germs:
 \begin{theorem}
 If a germ $\tilde{f}$ defined along the sphere $\Sn$ extends non-critically, then for any field $\mathbb{F}$, the canonical form of its FMC can be reduced to the canonical form of the FMC of a trivial germ, through the modifications pictured on Figures \ref{cross1}, \ref{cross1b}, \ref{cross2} and \ref{death}.
 Some explanations are needed to understand the figures:
 \begin{itemize}
 \item each of the figure is meant to be the graphic retranscription of a bifurcation occuring during a generic non-critical path of function. However, it is purely combinatoric and do not represent actual bifurcations;
 \item the figures only represent the pairs of vertices on which the modifications apply, but of course the canonical forms of the FMCs may display many others vertices;
 \item in all of the figures, the arrows (i.e. the labels) of the vertices are not represented, but the vertices with an arrow pointing down can only go down and the vertices with an arrow pointing up can only go up;
 \item the small vertical double arrows between the vertices in Figure \ref{cross1}, \ref{cross1b} and \ref{cross2} mean that either the highest vertex of the two go below the other one (and thus the highest vertex has its arrow pointing down), or that the lowest vertex of the two go above the other one (and thus its arrow points up);
 \item in Figures \ref{cross1}, \ref{cross1b} and \ref{cross2}, the big diagonal double arrows between the diagrams indicate that the modifications are reversible, with taking care of the sense of moving of the vertices indicated by their respective labels;
 \item in Figure \ref{death}, the arrows of the two vertices disappearing must have the same direction (thus the points have same labels) and the heights of the vertices must be consecutive.
 \end{itemize}
 Up to the rules described above, all modifications of the canonical form of the FMC are allowed.
 
 \end{theorem}
 \begin{figure}[!ht]
 \begin{center}
 \includegraphics[scale=0.7]{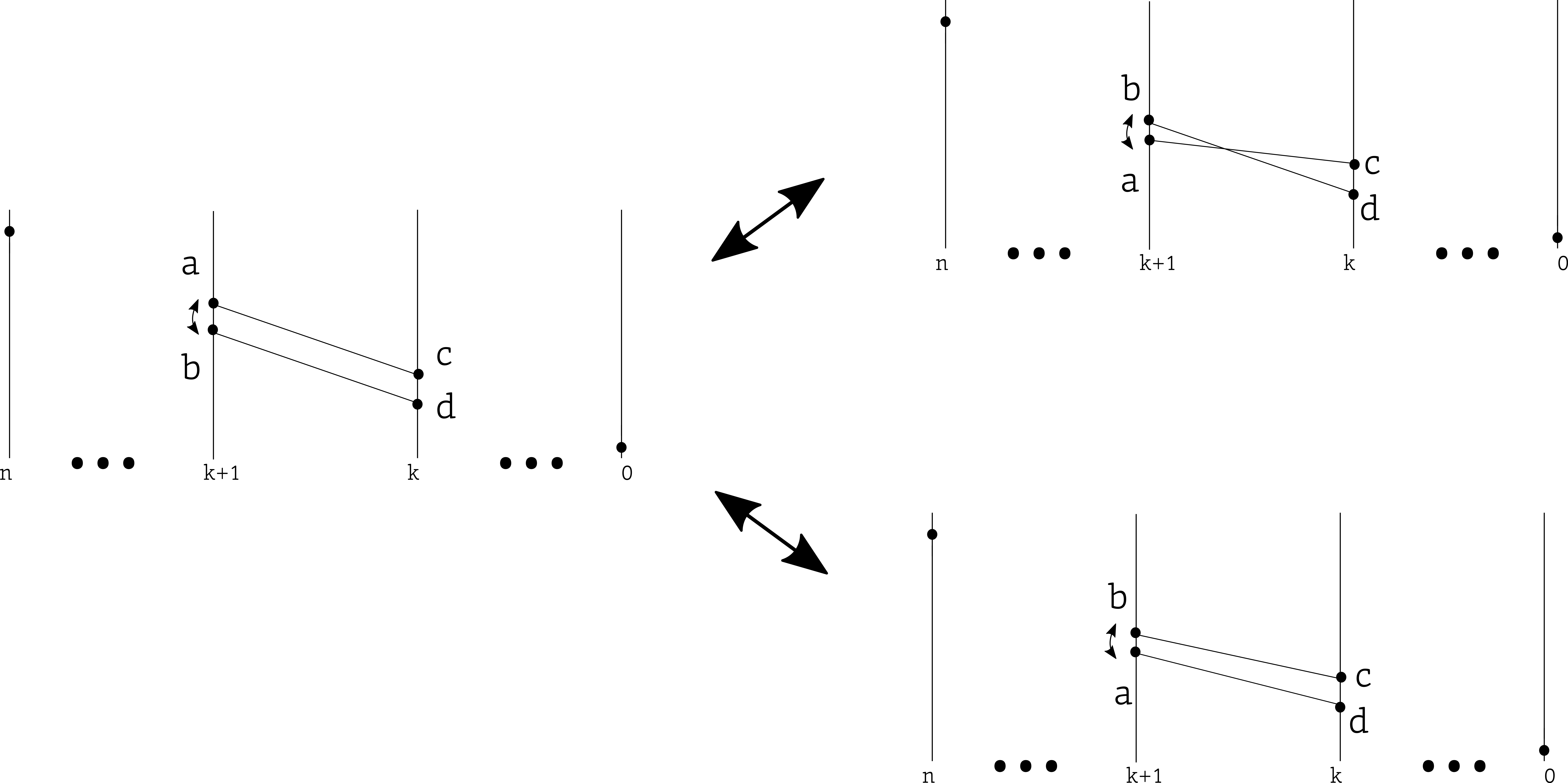}
 \caption{Crossing of two vertices of same index}
 \label{cross1}
\end{center}  
 \end{figure}
  \begin{figure}[!ht]
 \begin{center}
 \includegraphics[scale=0.7]{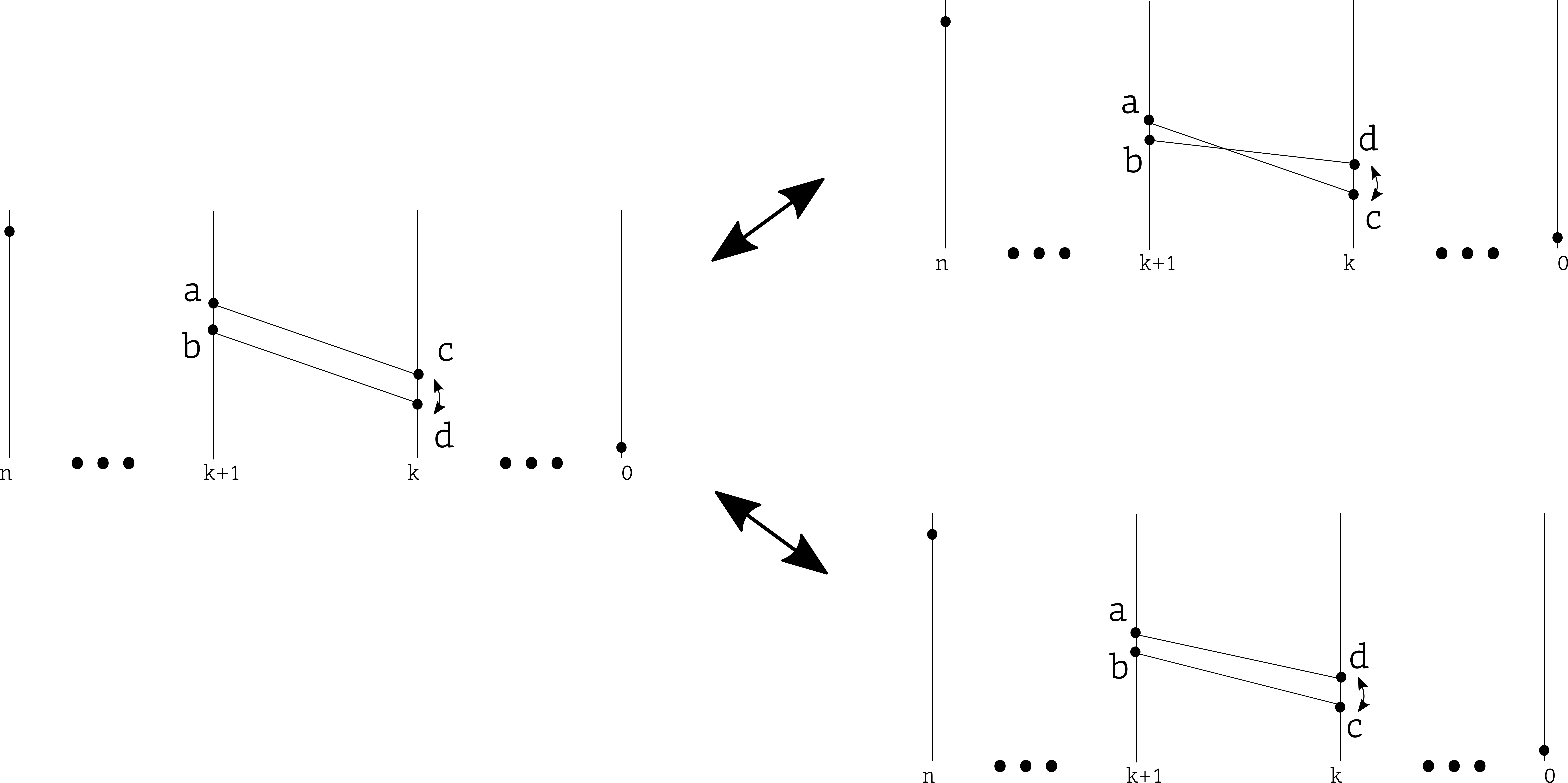}
 \caption{Crossing of two vertices of same index}
 \label{cross1b}
\end{center}  
 \end{figure}
  \begin{figure}[!ht]
 \begin{center}
 \includegraphics[scale=0.7]{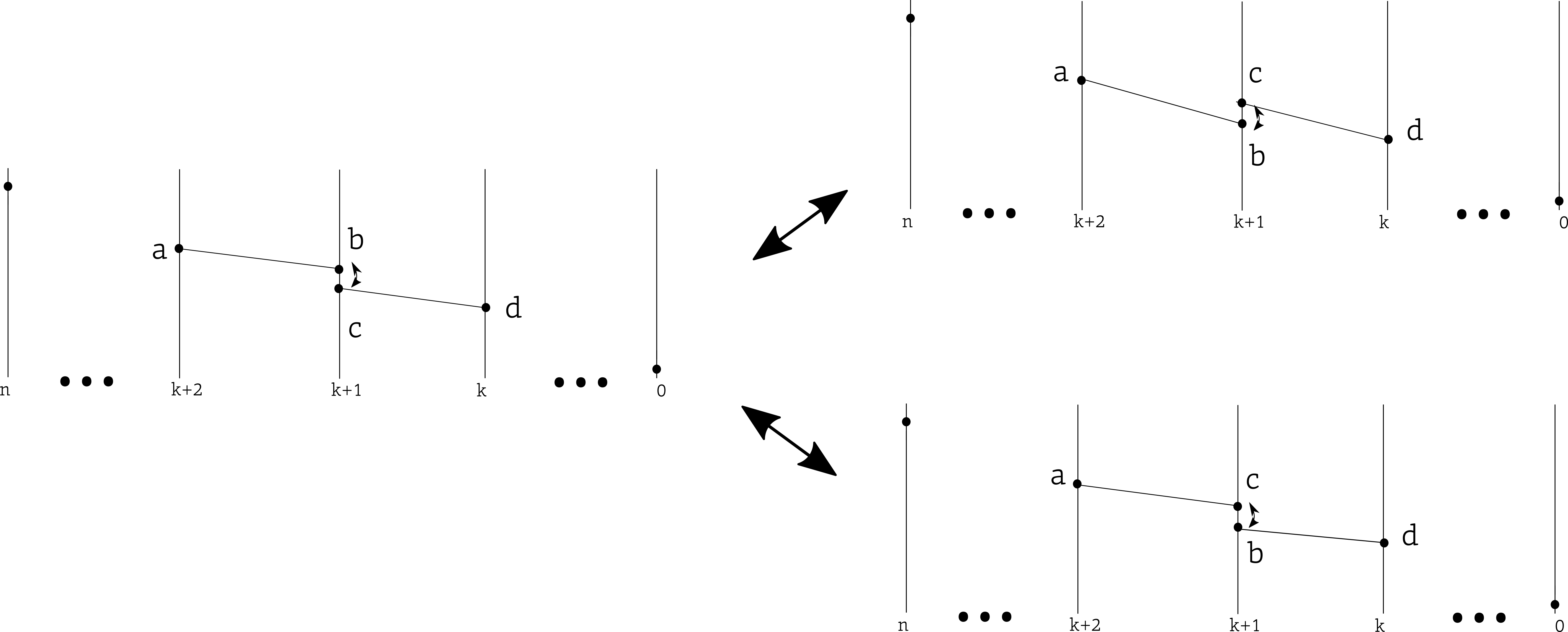}
 \caption{Crossing of two vertices of same index}
 \label{cross2}
\end{center}  
 \end{figure}
  \begin{figure}[!ht]
 \begin{center}
 \includegraphics[scale=0.7]{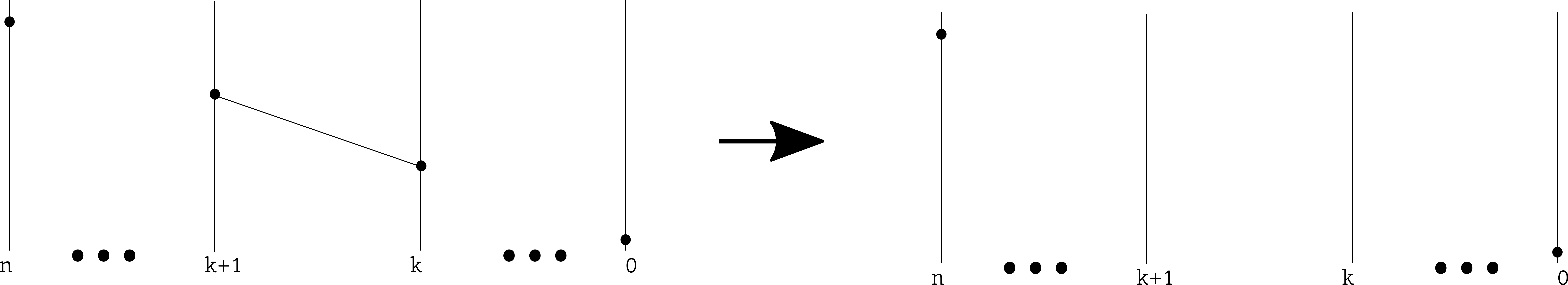}
 \caption{Death of a pair of vertices}
 \label{death}
\end{center}  
 \end{figure}

We insist that the theorem is of purely cominatoric nature.
 The main part of the proof is that if the canonical form can be reduced to the canonical form of a trivial germ, then it can be done without making appear new pairs of vertices.
 This is why the big horizontal arrow between the FMCs in Figure \ref{death} goes only from left to right. 
 Thus, given a field, we can always verify in a finite amount of steps if it reduces to the canonical form of a trivial germ.
 We also emphasize that the ability to be reduced to the canonical form of the FMC of a trivial germ depends on the field in consideration.

\subsubsection{Canonical forms of FMC of \texorpdfstring{$\tilde{f_0}$}{Lg}}

 \begin{figure}[!ht]
 \begin{center}
 \includegraphics[scale=0.5]{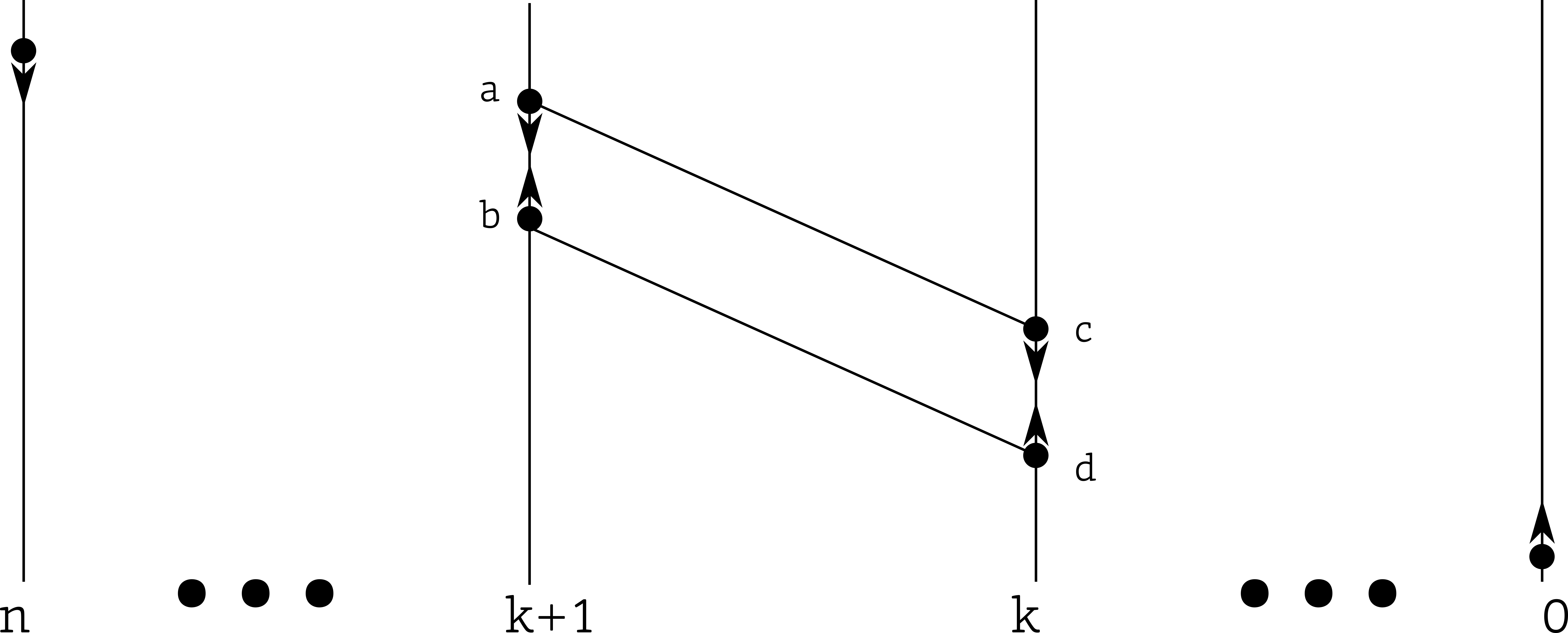}
 \caption{Canonical form of the FMC of $\tilde{f_0}$, when $\Char(\mathbb{F})=5$}
 \label{tatatitata}
\end{center}  
 \end{figure}
 
  \begin{figure}[!ht]
 \begin{center}
 \includegraphics[scale=0.5]{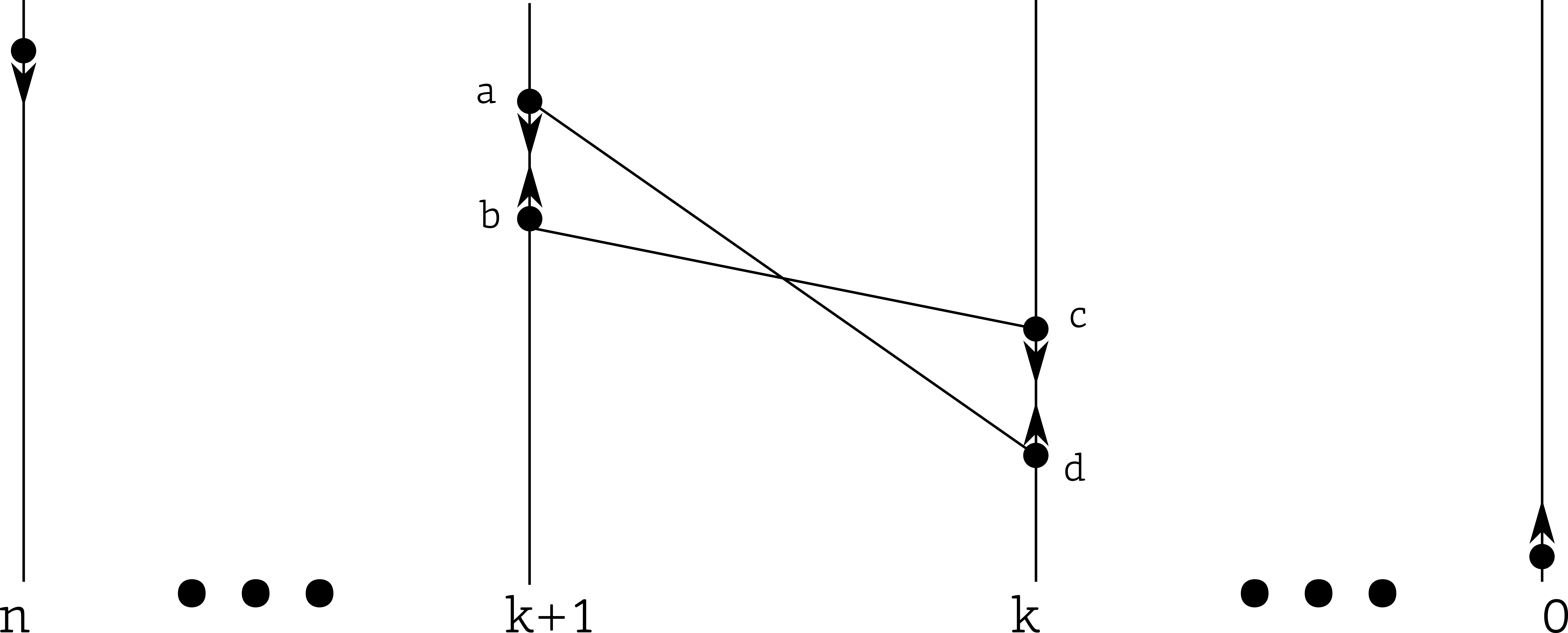}
 \caption{Canonical form of the FMC of $\tilde{f_0}$, when $\Char(\mathbb{F})\neq 5$}
 \label{tatatitata2}
\end{center}  
 \end{figure}

 For the germ $\tilde{f_0}$, the canonical form of the FMC with respect to $\mathbb{F}$ depends on the field, more precisely on its characteristic.
 \begin{itemize}
 \item \textbf{Case 1: $\mathbb{F}$ is of characteristic $5$.}
 
 Then, the canonical form of the FMC is as on Figure \ref{tatatitata}. We can modify the FMC in order to have the one on Figure \ref{oooo} which can be reduced to a trivial FMC, by cancelling the two pairs of vertices.
 
 \item \textbf{Case 2: the characteristic of $\mathbb{F}$ is different from $5$.}
 
 Then, the canonical form of the FMC is as on Figure \ref{tatatitata2}.
 We can move down the vertex representing the point $a$ below the vertex representing $b$.
 If we do so, two different canonical forms are allowed and are pictured on Figures \ref{oooo} and \ref{eeee}, depending on the boundary operator of the function after the crossing.
We can move down $a$ with respect to Figure \ref{oooo}, which can be reduced to a trivial FMC.
\end{itemize}

In conclusion, the condition given by Barannikov, using fields instead of the integers, is weaker than Theorem \ref{th1} in this case.

 \begin{figure}[!ht]
 \begin{center}
 \includegraphics[scale=0.5]{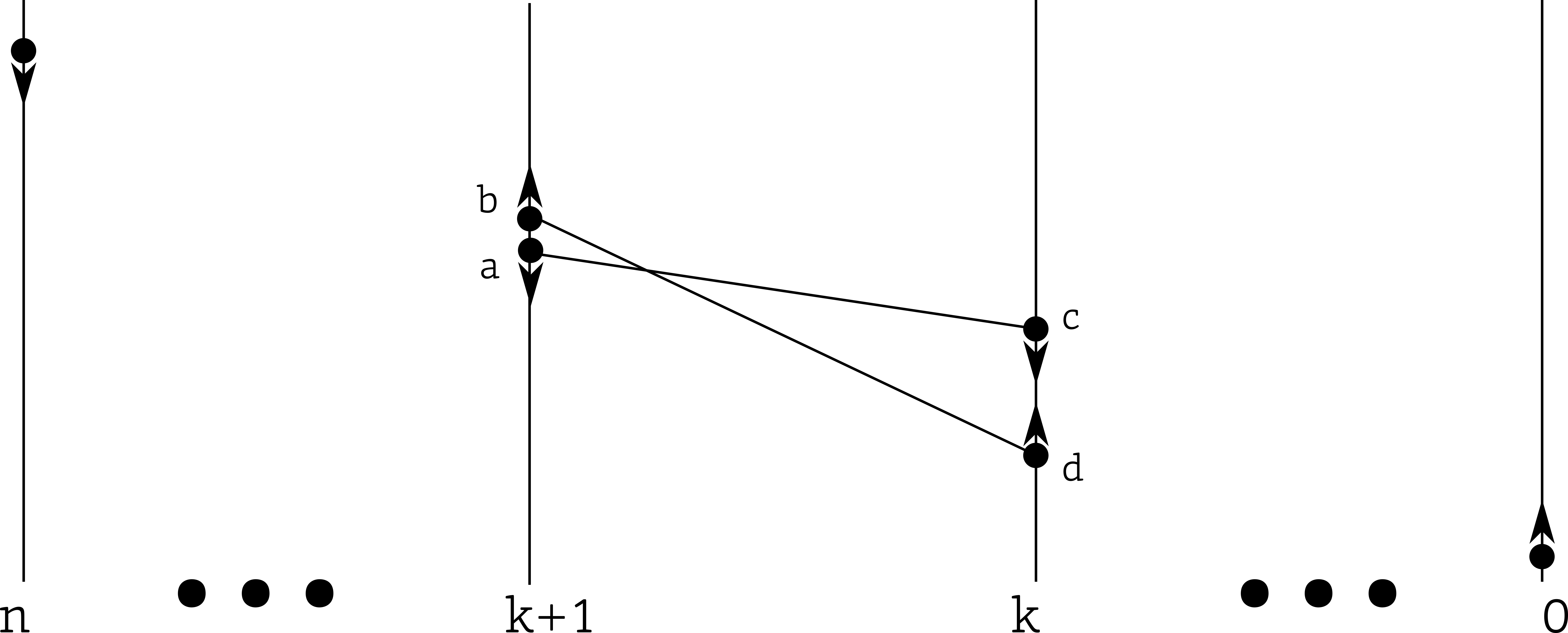}
 \caption{A possible modification of the canonical form of the FMC of $\tilde{f_0}$}
 \label{oooo}
\end{center}  
 \end{figure}
 
  \begin{figure}[!ht]
 \begin{center}
 \includegraphics[scale=0.5]{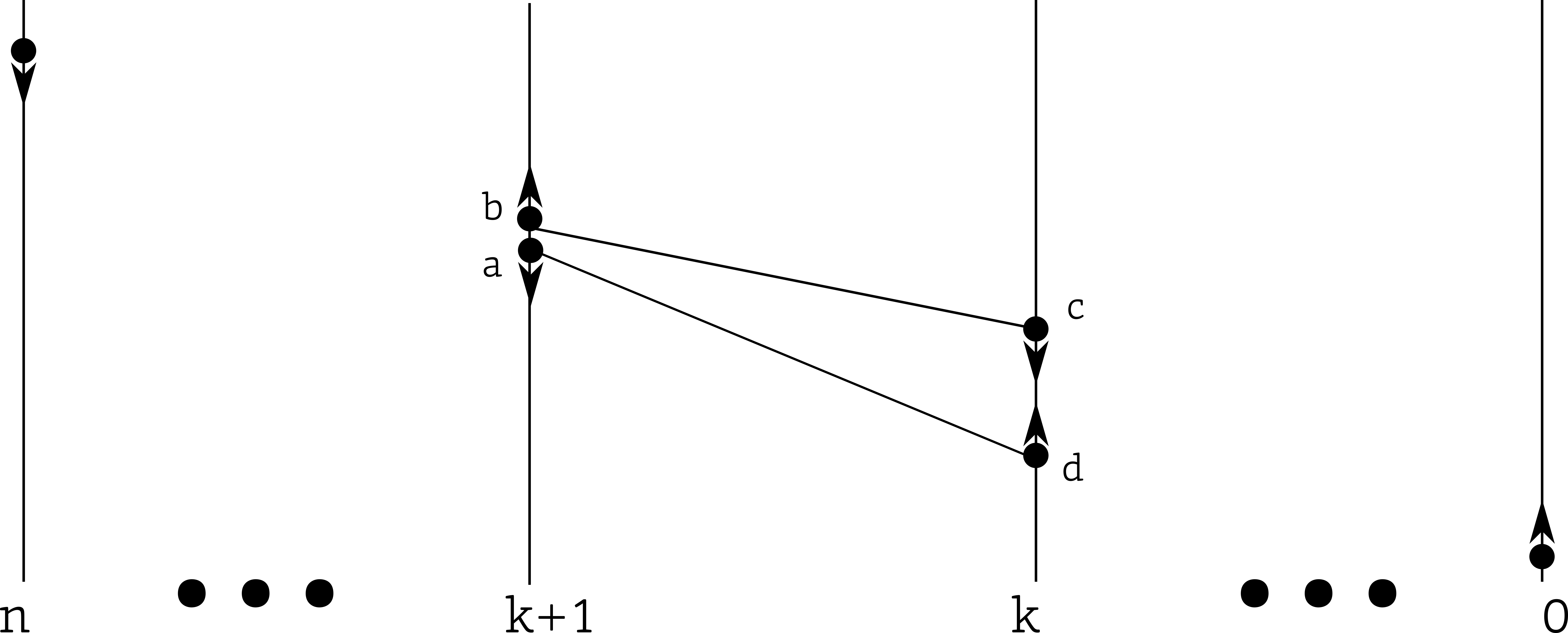}
 \caption{The other possible modification of the canonical form of the FMC of $\tilde{f_0}$}
  \label{eeee}
\end{center}  
 \end{figure}

 \section{From homology to homotopy}
  \label{homo}
  \subsection{Some germs for which condition of Theorem \ref{th1} is sufficient}
  In the previous section, we found a necessary condition for a germ to have a non-critical extension that deals with the Morse chain complex.
  We show now that given a germ $\tilde{f}$ which has property $\mathcal{P}$ and boundary operator $\partial$, there is always a germ $\widetilde{f^1}$ such that:
  \begin{itemize}
  \item $f^1$ and $f$ share the same adapted Morse-Smale pseudo-gradient. Thus they have same boundary operator through identification of the critical points of $f$ to those of $f^1$;
  \item $\widetilde{f^1}$ has property $\mathcal{P}$;
  \item $\widetilde{f^1}$ extends non-critically.
  \end{itemize}
  We will see the reason why $\widetilde{f^1}$ extends non-critically.
 
  We begin with a lemma.
    
  \begin{lemma}[Construction of a germ]
  \label{usefullemma}
  Let $n$ be an integer higher than $0$. Let $f$ be a Morse function defined on $\Sn$ and $\mathcal{C}(f)$ the set of critical points of $f$.
   For any map \[\phi : \mathcal{C}(f) \to \{-,+\}\] there is a Morse germ $\tilde{f}_{\phi}$ such that the label of a critical point of $f$ for $\tilde{f}_{\phi}$ is given by its image by $\phi$. 
  \end{lemma}
\begin{proof}
For any critical point $x$ of $f$ there is an open set $\mathcal{U}_x$ of $\Sn$ such that $x$ is the only critical point of $f$ in $\mathcal{U}_x$.
 For any $x$ there is a smooth function $g_x$ from $\Sn$ to $\mathbb{R}$ which is $1$, respectively $-1$, on $x$ if $\phi(x)=-$, respectively if $\phi(x)=+$, and $0$ outside $\mathcal{U}_x$. 
 Let \[g:= \sum\limits_{x \in \mathcal{C}(f)} g_x.\]
 A representative of $\tilde{f_{\phi}}$ is given by any function on $\Sn\times [0,\varepsilon)$ which is $f$ on $\Sn$ and whose derivative along $t$ is $g$.
 Notice that for $\varepsilon$ small enough, it is non-critical. 
\end{proof}

\begin{lemma}
\label{Freed+}
Let $\tilde{f}$ be a germ on $\Sn \times [0, \varepsilon)$ such that $f$ is excellent.
Let $X$ be an adapted Morse-Smale pseudo-gradient and let $\partial$ be the associated boundary operator.
Let $\mathcal{S}$ be a subset of $\mathcal{C}^+_k(\tilde{f})$ with $2\leq k \leq n-2$ such that all critical points in 
\[f^{-1}\left([\min \left(f\left(\mathcal{S}\right)\right),\max \left( f \left( \mathcal{S}\right) \right) ] \right)\] are in $\mathcal{S}$.

Then, for any isomorphism 
\[P:\mathbb{Z}\mathcal{C}(f)\to \mathbb{Z}\mathcal{C}(f) \]
 restricting to the identity on $\mathbb{Z}\left(\mathcal{C}(f)\setminus \mathcal{S}\right)$, there is a non-critical generic path of functions $F:\Sn\times[0,1]\to \RR$ continuing $\tilde{f}$ with no birth or death bifurcation such that $g:=F(\centerdot,1)$ is a Morse function and:
\begin{itemize}
\item the order of the critical points of $g$ with respect to the critical values is the same than the one of $f$,
\item there is a Morse-Smale pseudo-gradient $X_1$ adapted to $g$ such that, if $\partial_1$ is its associated boundary operator, we have  \[ \partial_1=P \partial P^{-1}. \]
\end{itemize}
\end{lemma}
\begin{proof}
We only need to consider $P=I+E_{i,j}$ to be an elementary matrix, result of a handle slide between two points of $\mathcal{S}$, as $GL(\mathbb{Z}\mathcal{S})$ is generated by the elementary matrices.
If $a\in \mathcal{S}$ and $b \in \mathcal{S}$ such that $f(a)<f(b)$, we can find a non-critical path continuing $\tilde{f}$ without birth or death bifurcations such that the endpoint $f^1$ of this path is a Morse function with $f^1(b^1)<f^1(a^1)$.
Indeed, it is possible with a minor modification of the proof of Lemma 2.1 about crossing bifurcation in \cite{Laud5}, which is inspired by results of Cerf \cite[p. 40-59]{Cerf}.
Moreover, it can be done such that 
\[ [\min \left(f^1\left(\mathcal{S}\right)\right),\max \left( f^1 \left( \mathcal{S}\right) \right) ] \subset [\min \left(f\left(\mathcal{S}\right)\right)-\varepsilon,\max \left( f \left( \mathcal{S}\right) \right)-\varepsilon ], \]
for $\varepsilon$ as small as wanted, by playing on the speed of the points of $\mathcal{S}$ during the path.
With such a path, we can perform a handle slide of $a^1$ over $b^1$, that we could not do with $f$.
We use \cite[Theorem 7.6]{Miln} here.
Finally, we can then reset the critical points of $\mathcal{S}$ in the same order as before without performing any handle slide, which will keep the same boundary operator.
\end{proof}
  We have the theorem:
\begin{theorem}
\label{best}
Let $n$ be an integer higher than $6$.
Let $\tilde{f}$ be a Morse germ such that $f$ has no critical points of index $0$, $1$, $n-1$ and $n$ different from its global maximum and its global minimum.
Assume that $\tilde{f}$ has property ($\mathcal{P}$).
There exists a germ $\widetilde{f^1}$ such that $G_k(\widetilde{f^1})=G_k(\tilde{f})$ for all $k$, and such that $\widetilde{f^1}$ extends non-critically to the ball $\Dn$.
\end{theorem}  
  \begin{proof}
  We build the germ $\widetilde{f^1}$ step by step from $\tilde{f}$.
  
  By classical Morse theory, see \cite[Sec. 4]{Miln}, there is a generic path of function $(H_t)_{t\in [0,1]}$ from $f$ to a function $f^1$ without birth or death bifurcations such that if $a^1$ and $c^1$ are two critical points of $f^1$ of respective indices $k+1$ and $k$, then $f^1(a^1)>f^1(c^1)$.
  We emphasize that we do not ask the path to be non-critical.
  We can even assume that if $a^0$ and $b^0$ are two critical points of $f$ of same index such that $f(a^0)>f(b^0)$, then the corresponding points $a^1$ and $b^1$ for $f^1$ verify $f^1(a^1)>f^1(b^1)$. 
  Moreover, this path can be such that no crossing of points of same index happens, by first lowering the critical values of the non-extrema points of minimal index, and then lowering the critical values of the other critical points by ascending induction on the index.
  Doing so, we see that no pair of critical points crosses twice.
  Using \cite[Corollary 2.2]{Laud5}, we can consider the same adapted Morse-Smale pseudo-gradient for $f$ and $f^1$. 
  Thus, we can also assume that $\partial(f^1)=\partial(f)$ by picking right Morse-Smale pseudo-gradients, where $\partial(f^1)$ (resp. $\partial(f)$) is the boundary operator associated to $f^1$ (resp. $f$) and where we identify critical points of $f^1$ with critical points of $f$ via the path between $f$ and $f^1$.
  
  Using Lemma \ref{usefullemma}, we can consider a germ $\widetilde{f^1}$ such that the label of a critical point $a^1$ of $g$ is the same as the one of the corresponding point $a^0$ of $f$.
  As the order of the critical points of same index is the same for $f^1$ and $f$, the groups $G_k(\widetilde{f^1})$ and $G_k(\tilde{f})$ are also the same.
  
  We prove that $\widetilde{f^1}$ extends to a non-critical Morse function $F^1$ on $\Dn$.
  From now on, all path of functions will be non-critical.
  As $G_k(\widetilde{f^1})=G_k(\tilde{f})$ for all $k$ and $\partial(\widetilde{f^1})=\partial(\tilde{f})$ there is an adapted pseudo-gradient for $f^1$ and its associated boundary operator $\partial(\widetilde{f^1})$ such that $\partial_{-+,k}(\widetilde{f^1})=0$ and $\partial_{++}(\widetilde{f^1})$ defines an acyclic boundary operator on modules generated by critical points of index between $2$ and $n-2$.
  We can move down every point of label $+$ and minimal index, say $k_0$, to a level just above the global minimum.
  We also move all points in $\mathcal{C}^+_{k_0+1}(\widetilde{f^1})$ below points in $\mathcal{C}^-_{k_0+1}(\widetilde{f^1})$, but still above points of $\mathcal{C}_{k_0}(f^1)$. 
As  $H_{k_0}(\partial(\widetilde{f^1}))=0$, and as $\partial_{k_0}(\widetilde{f^1})=0$ since $k_0\geq 2$, we have that  $\partial_{k_0+1}(\widetilde{f^1})$ is surjective.
As $\partial_{-+,k_0+1}(\widetilde{f^1})=0$, we thus have that $\partial_{++,k_0+1}(\widetilde{f^1})$ surjects on $\mathbb{Z}\mathcal{C}^+_{k_0}(\widetilde{f^1})$.  
   We can operate handle slides between points in $\mathcal{C}^+_{k_0+1}$ such that $\partial_{++,k_0+1}(\widetilde{f^1})$ is of the form $\begin{pmatrix}
  0 & P \\
  \end{pmatrix}$ where $P$ is unimodular, thanks to Lemma \ref{Freed+}.
  Up to handle slides between points $\mathcal{C}^+_{k_0}(\widetilde{f^1})$, we can suppose that $P=I_{p_{k_0}}$.
  
  Nothing prevents the points of label $+$ and index $k_0+1$ from going down to the points of index $k_0$ and label $+$, since we have $\partial_{-+,k_0+1}(\widetilde{f^1})=0$.
  We can cancel points in $\mathcal{C}^+_{k_0}(\widetilde{f^1})$ with the points generating the cokernel of $\partial_{++,k_0+1}(\widetilde{f^1})$ in $\mathcal{C}^+_{k_0+1}(\widetilde{f_1})$, as $k_0\geq 2$, using \cite[Theorem 6.4]{Miln} and Lemma \ref{cancellation}.
  We get a new germ $\widetilde{f^2}$ for which $\mathcal{C}^+_{k_0}(\widetilde{f^2})=\emptyset$.
  
  We can cancel every point of label $+$ and index $k_0+1$, by moving down all points in $\mathcal{C}^+_{k_0+1}(\widetilde{f^2})$ just above points in $\mathcal{C}^-_{k_0}(\widetilde{f^2})$.
  Using the same techniques, we can consider a pseudo-gradient for which $\partial_{-+,k}(\widetilde{f^2})=0$ for all index $k$, and such that \[\partial_{++,k_0+2}(\widetilde{f^2})=\begin{pmatrix}
  0 & I_{p_{k_0+1}-p_{k_0}} \\
  \end{pmatrix}.\]
  We can then kill all points in $\mathcal{C}^+_{k_0+1}(\widetilde{f^2})$ through a non-critical path of function using \cite[Theorem 6.4]{Miln} and Lemma \ref{cancellation}.
  By induction, we can kill successively and non-critically all points in $\mathcal{C}^+_k(\widetilde{f})$.
  
  Only points of label $-$ remain, but now there is no problem to kill them all, then again thanks to Lemma \ref{Freed+}.
  We get a trivial germ that has been proved to extend non-critically, using Lemma \ref{trivial}.
  
  \end{proof}
  
  Through identifications of the critical points of $f^1$ with those of $f$ via the paths of functions between them, we also showed that there are Morse-Smale pseudo-gradients such that $f^1$ and $f$ have the same boundary operators.
  From the proof of the theorem, we even have:
\begin{proposition}   
    $f$ and $f^1$ can be given the same adapted pseudo-gradient.
    \end{proposition}

  \subsection{Surgery on manifolds with boundary}
\label{surgmb}
  In this section, we recall results of Borodzik, Nemethi and Ranicki \cite{BNR} who, among other things, study the topological changes of the level sets (with boundary) of the extension when one goes through a critical point on the boundary. Morse and Van Schaak \cite{M_VS} already knew about Morse inequalities for Morse functions on manifold with boundary. 
  The following results will only be used in the next section.
  First, we introduce some terminology taken from \cite{Kos}.
 
 Let $f$ be a Morse function on a closed manifold $M$. 
 If $a$ is a critical point of index $k$ of $f$ and of critical value $\alpha$, then $W^u(a)\cap f^{-1}(\{\alpha-\varepsilon\})$ is diffeomorphic to a standard sphere of dimension $k-1$.
We will always identify this sphere in $f^{-1}(\{\alpha-\varepsilon\})$ with its isotopy class, which does not depend on the adapted pseudo-gradient $X$.
We call it the \emph{attaching sphere} of $a$, and denote it by $\sigma_a$.
Up to renormalization of $X$, we can transport this sphere in a level set $f^{-1}(\{z\})$, with $z<\alpha$, as long as there is no critical point $a'$ of critical value $\alpha > \alpha' > z$ such that 
\[ W^s(a')\cap W^u(a) \neq \emptyset. \]
If there is no possible confusion, we will still use the notation $\sigma_a$ for any sphere which is the image of the one canonically defined in $f^{-1}(\{\alpha-\varepsilon\})$ by a reparametrization of the flow of $X$.
We will also use the notation $\sigma_a$ to denote the homotopy class or the homology class of this sphere in the level set where it can be defined.
Dually, we can define the \emph{belt sphere}, denoted by $\sigma_a^*$ which is the attaching sphere of $a$ for the Morse function $-f$.
The notation can also be extended, as for the attaching sphere.

Let now $M$ be a manifold with boundary $\partial M$.
In \cite{BNR}, Borodzik, Nemethi and Ranicki consider Morse functions $F$ on manifolds with boundary that can have non-degenerate critical points on the boundary, that is points $a\in \partial M$ such that
 \[(d_x F (a),\partial_t F (a))=(0,0),\]
  where $x$ is the coordinate on the boundary and $t$ the coordinate going inside the manifold, and $a$ is a boundary critical point of $F$ of critical value $\alpha$.
In a neighborhood of $a$ in the manifold with boundary, we have:
\[F=\alpha\pm t^2 +\sum \pm x_i^2 \] in a chart, with $t\geq 0.$
 
 Throughout the previous sections, we considered critical points for the induced Morse function $f$ which are not critical for $\tilde{f}$.
 For one of this critical point, there is a chart $(t,x)$ around it for which \[F=\pm t +\sum \pm x_i^2 \]
 in this chart.
 But, as the map $t\mapsto t^2$ is a continuous reparametrization of the map $t\mapsto t$ on $\mathbb{R}_+$ (modulo the homeomorphism of $\mathbb{R}_+$ $t\mapsto \sqrt{t}$), the \emph{topological} modifications of the level sets with boundary are the same if one considers a point $a$ critical for the function restricted to the boundary but not critical for $F$ or if one considers critical points on the boundary in the sense of Borodzik, Nemethi and Ranicki.

Let $F$ be a Morse function on $M$ and assume $M$ is of dimension $n+1$. 
  Denote by $f$ the restriction of $F$ to the boundary of $M$. 
Let $a$ be a critical point of $f$ of index $k$ and critical value $\alpha$.
 One can find the following result in \cite{BNR} directly using Lemma 2.20 or Lemma 2.21 and Theorem 2.27:

\begin{theorem}
\label{pointmoins}
If $a$ is labeled $-$ then 
$ F^{-1}(\alpha+\varepsilon)$ is homotopy equivalent to $F^{-1}(\alpha-\varepsilon)\cup_\varphi \Dkk $
 where $\varphi$ is an embedding of the belt sphere of the critical point of $F|_{\partial M}$.
\end{theorem}
	
 	Notice the following homological changes, suppose $\varphi \neq 0 $ in $H_{k-1}(F^{-1}(\alpha-\varepsilon))$:
 \[H_{k-1}(F^{-1}(\alpha+\varepsilon))\simeq H_{k-1}(F^{-1}(\alpha-\varepsilon))/(\langle\varphi \rangle)\]
	and
\[H_j(F^{-1}(\alpha+\varepsilon))\simeq H_j(F^{-1}(\alpha-\varepsilon))\] if $j< k-1$.
	If $\varphi=0$, we have:
\[H_j(F^{-1}(\alpha+\varepsilon))\simeq H_j(F^{-1}(\alpha-\varepsilon))\] if $j< k$,
	and
\[H_k(F^{-1}(\alpha+\varepsilon))\simeq H_k(F^{-1}(\alpha-\varepsilon))\oplus \mathbb{Z}.\]

We also have a description of the topological modifications of the level sets of $F$ when one passes above a boundary critical point of label $+$.
For two manifolds with boundary $M$ and $N$, the symbol $M \setminus_{(\phi, \varphi)} N$ denotes the operation of cutting off a manifold diffeomorphic to $N$ from $M$, where $\varphi$ embeds $\partial N$ in $\partial M$ and $\phi$ embeds $N$ in $M$.
We denote by $\mathring{\mathbb{D}}^d$ the open ball of dimension $d$. 
\begin{theorem}
\label{nullhomo}
If $a$ is labeled $+$ and is of index $k$ for the restriction $F|_{\partial M}$ of $F$ to the boundary $\partial M$, then 
$F^{-1}(\alpha+\varepsilon)$ is diffeomorphic to \[F^{-1}(\alpha-\varepsilon)\setminus_{(\phi,\varphi )} (\Dkk\times \mathring{\mathbb{D}}^{n-k}) \] where $\varphi$ is an embedding of $\mathbb{S}^{k-1}\times \mathring{\mathbb{D}}^{n-k}$ in $\partial M$, and $\phi$ embeds $\Dkk\times\mathring{\mathbb{D}}^{n-k}$ in $ F^{-1}(\alpha -\varepsilon)$.  
\end{theorem}

Implicitly, the theorem states that any sphere $\mathbb{S}^{k-1}\times \{\star\}$ represented by $\varphi$ bounds a ball in $F^{-1}(\alpha - \varepsilon)$, and so is trivial in homology and homotopy. 
This sphere is also the attaching sphere of $a$ for the restriction $F|_{\partial M}$.
More generally, for any representative $\tilde{f}$ of a germ, if $a$ is of label $+$, the attaching sphere $\sigma_a$ is null homotopic in some $\tilde{f}^{-1}(f(a)-\varepsilon)$ for $\varepsilon$ small enough.
Let us consider $-F$.
 Critical points of index $k$ for $F|_{\partial M}$ and labeled $+$ are turned into critical points of index $n-k$ and label $-$ for $-F|_{\partial M}$, and the same property stands for critical points labeled $-$.
  We get that $F^{-1}(\alpha-\varepsilon)$ is homotopy equivalent to 
  \[F^{-1}(\alpha+\varepsilon)\cup_\varphi \mathbb{D}^{n+1-k}\]
 where $\varphi$ is an embedding of the attaching sphere of the critical point for $-F|_{\partial M}$ corresponding to $a$, that is, the belt sphere of $a$.
 
  We have, still with $\varphi\neq 0$: 
\[H_{n-k-1}(F^{-1}(\alpha-\varepsilon))\simeq H_{n-k-1}(F^{-1}(\alpha+\varepsilon))/(\langle\varphi \rangle)\]
and 
\[H_j(F^{-1}(\alpha-\varepsilon))\simeq H_j(F^{-1}(\alpha+\varepsilon))\]
 for $j< n-k-1$. If $\varphi=0$, we have:
\[H_j(F^{-1}(\alpha-\varepsilon))\simeq H_j(F^{-1}(\alpha+\varepsilon))\]
 if $j< n-k$, and
 \[H_{n-k}(F^{-1}(\alpha-\varepsilon))\simeq H_{n-k}(F^{-1}(\alpha+\varepsilon))\oplus \mathbb{Z}.\]

  \subsection{A germ with right homological assumptions that do not extend non-critically }
  \label{theexample}
   
   We exhibit in this section a Morse germ extending a Morse function which has proerty ($\mathcal{P}$) but does not extend non-critically.
   The critical points which are not local extrema of the induced Morse function are distributed on 3 indices. 
 Before building the germ, we will need the following fact, where $\sharp$ denote the connected sum of two manifolds defined in \cite{KM}:
 \begin{lemma}
 \label{trivialsurg}
Let $M$ be a closed manifold of dimension $n$. 
 There are surgeries on the trivial homotopy class $0\in \pi_{k-1}(M)$ in $M$ such that the produced manifold $M'$ is diffeomorphic to $ M \sharp \left( \Sk \times \mathbb{S}^{n-k}\right)$. 
 \end{lemma}
 We do not prove this lemma which is a standard fact of surgery theory. See for example \cite[Example 4.17]{Ran}.
  Notice that the homotopy class of the factor $\Sk$ in $M'$ depends on the diffeomorphism between $M'$ and $M \sharp \left( \Sk \times \mathbb{S}^{n-k} \right)$. 
   Now, if we get $M'$ from $M$ as a level set by passing above a critical point $a$ of index $k$, we then will denote by $\sigma_a^+$ a sphere $\Sk \times \{\star\}$ in $M'$, or its homotopy class in $\pi_k(M')$.
 
  The Curley graph of our example that can not extend non-critically is pictured on Figure \ref{contrexemple1}. We will assume that $n$ is really large in order to have (at least) $2(k+1)< n-1$, and $k\geq 3$, in order to have $\pi_{k+1}(\mathbb{S}^k)\simeq \mathbb{Z}/(2)$.
  We will use the lemma:
\begin{lemma}
If $n$ is large enough, then $\pi_{k+1}\left( (\mathbb{S}^k\times\mathbb{S}^{n-k-1}) \sharp (\mathbb{S}^{k+1} \times \mathbb{S}^{n-k-2})\right)$ is isomorphic to $\mathbb{Z}/(2)\oplus \mathbb{Z}.$
\end{lemma}  
\begin{proof}
Denote by $M$ the manifold $(\mathbb{S}^k\times\mathbb{S}^{n-k-1}) \sharp (\mathbb{S}^{k+1} \times \mathbb{S}^{n-k-2})$.
We first show that $\pi_{k+1}(M)$ is isomorphic to $\pi_{k+1}(\mathbb{S}^k \vee \mathbb{S}^{k+1})$, where $\Sk \vee \mathbb{S}^{k+1}$ is the wedge sum of two spheres.
Denote by $W$ the manifold \[\mathbb{S}^k \times \mathbb{D}^{n-k}\cup_{\psi_1} \mathbb{D}^1\times \mathbb{D}^{n-1} \cup_{\psi_2} \mathbb{S}^{k+1}\times \mathbb{D}^{n-k-1},\]
where $\psi_1$ and $\psi_2$ are gluing maps, embedding $\{0\}\times \mathbb{D}^{n-1}$ and respectively $\{1\}\times \mathbb{D}^{n-1}$ into $\mathbb{S}^k \times \mathbb{S}^{n-k-1}$ and respectively $\mathbb{S}^{k+1} \times \mathbb{S}^{n-k-2}$.
The manifold $W$ is a manifold with boundary, and its boundary is diffeomorphic to $M$.
As a CW-complex, $W$ is obtained from $M$ by gluing cells of dimensions $n$, $n-k$ and $n-k-1$.
We have an injection $M \hookrightarrow W$. 
As $W$ is obtained from $M$ by adding high dimensional cells, when $k<<n$, it induces equality for $j$-dimensional homotopy groups with $j$ small with respect to $n$.
See for example \cite[Cor. 4.12, p. 351]{Hat}.
We have a sequence of injections $\Sk \vee \mathbb{S}^{k+1} \hookrightarrow M \hookrightarrow W$.
Notice that $W$ retracts on $\Sk \vee \mathbb{S}^{k+1}$.
Thus, $\pi_j(W,\Sk \vee \mathbb{S}^{k+1})=0$ for all $j$.
The long sequence of homotopy groups and the fact that $k<<n$ shows that we have an isomorphism:
\[0 \to \pi_{k+2}(W,M) \to \pi_{k+1}(M, \Sk \vee \mathbb{S}^{k+1}) \to 0\]
 As $\pi_{k+2}(W,M)=0$ since $k<<n$, we then have $\pi_{k+1}(M,\Sk \vee \mathbb{S}^{k+1})=0$.
Hence, $\pi_{k+1}(\Sk \vee \mathbb{S}^{k+1})$ and $\pi_{k+1}(M)$ are isomorphic.

Finally, using the long exact sequence of homotopy groups for the sequence 
\[\emptyset \hookrightarrow \Sk \vee \mathbb{S}^{k+1} \hookrightarrow \Sk \times \mathbb{S}^{k+1},\] and for $k$ higher than $2$, we have that $\pi_{k+1} ( \Sk \times \mathbb{S}^{k+1})$ and $\pi_{k+1}(\Sk \vee \mathbb{S}^{k+1})$ are isomorphic.
But $\pi_{k+1}(\Sk \times \mathbb{S}^{k+1} )$ is isomorphic to $\pi_{k+1}(\Sk) \oplus \pi_{k+1}(\mathbb{S}^{k+1})$, which is isomorphic to $\mathbb{Z}/(2)\oplus \mathbb{Z}$.
\end{proof}
  
  \begin{figure}[!ht]
  \begin{center}
  \includegraphics[scale=0.6]{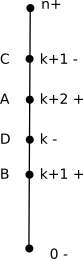}
  \caption{Curley graph of some germ that does not extend non-critically}
  \label{contrexemple1}
  \end{center}
  \end{figure}

We build the germ step by step, starting from a height function and deforming the induced Morse functions through generic paths.
We then get a germ using Lemma \ref{usefullemma} to get the labels we desire.
In order to avoid confusion, we will not indicate the time-dependency of the critical points and their respecting attaching and belt spheres.
 
  Consider a Morse function whose Reeb graph is the one of Figure \ref{contrexemple1b}.
  The level sets right above $b$ and $d$ respectively are obtained as in Lemma \ref{trivialsurg}.
  The homotopy classes of $\sigma_b$ and $\sigma_d$ are zero in their respective level sets, and the homotopy classes of $\sigma_a$ (resp. $\sigma_c$) in its level set is the one of $\sigma_b^+$ (resp. $\sigma_d^+$). 
   This Morse function is a Morse function that one gets from the height Morse function on the sphere $\Sn$ after the births of the two critical pairs $(a,b)$ and $(c,d)$.

\begin{figure}[!ht]
  \begin{center}
  \includegraphics[scale=0.6]{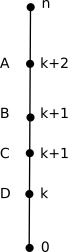}
  \caption{}
  \label{contrexemple1b}
  \end{center}
  \end{figure}  
  We first use a path of functions making $b$ move below $c$ and $d$.
  It is possible since $\sigma_b$ is null homotopic in the level set below $b$.
  We get a new Morse function whose Reeb graph is pictured on Figure \ref{function2}.
  We denote by $M_1$ the level set below $c$.
  It is diffeomorphic to ${(\Sk \times \mathbb{S}^{n-k-1}) \sharp (\mathbb{S}^{k+1}\times \mathbb{S}^{n-k-2})}$.
  As $k<<n$, the group $\pi_{k+1}(M_1)$ is isomorphic to $\mathbb{Z}/(2)\oplus \mathbb{Z}$.

\begin{figure}[!ht]
\begin{center}
\includegraphics[scale=0.6]{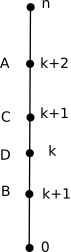}
\caption{}
\label{function2}
\end{center}
\end{figure}  
  
  Using $2(k+1)+1 < n-1$ and Whitney embeddings theorems, there is a representative $S$ of the class $(1,0)$ which is an embedded sphere into $M_1$.
  Moreover, since $k<<n$ and using \cite[Lemma 4.6]{Miln}, we can assume that $\sigma_c$ and $S$ are disjoint.
  Thus, using a reparametrization of the pseudo-gradient, there is a $(k+1)$-dimensional sphere $S'$ embedded into the level set below $a$ such that $S'$ is sent to $S$ by the flow of the pseudo-gradient.
  Notice that the level set in $\Sn$ below $a$ is diffeomorphic to $\mathbb{S}^{k+1}\times \mathbb{S}^{n-k-2}$.
 In this level set, $S'$ is null-homotopic. 
 As we kill the homotopy class represented by $\Sk \times \{\star\}$ in $M_1$ when passing above $c$, we also kill the class of $S$, leading to $S'=0$ in homotopy.
 Finally, $\sigma_a$ is isotopic to $\sigma_a + S'$.
 
 Let $X'$ be an adapted pseudo-gradient for which the attaching sphere of $a$, denoted $\sigma'_a$, is $\sigma_a+S'$, and such that $\sigma'_a$ and the belt sphere of $c$ are still disjoint. 
 The last assumption is possible due to the fact that the belt sphere of $c$ is null-homotopic in the level set below $a$, for any adapted pseudo-gradient.
There is no obstruction for $a$ to move below $c$ along the descending disk of $a$ for $X'$, see for example \cite[Lemma 2.1]{Laud5}.
      By doing that and defining the germs with labels of Figure $\ref{contrexemple1}$, we get, with the use of Lemma \ref{usefullemma}, a germ $\tilde{f}$. We now show:
      \begin{proposition}
      $\tilde{f}$ can not extend non-critically.
      \end{proposition}
      In the proof, we keep using the same notation introduced just above.
      \begin{proof} 
     Assume $\tilde{f}$ has a non-critical extension $F$.
      By the description given in Section \ref{surgmb}, the only topological modifications of the level sets of $F$ are given by the ones described in Theorems \ref{nullhomo} and \ref{pointmoins}, with respect to the labels of the critical points of $f$.
 Thus, all the level sets $F^{-1}(\alpha)$ for $\alpha \in ]f(d),f(a)[$ are diffeomorphic.
We then show that these topological modifications can not lead to a manifold with boundary $F^{-1}(f(a)-\varepsilon)$ in which $\sigma_a$, the attaching sphere of $a$, is null-homotopic, which is absurd.  
We successively have, using results of Section \ref{surgmb}:
\begin{itemize}
\item The level set of the extension above the minimum must be diffeomorphic to a disk $\mathbb{D}^n$.
\item The level set of the extension above $b$, denoted $\Sigma$, must be diffeomorphic to 
\[\mathbb{D}^{n} \setminus_{(\phi, \varphi)}(\mathbb{D}^{k+1}\times \mathring{\mathbb{D}}^{n-(k+1)}) \]
 where $\varphi$ embeds $\mathbb{S}^{k}\times\{\star\}$ in $\mathbb{S}^{n-1}=\partial \Dnn$.
We have that $H_j(\Sigma)=0$ for all $1 \leq j \leq k+1$ by the description of the modifications of the homology groups of the level sets of the extension, given in Section \ref{surgmb}, and the fact that $k+1 < \frac{n}{2}$. 
  We then have $\pi_{j}(\Sigma)=0$ for all $1 \leq j \leq k+1$ using the Hurewicz theorem on higher homotopy groups of a CW-complex.
  In particular, $\pi_{k+1}(\Sigma)=0$.
 \item The level set of the extension above $d$ is homotopy equivalent to $\Sigma\cup_{\varphi}\mathbb{D}^k $, where $\varphi$  embeds a sphere $\mathbb{S}^{k-1}$ in the boundary of $\Sigma$.
 Moreover, $\varphi$ is null-homotopic in $\Sigma$ and bounds a ball $\mathbb{D}_2^{k}$, as $\pi_k(\Sigma)=0$.
 We have that $S$ in this level set is the non-null ($k+1$)-dimensional homotopy class of the complex $\mathbb{D}_2^{k}\cup_{\varphi}\mathbb{D}^k$ homotopy equivalent to a sphere $\mathbb{S}_2^k$.
  It remains non-null homotopic in $\Sigma\cup_{\varphi}\mathbb{D}^k $ which is homotopy equivalent to $\Sigma\vee \mathbb{S}_2^k$.
  Hence, the homotopy class of $\sigma'_a$ is not null-homotopic in the level set of the extension. 
  We can now use Theorem \ref{nullhomo}, stating that it must be null homotopic if the germ extends non critically.

\end{itemize}      
\end{proof}
We now prove:
\begin{lemma}
The germ $\tilde{f}$ has property ($\mathcal{P}$).
\end{lemma}
\begin{proof}
     From the Curley graph pictured on Figure \ref{contrexemple1} and the construction of $\tilde{f}$, we see that there is a Morse-Smale psueod-gradient adapted to $f$ for which the boundary operator $\partial$ is such that:
     \begin{itemize}
     \item $\partial a = \pm b$
     \item $\partial c = \pm d$
     \end{itemize}
     In fact, as we are on the sphere whose homology vanishes in degree $k$ and $k+1$, and regarding the critical values of the points, we see that the boundary operator of any pseudo-gradient $X$ adapted to $f$ is as $\partial$. 
     With the labels, we see that $\tilde{f}$ has property ($\mathcal{P}$).
\end{proof}     
     
     As a conclusion we have proved:
     \begin{theorem}
     There are Morse germs which have property ($\mathcal{P}$) but that do not extend non-critically.
     \end{theorem}
  
Finally, in order to give a necessary and sufficient condition for a generic germ to get a non-critical extension, we need to care about the kind of obstruction displayed by the previous example, that is, of homotopical nature.
 Unfortunately, it seems to be very difficult to have a precise description of the homotopy classes of the attaching spheres of the critical points of a Morse function $f$ defined on $\Sn$.

  \section{Two indices}
\label{twoindices} 
We consider in this section some special case of germs for which Theorem \ref{th1} gives a sufficient condition when $n \geq 6$.
 For this class of germs, we also give a computable arithmetical condition assuming that points of label $+$ are all above points of label $-$ of same index.
 We use the notations of Theorem \ref{th1}.
 Let $\tilde{f}$ be a Morse germ, $G_k(\tilde{f})$ denotes the group introduced in Definition \ref{group}.
  If $M_k \in G_k(\tilde{f})$, then $N_k$ is its down-left submatrix. 
Recall $p_k$ denotes the number of critical points of index $k$ and label $+$.
 
 \begin{theorem}[Sufficiency when we have two indices]
 Let $n \geq 6$.
 Let $\tilde{f}$ be a non-critical germ such that $f$ is a Morse function with one local maximum, one local minimum, and such that the indices of all the other critical points can take two values, $k$ and $k+1$, with $2\leq k\leq n-3$.
 Let $X$ be an adapted pseudo-gradient which is Morse-Smale. 
 The germ $\tilde{f}$ extends non-critically if and only if $p_{k+1}=p_k$ and there are matrices $M_{k+1}$ in $G_{k+1}(\tilde{f})$ and $M_k$ in $G_k(\tilde{f})$ such that 
 \[\partial_{++,k+1}-\partial_{+-,k+1}N_{k+1}\] is invertible and 
 \[\partial_{-+,k+1}+N_k\partial_{--,k+1}-\partial_{++,k+1}N_{k+1}-N_k\partial_{+-,k+1}N_{k+1}=0,\]
  where $N_{k+1}$ is the down-left submatrix of $M_{k+1}$ (resp. $M_k$). 
 \end{theorem}
 \begin{proof}
 The fact that it is necessary is a direct consequence of Theorem \ref{th1} in this special case.
 
 We show it is sufficient. 
 Let $\tilde{f}$ be such a germ and $X$ a Morse-Smale pseudo-gradient adapted to $f$, with the hypotheses of the theorem right above.
 We do all the handle slides corresponding to the matrices $M_{k+1}$ and $M_k$.
 We get a new pseudo-gradient $X^1$ such that $\partial^1_{++,k+1}$ is an invertible matrix and $\partial^1_{-+,k+1}=0$.
 Up to changing the orientations of the unstable manifolds of the critical points, we assume that $\partial^1_{++,k+1}$ is a matrix in $SL_{p_{k+1}}(\mathbb{Z})$.
 
 By a generic path of functions, we can move down all points of index $k$ and label $+$ below all the other points of index $k$, just above the minimum, such that the path, as a function on $\Sn \times [0,1]$, has no critical point.
 We can do that by keeping the same pseudo-gradient, identifying the critical points through the path, see \cite[Corollary 2.2]{Laud5}.
 
 Moreover, as $\partial^1_{-+,k+1}=0$, there is no gradient lines connecting points in $\mathcal{C}^+_{k+1}(\tilde{f})$ and points in $\mathcal{C}^-_k(\tilde{f}$.
  We can then move all points of index $k+1$ and label $+$ along their descending disks below all points of label $-$ and index $k$.
 We move them just above the points of index $k$ and label $+$.
 We get a Morse function $f^2$ which forms with $f$ the endpoints of a path of functions which has no critical points. 
 Using Theorem \ref{Freed+} and the fact that $SL_{p_{k+1}}(\mathbb{Z})$ is generated by transvections, we can make all the handle slides we want between points of index $k+1$ and label $+$ to get a pseudo-gradient $X^2$ with $\partial^2_{++,k+1}=I_{p_{k+1}}$. 
 We can kill non-critically all points of label $+$ but the maximum, due to $2\leq k \leq n-2$, using \cite[Theorem 6.4]{Miln} about cancellation of critical points with right indices.
  We get a Morse function $f^3$ on the sphere whose critical points which are not extrema all are of label $-$.
 In terms of Morse homology, we have a complex:
 \[0\rightarrow \mathbb{Z}\mathcal{C}^-_{k+1}(\widetilde{f^3}) \rightarrow \mathbb{Z}\mathcal{C}^-_{k}(\widetilde{f^3})\rightarrow 0,\]
 which is acyclic since we are on the sphere and that $2\leq k \leq n-3$.
 It implies that $\partial^3_{--,k+1}$ is an invertible matrix. 
 There again, using Theorem \ref{Freed+} -- but now with points of label $-$ --, we can assume that we have a pseudo-gradient vector field $X^4$ adapted to $f^3$ for which $\partial^4_{--,k+1}=I_{q_{k+1}}$.
 Using \cite[Theorem 6.4]{Miln}, we can kill all points of label $-$.
 We get a trivial germ, and we know from Theorem \ref{trivial} that we can extend it without critical points.
 The theorem is then proved.
 \end{proof}

Let $\tilde{f}$ be a Morse germ.
 From now on, we assume that all critical points of $f$ of index $k+1$ and label $+$ are above critical points of label $-$. Thus $G_{k+1}(\tilde{f})$ can be identified to $\Hom (\mathbb{Z}\mathcal{C}^+_{k+1}(\tilde{f}), \mathbb{Z}\mathcal{C}^-_{k+1}(\tilde{f}))$, via the identification of an element $M_{k+1}$ in $G_{k+1}(\tilde{f})$ to the down left submatrix of its restriction to $\mathbb{Z}\mathcal{C}(f)$.
 
In the rest of the article we will prove the following theorem:

\begin{nntheo}[Theorem \ref{2indices}]
With these hypotheses, $\tilde{f}$ extends non-critically to a function $F$ on the disk $\Dn$ if and only if the two following conditions hold:
\begin{itemize}
\item  $p_k$ is the rank of $\mathbb{Z}\mathcal{C}_{k+1}^+(\tilde{f})$, 
\item $\det (\partial_{++,k+1}) \equiv \pm 1 ~[d_1(\partial_{+-,k+1})].$
\end{itemize}
\end{nntheo}

Using Theorem \ref{th1}, $\tilde{f}$ extends non-critically if and only if there is a matrix $N$ such that $\partial_{++,k+1}+\partial_{+-,k+1}N$ is invertible (it implies in particular that $\partial_{++,k+1}$ is square and thus that $p_{k+1}=p_k$).
We then are interested in the following problem.

Let $B$ be an integral matrix of size $p\times p$ and $C$ be an integral matrix of size $p\times r$ where $r$ is any integer, even $0$, in which case the matrix is empty. 

\begin{center}\textbf{Problem $\Omega$:} \emph{When is there a matrix $N$ such that $B+CN$ is invertible ?}
\end{center}
The problem is invariant under the following actions on $(B,C)$.
 The matrices $U$, $V$ and $W$ in the following are understood to be matrices in, respectively, $SL_p(\mathbb{Z})$, $SL_p(\mathbb{Z})$ and $SL_r(\mathbb{Z})$:
 \begin{itemize}
 \item $(B,C)\to (UB,UC)$,
 \item $(B,C) \to (BV,C)$,
 \item $(B,C)\to (B,CW)$. 
 \end{itemize}

This invariance under unimodular matrices' actions allows us to consider the Smith normal form of $C$. 
We recall the definitions below, but we need to introduce the notion of determinantal divisors first.
 See \cite[p.25]{Morr}.
\begin{definition}[Determinantal divisors]
Let $M$ be an integral matrix of size $p\times q$. 
For an integer $r$, let $\mathcal{I}_k(r)$ be the set of $k$-tuples of $\{1,...,r\}$.
For $\omega \in \mathcal{I}_k(p)$ and $\tau \in \mathcal{I}_k(q)$, let $M(\omega , \tau)$ be the submatrix of $M$ whose column indices are in $\tau$ and row indices are in $\omega$. 
Let $m(\omega ,\tau)$ be the determinant of $M(\omega , \tau )$.
The \emph{$k$-th determinantal factor} $d_k(M)$ is the g.c.d. of the family $(m(\omega , \tau ))_{\omega \in \mathcal{I}_k(p), \tau \in \mathcal{I}_k(q)}$.
We set $d_0(M):=1$ as a convention, and $d_r(M)=0$ for $r$ higher than $\max(q,p)$.
\end{definition}

The following theorem-definition is taken from \cite[Theorem II.9]{Morr}: 
\begin{theorem}[Existence and definition of the Smith Normal Form (S.N.F.) of an integral matrix ]
Let $M$ be an integral matrix of size $p\times q$. There are unimodular matrices $U$ and $V$ such that the matrix $UMV$ is of the form:
\[\begin{pmatrix}
s_1(M) & 0      & \ldots & 0       \\
0      & s_2(M) &        & 0       \\
\vdots & 0      & \ddots & \vdots  \\
0      & 0      &   0    & s_q(M)  \\
0      & \ldots & \ldots & 0       \\
\vdots &        &        & \vdots  \\
0      & \ldots  & \ldots & 0       \\
\end{pmatrix}\]
if $p \leq q$ or 
\[\begin{pmatrix}
s_1(M) & 0      & \ldots & 0      & 0      & \ldots & 0      \\
0      & s_2(M) &        & 0      & \vdots &        & \vdots \\
\vdots & 0      & \ddots & \vdots & \vdots &        & \vdots \\
0      & 0      &   0    & s_p(M) & 0      & \ldots & 0      \\
\end{pmatrix}\]
if $q\leq p$, where $s_k(M):= \frac{d_k(M)}{d_{k-1}(M)}$ for $k\geq 1$.  
Moreover, $s_k (M)$ divides $s_{k+1}(M)$ for all $k$.
We say that $UMV$ is the (unique) \emph{Smith Normal Form} (abridged S.N.F.) of $M$.
\end{theorem}

The proof of the existence of the unimodular matrices $U$ and $V$ can be found in \cite[Theorem II.9]{Morr}.
From the last definition, we have the following proposition:

\begin{proposition}
Let $p$ and $q$ be two integers.
A matrix $B$ of size $p\times q$ is surjective if and only if $d_p(B)=1$. 

\end{proposition}

\begin{proof}
$B$ is surjective if and only if its Smith normal form is surjective. 
It is the case if and only if $s_j(B)=1$ for all $j$ from $1$ to $p$.
It implies in particular that $d_j(B)=1$ for all $j$ between $1$ and $p$.
\end{proof}

We also recall the following proposition:
\begin{theorem}[Hermite normal form of an integral matrix]
Let $B$ be an integral matrix in $\Mp$. There is a unimodular matrix $U$ such that $BU$ is upper triangular.
\end{theorem}

See \cite[Theorem II.2]{Morr} for a proof.
The theorem also gives relations between the coefficients of $BU$, but we do not need that here. 

If $M$ is a matrix in $\Mp$, we denote by $M_k$ its $k$-th row, and we use the notation \[M= \begin{pmatrix}
M_1 \\
\vdots \\
M_p \\
\end{pmatrix}.\]

We give a first answer to Problem $\Omega$ which will be improved later:
\begin{lemma}
\label{unim}
Let $B$ and $C$ be two integral matrices, both in $\Mp$.
Assume $B$ is upper triangular and $C$ is in S.N.F.
Denote by $b_j$ (resp. $c_j$) the $j$-th diagonal coefficient of $B$ (resp. $C$).
Thus $d_1(C)=c_1$.
 There is a matrix $N$ such that $B+CN$ is invertible in $\Mp$ if and only if \[\det (B) \equiv \pm 1 ~[c_1],\]
and \[\gcd (b_j, c_j)=1 \text{ for all }j.\]
\end{lemma}

\begin{remark}
The condition 
\[ \gcd (b_j, c_j)=1 \text{ for all }j,\]
is equivalent to $d_p(B~C)=1$, see Proposition \ref{simplprop} below, where $(B~C)$ is the concatenation of $B$ and $C$.
\end{remark}

\begin{proof}
If there is such a matrix $N$, then $\det(B+CN)= \pm 1$, and projecting the matrices to $\mathcal{M}_p(\mathbb{Z}/(c_1))$ leads to the equation \[\det (B) \equiv \pm 1 ~[c_1].\]
Notice that the $(i,j)$ coefficient of $CN$ is given by $N_{i,j}c_i$. 
Recall that $c_k$ divides $c_{k+1}$ for all $k$.
Then, in $\mathbb{Z}/(c_2)$, all rows of $CN$ project to $0$ except the first one.
Thus $B+CN$ projects to a matrix with coefficients in $\mathbb{Z}/(c_2)$ which is
\[\begin{pmatrix}
 c_1N_1+B_1 \\
 B_2 \\
 \vdots \\
 B_p \\
\end{pmatrix}.\]
Using the linearity of the determinant with respect to the first row, we get: 
\[\pm 1 \equiv \det (B) + c_1 \det \begin{pmatrix}
N_1 \\
B_2 \\
\vdots \\
B_p \\
\end{pmatrix}  ~~ [c_2],\]
But $B$ is upper triangular, so 
\[\det \begin{pmatrix}
N_1 \\
B_2 \\
\vdots \\
B_p \\
\end{pmatrix} = N_{1,1}b_2...b_p \] 
and
\[\det (B)= b_1...b_p.\]
If $B+CN$ is in $GL_p(\mathbb{Z})$, we then have that
\[\pm 1 \equiv b_2...b_p(b_1+c_1 N_{1,1}) ~~ [c_2], \]
implying that $b_j$ is coprime with $c_2$ for any $j$ between $2$ and $p$.
In the same way, we have that 
\[CN \equiv \begin{pmatrix}
c_1 N_1 \\
\vdots \\
c_k N_k \\
0 \\
\vdots \\
0 \\
\end{pmatrix}~~ [c_{k+1}].\]
We get:
\[ \pm 1 \equiv b_{k+1}...b_p \det \begin{pmatrix}
B_1 + c_1 N_1  & \\
\vdots &  \\
B_k+c_k N_k  & \\
0 & I_{p-k} \\
\end{pmatrix}
 ~~ [c_{k+1}].\]
It shows that $b_{k+1}$ and $c_{k+1}$ are coprime.

The sufficient part reduces to some computations and use of B\'ezout's theorem.
We know that the problem is unchanged by the action of $GL_p(\mathbb{Z})^3$ previously described, but also of course, by the action $(B,C) \mapsto (B+CY,C)$ for any matrix $Y$ in $\Mp$. 
Adding the two invariance properties together, we can consider any matrix $BU+CY$ instead of $B$, where $Y$ is in $\Mp$ and $U$ is invertible. 
As $b_k$ and $c_k$ are coprime, there are integers $u_k$ and $v_k$ for each $k$ such that $u_k b_k+v_k c_k =1$.
We say that a matrix is \emph{upper unitriangular} if its diagonal coefficients are $1$ and if it is an upper triangular matrix.
Consider an invertible matrix $U$ which is upper unitriangular and whose coefficients above the diagonal are all $0$ except on the first row for which it is defined to be:
\[U_{1,j}=-u_1 B_{1,j}.\]
Let also $Y$ be an upper triangular matrix whose diagonal elements are all $0$ and whose coefficients above the diagonal are all zero except on the first row for which it is defined to be:
\[Y_{i,j}=-v_1 B_{1,j}, \]
for $j\geq 2$.

Then, we can see that all coefficients on the first row  and off the diagonal of $BU+CY$ are zero. 
The coefficient on the diagonal still being $b_1$.
We then can iterate the same operation on the $j$-th row for $j$ from $2$ to $p$ to get a diagonal matrix being
\[\begin{pmatrix}
b_1    & 0      & \ldots & 0 \\
0      & \ddots &        & \vdots \\
\vdots &        & \ddots & 0 \\   
0      & \ldots & 0      & b_p \\
\end{pmatrix}.\]%
We then assume that $B$ is such a diagonal matrix.
We now do the following operations (multiplication of $B$ on the right by a unimodular matrix, multiplication of $B$ and $C$ on the left by a unimodular matrix, adding a matrix $CY$ to $B$):%
\begin{itemize}
\item $BX + CY \rightarrow B'$ where $X$ is the matrix 
\[X:=\begin{pmatrix}
1      & 0      & \ldots & \ldots & 0      \\
0      & \ddots & \ddots &        & \vdots \\ 
\vdots & \ddots & \ddots & 0      & \vdots      \\ 
\vdots &        & \ddots & 1      & 0    \\
0      & \ldots & \ldots & u_p      & 1      \\
\end{pmatrix}\] 
and $Y$ is the matrix 
\[Y:= \begin{pmatrix}
0      & \ldots & \ldots & \ldots & 0      \\
\vdots & \ddots &        &        & \vdots \\
\vdots &        & \ddots &        & \vdots     \\
\vdots &        &        & \ddots & 0    \\
0      & \ldots & \ldots & v_p    & 0      \\
\end{pmatrix}.\]
$B'$ then corresponds to the matrix 
\[B':= \begin{pmatrix}
b_1    & 0      & \ldots & \ldots & 0      \\
0      & \ddots & \ddots &        & \vdots \\
\vdots & \ddots & \ddots & 0      & 0      \\
\vdots &        & \ddots & b_{p-1}& 0      \\
0      & \ldots & \ldots & 1      & b_p    \\
\end{pmatrix}.\]
\item We now do row operations, keeping in mind that row operations are made on both matrices $B$ and $C$. 
We add $-b_{p-1}$ times the $p$-th rows to the $p-1$-th ones. 
We change $B$ into $B'$ where we have:
\[ B':= \begin{pmatrix}
b_1    & 0      & \ldots & \ldots  & \ldots      & 0      \\
0      & \ddots & \ddots &         &             & \vdots \\
\vdots & \ddots & \ddots & 0       & 0           & 0      \\
\vdots &        & \ddots & b_{p-1} & 0           & 0      \\
\vdots &        &        & 0       & 0           & -b_p b_{p-1} \\
0      & \ldots & \ldots & 0       & 1           & b_p    \\
\end{pmatrix}\]
and for $C$:
\[C':= \begin{pmatrix}
c_1    & 0      & \ldots & \ldots      & 0      \\
0      & \ddots & \ddots &             & \vdots \\
\vdots & \ddots & \ddots & 0           & 0      \\
\vdots &        & \ddots & c_{p-1}     & -b_{p-1}c_p  \\
0      & \ldots & \ldots & 0           & c_p    \\
\end{pmatrix}.\]

\item We can do operations on the columns of $B$ and add $b_r$ times the $r-1$-th column to the $r$-th one to cancel $b_r$. 
We can then switch these two last columns and get for the matrix $B$:
\[ B':= \begin{pmatrix}
b_1    & 0      & \ldots & \ldots      & 0      \\
0      & \ddots & \ddots &             & \vdots \\
\vdots & \ddots & \ddots & 0           & 0      \\
\vdots &        & \ddots & -b_{p-1}b_p & 0      \\
0      & \ldots & \ldots & 0           & 1      \\
\end{pmatrix}.\]
Moreover, as $c_{r-1}$ divides $c_r$, we can cancel the $(r-1,r)$ coefficient of $C$ by column operations to get back to:
\[ C':= \begin{pmatrix}
c_1     & 0      & \ldots & 0      \\
0       & c_2    & \ddots & \vdots \\ 
\vdots  & \ddots & \ddots & 0      \\ 
0       & \ldots & 0      & c_p    \\
\end{pmatrix}.\] 

\item We have that $b_p$ and $c_p$ are coprime and that $c_{p-1}$ divides $c_p$. Thus $b_p$ and $c_{p-1}$ are also coprime.
The hypotheses implies then that $-b_{p-1}b_{p}$ is relatively prime with $c_{p-1}$. 
The same operations can be carried out on the submatrices of $B$ and $C$ constituted by the $p-1$-th columns and $p-1$-th rows. 
An induction then leads to a matrix $B$ of the following form: 
\[ B':= \begin{pmatrix}
\pm \det(B) & 0      & \ldots & 0      \\
0       & 1      & \ddots & \vdots \\ 
\vdots  & \ddots & \ddots & 0      \\ 
0       & \ldots & 0      & 1      \\
\end{pmatrix},\]
and the matrix $C$ remains unchanged.
\end{itemize}

As \[\det (B) \equiv \pm 1 ~~ [c_1],\] we can finally substract to $B'$ the matrix $CX'$ where coefficients of $X'$ are zero every where but in position $(1,1)$ for which it is $\pm\frac{\det(B) \pm 1}{c_1}$. 
We then get the identity matrix and prove the theorem.
\end{proof}

We now refine the condition "\emph{$b_j$ and $c_j$ are coprime for all $j$}".

\begin{proposition}
\label{simplprop}
Let $B$ and $C$ be two matrices in $\Mp$, such that $B$ is upper triangular and $C$ is in S.N.F..
The $j$-th diagonal coefficient of $B$ (resp. $C$) is denoted $b_j$ (resp. $c_j$).
The two following properties are equivalent:
\begin{itemize}
 \item $b_j$ and $c_j$ are coprime for all $j$,
 \item \[d_p \left(\begin{pmatrix}
B & C \\
\end{pmatrix}\right)=1.\]
Equivalently, the matrix $\begin{pmatrix}
B & C \\
\end{pmatrix}$ is surjective.
\end{itemize}
\end{proposition}

\begin{proof}
If $b_j$ and $c_j$ are coprime for all $j$, then we saw in the proof of Lemma \ref{unim} that the matrix $\begin{pmatrix}
B & C \\
\end{pmatrix}$ can be put in the form $\begin{pmatrix}
B' & C' \\
\end{pmatrix}$ where we have 
\[ B':= \begin{pmatrix}
\pm \det(B) & 0      & \ldots & 0      \\
0       & 1      & \ddots & \vdots \\ 
\vdots  & \ddots & \ddots & 0      \\ 
0       & \ldots & 0      & 1      \\
\end{pmatrix},\]
and $C'$ is the Smith normal form of $C$, modulo change of basis modifying $B$ and $C$ separately.
As $c_j$ and $b_j$ are coprime and $c_j$ divides $c_{j+1}$ for all $j$, we have that $c_1$ and $\det(B)$ are coprime. 
Thus, 
$\begin{pmatrix}
B' & C' \\
\end{pmatrix}$
 is surjective, and 
$\begin{pmatrix}
B & C \\
\end{pmatrix}$ 
as well.

We prove the other ?sense?.
We can also assume that $B$ is upper triangular and $C$ in Smith normal form. 
As earlier, we use the notation $b_j$ (resp. $c_j$) for the diagonal coefficients of $B$ (resp. $C$). 
Assume that for some $\ell$, $b_{\ell}$ and $c_{\ell}$ are not coprime.
As $c_{\ell}$ divides $c_m$ for $m \geq \ell$, the integers $b_{\ell}$ and $c_m$ are not coprime.
Let $D$ be a submatrix of size $p \times p$ of $\begin{pmatrix}
B & C \\
\end{pmatrix}$ whose determinant is not $0$.
Then, for all $j$ between $1$ and $p$, the matrix $D$ must have the $j$-th column of $B$ or the $j$-th column of $C$ as one of its column. 
If not, we would have $\det (D) =0$ as $B$ is upper triangular and $C$ is diagonal.
In particular, $D$  must have the $\ell$-th column of $B$ or the $\ell$-th column of $C$ as one of its column.
Thus, $\det (D)$ is divisible by the g.c.d. of $b_{\ell}$ and $c_{\ell}$, which is not $1$.
Hence, all submatrix of $\begin{pmatrix}
B & C \\
\end{pmatrix}$ is divisible by this g.c.d. which implies that $d_p \left(\begin{pmatrix}
B & C \\
\end{pmatrix}\right)$ is not $1$ and finally that $\begin{pmatrix}
B & C \\
\end{pmatrix}$ is not surjective.

\end{proof}

We now go back to the problem of extending non-critically a germ $\tilde{f}$ such that $f$ has only one maximum, one minimum and critical points of indices $k$ and $k+1$ where $2\leq k \leq n-2$. 
We also assume that all points of index $k$ (resp. $k+1$) and label $+$ have critical values higher than those of points of index $k$ (resp. $k+1$) and label $-$. 
Thus, for $j \in \{k,k+1\}$, the group $G_{j}(\tilde{f})$ can be identified to $\Hom (\mathbb{Z}\mathcal{C}^+_{j}(\tilde{f}), \mathbb{Z}\mathcal{C}^-_{j}(\tilde{f}))$, through the identification of an element $M\in G_{j}(\tilde{f})$ with its down-left submatrix $N$.

We are in position to prove the main theorem of the section, restated below:
\begin{theorem}
\label{2indices}
  $\tilde{f}$ extends non-critically to a function $F$ on the disk $\Dn$ if and only if 
  $p_k$ is the rank of $\mathbb{Z}\mathcal{C}_{k+1}^+(\tilde{f})$ and that 
\[\det (\partial_{++,k+1}) \equiv \pm 1 ~[d_1(\partial_{+-,k+1})].\]
\end{theorem}

Notice that the boundary operator $\partial_j$ is non-zero only for $j=k+1$.

\begin{proof}
Assume that the germ extends non-critically.
 Using Theorem \ref{th1}, let \[\partial'_{++,k+1}:=\partial_{++,k+1}+\partial_{+-,k+1}N_{k+1}\] be a matrix such that the chain complex
 \[0 \to \mathbb{Z}\mathcal{C}^+_{k+1}(\tilde{f}) \overset{\partial'_{++,k+1}}{\longrightarrow} \mathbb{Z}\mathcal{C}^+_{k+1}(\tilde{f}) \to 0\] is acyclic.
Then $\partial'_{++,k+1}$ must be an isomorphism implying that $p_k$ is the rank of $\mathbb{Z}\mathcal{C}^+_{k+1}(\tilde{f})$
and that there is a matrix $N_{k+1}$ such that $\partial_{++,k+1}+\partial_{+-,k+1}N_{k+1}$ is invertible.
Then \[\det (\partial_{++,k+1}) \equiv \pm 1 ~[d_1(\partial_{+-,k+1})].\]

Assume now that $p_k$ is the rank of $\mathbb{Z}\mathcal{C}_{k+1}^+(\tilde{f})$ and that 
\[\det (\partial_{++,k+1}) \equiv \pm 1 ~[d_1(\partial_{+-,k+1})].\]
From Proposition \ref{simplprop}, we know that there is a matrix $N_{k+1}$ such that $\partial_{++,k+1}-\partial_{+-,k+1}N_{k+1}$ is invertible if and only if the matrix  $\begin{pmatrix}
\partial_{++,k+1} & \partial_{+-,k+1}
\end{pmatrix}$ is surjective.
The complex 
\[ 0 \to \mathbb{Z}\mathcal{C}_{k+1}(f) \overset{\partial_{k+1}}{\rightarrow}  \mathbb{Z}\mathcal{C}_{k}(f) \to 0 \] is acyclic, since it comes from a chain complex giving the homology of the sphere $\Sn$.
Thus, $\partial_{k+1}$ is invertible, and in particular surjective.
Hence, the matrix $\begin{pmatrix}
\partial_{++,k+1} & \partial_{+-,k+1}
\end{pmatrix}$ is surjective as well.
Let now $N_{k+1}$ be such that $\partial_{++,k+1}+\partial_{+-,k+1}N_{k+1}$ is invertible.
Make all the handle slides corresponding to $N_{k+1}$.
Using Proposition \ref{Freed+}, we can make handle slides between points of index $k+1$ and label $+$ to have a pseudo-gradient $X^1$ for which $\partial^1_{++,k+1}$ is the identity. 
As $G_k(\tilde{f})$ is maximal, we can make handle slides of points of index $k$ and label $+$ over points of index $k$ and label $-$ corresponding to the matrix $N_k = - \partial^1_{-+,k+1}$. 
It leads to a pseudo-gradient $X^2$ for which $\partial^2_{++,k+1} = I_{p_k}$ and $\partial^2_{-+,k+1}=0$.
Finally, we can use Theorem \ref{2indices} and conclude.
\end{proof}

\begin{acknow}
I would like to thank J.-C. Sikorav, M. Boileau and J. Barge for their comments and the fruitful discussions I had with them.\\
I am also grateful to E. Ghys for his numerous helpful comments, his advice and for his support.\\
Lastly, I would like to thank F. Laudenbach for the interest he had for this work, for his comments, advice and all the precious discussions I had with him.
\end{acknow}
  \newpage

\bibliographystyle{plain}
\bibliography{morsec}

\begin{address}
\noindent
UMPA - UMR 5669 CNRS \\
ENS de Lyon (site Sciences)\\
46, all\'ee d'Italie\\
69364 Lyon Cedex 07 \\
France;
\begin{email}
\href{valentin.seigneur@ens-lyon.fr}{valentin.seigneur@ens-lyon.fr}
\end{email}
\end{address}
\end{document}